\documentclass[12pt]{amsart}
\usepackage{amsmath, amsthm, amscd, amssymb, amsfonts, latexsym, mathtools}
\usepackage{fullpage}
\usepackage{bm,mathdots}
\usepackage{dsfont}
\usepackage{mathrsfs}
\usepackage[all]{xy}
\usepackage{float}
\usepackage{fancybox}
\usepackage{ytableau}
\usepackage{tikz-cd}
\usepackage{algorithm}
\usepackage{algorithmic}
\usepackage{bbm}
\usepackage[colorlinks=true, linkcolor=blue, citecolor=blue, urlcolor=blue]{hyperref}
\theoremstyle{plain} 

\numberwithin{equation}{section}

\newenvironment{red}{\relax\color{red}}{\relax}
\newenvironment{blue}{\relax\color{blue}}{\hspace*{.5ex}\relax}
\newenvironment{jaune}{\relax\color{magenta}}{\relax}
\newcommand{\ber}{\begin{red}}
\newcommand{\er}{\end{red}}
\newcommand{\beb}{\begin{blue}}
\newcommand{\eb}{\end{blue}}
\newcommand{\bej}{\begin{jaune}}
\newcommand{\ej}{\end{jaune}}
\newcommand{\norm}[1]{\left \lVert #1 \right \rVert}

\let\cal\mathcal

\let\bb\mathbb
\let\scr\mathscr

\newtheorem{lem}{Lemma}[section]
\newtheorem{prop}[lem]{Proposition}
\newtheorem{thm}[lem]{Theorem}
\newtheorem*{theorem*}{Theorem}
\newtheorem{defn}[lem]{Definition}
\newtheorem{cor}[lem]{Corollary}

\theoremstyle{definition}

\newtheorem{rem}[lem]{Remark}

\author{Alexey Pozdnyakov}
\title{On (not) learning the Möbius function}
\date{\today}

\begin{document}

\begin{abstract}
    We prove lower bounds on learning the Möbius or Liouville function with a variety of standard learning techniques, including kernel methods, noisy gradient methods, and correlational statistical query algorithms. These results follow from quantitative bounds on the correlation of Möbius with digital characters of various finite abelian groups, where the group is dictated by the type of input data the algorithm is given. Using residues mod $p$ for many different primes corresponds to a cyclic group, and using the base $p$ expansion for a fixed prime corresponds to an elementary abelian $p$-group. We also note that lower bounds of this form are closely related to certain types of digital prime number theorems. 
\end{abstract}

\maketitle

\addtocontents{toc}{\protect\enlargethispage{\baselineskip}}
\tableofcontents

\section{Introduction}
There is a longstanding belief that the Möbius function, defined by
\begin{equation*}
    \mu(n) = \begin{cases}
        1 & \text{ if } n = 1, \\
        (-1)^r & \text{ if } n = p_1\cdots p_r \text{ with } p_i \text{ distinct},\\
        0 & \text{ if } n \text{ has a square factor} ,
    \end{cases}
\end{equation*}
behaves randomly. To establish that the Möbius function mimics random behavior is a deep problem in number theory. For example, the bound
\begin{equation*}
    \left|\sum_{n \leq N} \mu(n) \right| = O\left(N^{1/2+\varepsilon}\right) \quad \text{ for all }\quad\varepsilon > 0,
\end{equation*}
which is what one expects for a sequence of coin flips, is equivalent to the Riemann hypothesis. 

In \cite{IwaniecKowalski2004}, this is stated informally as the Möbius randomness law, which says that Möbius does not correlate with any \textit{reasonable} sequence. This notion was made formal by Sarnak's disjointness conjecture \cite{SarnakMobiusLectures}, which says that Möbius does not correlate with any observable in a zero-entropy dynamical system. Precisely, if $(X,T)$ is a zero-entropy dynamical system, then for every $x \in X$ and every $f \in C(X)$, we have
\begin{equation*}
    \left|\sum_{n \leq N} \mu(n)f(T^nx)\right| =o(N).
\end{equation*}
We remark that even the simplest cases of this conjecture correspond to significant results in number theory. For instance, proving this when $X$ is a point is equivalent to the prime number theorem.  

There are also results on the computational complexity of Möbius. Identify $\Omega =[0, 2^{d}) \cap \mathbb{Z}$ with the binary hypercube $\cal{H}_d = \{-1,1\}^d$ via their binary expansions, and view $\mu\vert_{\Omega}$ as a function on the binary hypercube. Green \cite{green2012notcomputingmobiusfunction} proved that if $F : \cal{H}_d \to \{-1,1\}$ is any function computable by a bounded depth circuit consisting of AND, OR, and NOT gates, then $F$ does not correlate with Möbius. Similarly, Bourgain \cite{bourgain2011fourierwalshspectrummoebiusfunction} proved that if $F$ is a monotone Boolean function, then it does not correlate with Möbius. These results both follow from bounding the Fourier-Walsh spectrum of the Möbius function. Crucially, these results go beyond disjointness for the corresponding dynamical system in the sense that they require an effective rate.

Given that Möbius is expected to behave randomly, it is reasonable to expect that it is difficult for a machine learning algorithm to learn Möbius. To illustrate a concrete realization of our results, we give explicit lower bounds for a prototypical deep learning architecture in the following subsection.

\subsection{Sample results for neural networks}

Fix $r$ distinct primes $p_1,\ldots, p_r$ and integers $d_i \geq 1$ for each $1 \leq i \leq r$ with $\sum_{i=1}^r d_i = d$. Let $\mu_{q}$ denote the $q$-th roots of unity. Consider the Möbius function restricted to $\Omega_d = [0, X_d) \cap \mathbb{Z}$, where $X = \prod_{i=1}^r p_i^{d_i}$, and identify via the Chinese remainder theorem $\Omega_d$ with the space
\begin{equation}\label{eq:general_space}
    \cal{X}_d = \prod_{i=1}^r \mu_{p_i}^{d_i} \subseteq \mathbb{C}^d \cong \mathbb{R}^{2d},
\end{equation}
by encoding the base $p_i$ expansions $x = \sum_{j<d_i} x_{i,j}p_i^{j}$ with the points $e_{p_i}(x_{i,j}) \in \cal{X}_d$. We also fix the uniform probability measure $\mathcal{\mu}_{\mathcal{X}_d}$ on $\mathcal{X}_d$. Now we view $\mu\vert_{\Omega_d}$ as a function on $\cal{X}_d$, which we still denote by $\mu$. Then we train a fully connected neural network to predict this function with noisy gradient descent; see Section~\ref{sec.neural_networks} for precise definitions. 

We show that under practical assumptions about the initialization of the neural network and the hyperparameters of the training algorithm, gradient descent will fail in this setting for $d$ sufficiently large. The two cases we consider are
\begin{equation}\label{eq.solved_Cases}
    \cal{X}_d = \prod_{i=1}^{d} \mu_{p_i} \quad \text{ or } \quad  \cal{X}_d = \mu_{p}^d,
\end{equation}
where the neural network sees $n \bmod p_i$ for many primes, or many base $p$ digits of $n$ for a fixed prime $p$. We are free to select any fully connected architecture, so the depth and width can be arbitrary. We place restrictions on the following hyperparameters related to the descent algorithm:
\begin{enumerate}
    \item the number of iterations of gradient descent $T_d$,
    \item the clipping radius $R_d$,
    \item the noise rate $\tau_d$.
\end{enumerate}
In particular, we assume that the training hyperparameters are bounded by a polynomial in $d$, i.e. there exists a $C>0$ such that $T_d,R_d,\tau_d^{-1} \ll d^C$. For each $d$, you initialize a new neural network $f_{\mathrm{NN},d}(\cdot;\theta_0) : \mathbb{R}^{2d} \to \mathbb{R}$ with Gaussian weights $\theta_0$ and you train it to predict $\mu$. Let $f_{\mathrm{NN},d}(\cdot;\theta_{T_d})$ denote the neural network after running noisy gradient descent for $T_d$ steps. We show that for $d$ sufficiently large, the probability of the trained network $f_{\mathrm{NN},d}(\cdot;\theta_{T_d})$ being a better approximation in $L^2$ than the zero function decays to zero. For learning Möbius from residues mod $p_i$ we obtain the following theorem.
\begin{thm}\label{thm.main_cyclic}
    Let $\cal{X}_d = \prod_{i=1}^d \mu_{p_i}$ and let $f_{\mathrm{NN},d}(\cdot;\theta_{T_d})$ be as above. For any $C, c>0$, where $d^C$ bounds the training parameters, we have that
    \begin{equation}\label{eq.probability_of_learning}
        \mathbb{P}\left[\norm{\mu-f_{\mathrm{NN},d}(\cdot;\theta_{T_d})}_2^2 \leq \norm{\mu}_2^2 - d^{-c}\right] < d^{-A}
    \end{equation}
    for all $A > 0$ and for all $d \gg_{A,C,c} 1$. 
\end{thm}

Consequently, the neural network may be able to learn to detect squarefree integers non-trivially, but not the parity. This is consistent with the experiments of Lowry-Duda \cite{lowryduda2025studyingnumbertheorydeep}, who showed that transformers\footnote{A transformer is a more sophisticated machine learning model than a neural network.} can learn to predict $\mu^2(n)$ from $n \bmod p_i$ with non-trivial accuracy, but struggle to predict $\mu(n)$ better than guessing. It is not surprising that one can observe correlations between square factors and divisibility by small primes, as explained in \cite[Section 3.3]{lowryduda2025studyingnumbertheorydeep}. From the perspective of dynamics, we understand the squarefree factor of the Möbius system well. In particular, this is a deterministic measure-preserving system under a natural ergodic measure, as explained in \cite[Theorem 9]{SarnakMobiusLectures}. It is the parity of Möbius which forms a completely positive entropy extension of the squarefree system, and thus the signs are the interesting values to attempt to learn. For this purpose, it is natural to focus on the Liouville lambda function instead of the Möbius function, see the discussion in Section~\ref{sec:discussion}. For learning Möbius from base $p$ expansions, we obtain the following theorem.

\begin{thm}\label{thm.main_pgroup}
    Let $\cal{X}_d = \mu_p^d$ and let $f_{\mathrm{NN},d}(\cdot;\theta_{T_d})$ be as above. For any $C, c>0$, where $d^C$ bounds the training parameters, there exists an $A > 0$ such that 
    \begin{equation}
        \mathbb{P}\left[\norm{\mu-f_{\mathrm{NN},d}(\cdot;\theta_{T_d})}_2^2 \leq \norm{\mu}_2^2 - d^{-c}\right] < \exp\left(-Ad^{1/10}\right)
    \end{equation}
    for $d \gg_{C,c} 1$. 
\end{thm}

\subsection{Summary of results}
For the two regimes in Equation~\eqref{eq.solved_Cases}, we prove lower bounds on three types of machine learning methods. See Section~\ref{sec.lower_bounds} for a more complete description of these algorithms. In all cases, we will ask for the method to weak-learn the Möbius function in the sense that we want to beat the trivial approximation $\mathbb{E}_{x \sim \mu_\mathcal{X}}[\mu(x)] = o(1)$ in $L^2(\mu_\mathcal{X})$ by atleast a polynomial improvement $(\log X_d)^{-C}$ with atleast polynomial probability. We will see that this always requires superpolynomial parameters. 

First, we consider kernel methods, where one performs linear regression after fixing a potentially non-linear feature map $\Phi: \mathcal{X} \to \mathscr{F}$ to some Hilbert space $\mathscr{F}$. This encompasses many classical learning techniques such as ridge regression or support vector machines, as well as fixed-depth neural networks in a certain infinite-width limit, see \cite{jacot2018ntk} for details on this limit. 

Second, we consider gradient methods, where one fits a parametric model $f(\cdot ; \theta) : \mathcal{X} \to \mathbb{R}$ by performing noisy gradient descent. This models a vast range of deep learning techniques, although noisy gradient descent is a proxy for the more sophisticated stochastic algorithms used in practice.

Third, we consider correlational statistical query algorithms, which is an abstract class of algorithms with access to an oracle $\mathrm{CSQ}(\tau)$ which takes as input a query $\phi : \mathcal{X} \to [-1,1]$ and returns the correlation $\mathbb{E}_{x \sim \mu_\mathcal{X}}[\mu(x)\phi(x)]$ up to some absolute error $\tau > 0$. 

In all cases, we begin with a lower bound for the learning complexity of such algorithms against an arbitrary hypothesis class $\mathcal{H} \subseteq L^2(\mu_\mathcal{X})$. We then take a compact group $G$ with Haar measure $\mu_G$, which will be abelian in our analysis, and we introduce a $G$-equivariance assumption (see Definition \ref{def:G_equivariant}). This assumption is what will allow us to produce results for a fixed target function $h \in L^2(\mu_\mathcal{X})$. Precisely, we get lower bounds on learning in terms of the alignment, defined by
\begin{equation}\label{eq:def_alignment}
    \mathrm{A}(f,G) := \sup_{\norm{\phi}_2 = 1} \mathbb{E}_{g \sim \mu_G}[|\langle f \circ g, \phi \rangle|^2].
\end{equation}
Taking $G = \mathcal{X}_d$ acting on itself, we deduce the following three theorems. Note that these bounds can be improved under GRH.  

\begin{thm}\label{thm:main_kernel}
    Let $G = \mathcal{X}_d$ and suppose that $\mathcal{A}$ is the kernel method with RKHS $\mathscr{H}$ and $G$-invariant kernel $K$. For any $C>0$, if we train $\mathcal{A}$ on the sample $S = \{(x_i, \mu(x_i)\}_{i=1}^n$ and obtain accuracy
    \begin{equation*}
        \norm{\mathcal{A}(S) - \mu}_2^2 < 1-d^{-C},
    \end{equation*}
    then for some $B  \gg_{C,p} 1$, all $A > 0$, and all $d \gg_{A,C} 1$,
    \begin{equation*}
        \min(n, \eta) \geq \begin{cases}
            (\log X_d)^A & \text{ if } \mathcal{X}_d = \prod_{i=1}^d \mu_{p_i} \\
            \exp(B (\log X_d)^{1/10}) & \text{ if } \mathcal{X}_d = \mu_p^{d} 
        \end{cases}
    \end{equation*}
    with $\eta = \dim \scr{H}$. In particular, we need a superpolynomial number of samples to weak-learn.
\end{thm}

\begin{thm}\label{thm:main_NGD}
    Let $G = \mathcal{X}_d$ acting on itself and for any $C>0$, suppose that $f(\cdot; \theta_{T_d})$ is a $G$-equivariant parametric model trained by noisy gradient descent for $1 \leq T_d \ll (\log X_d)^{C}$ steps, and gradient precision $0<R_d/\tau_d \ll (\log X_d)^{C}$. Then for some $B  \gg_{C,p} 1$, all $A > 0$, and all $d \gg_{A,C} 1$,
    \begin{equation*}
        \mathbb{P}_{\theta_{T_d}}[\norm{f(\cdot;\theta_{T_d}) - \mu}_2^2 \leq \norm{\mu}^2_2 - d^{-C}] \leq \begin{cases}
            (\log X_d)^{-A} & \text{ if } \mathcal{X}_d = \prod_{i=1}^d \mu_{p_i} \\
            \exp(-B (\log X_d)^{1/10}) & \text{ if } \mathcal{X}_d = \mu_p^{d} 
        \end{cases}
    \end{equation*}
    In particular, we need a superpolynomial time $T_d$ or gradient precision $R_d/\tau_d$ to weak-learn.
\end{thm}

\begin{thm}\label{thm:main_CSQ}
    Let $G = \mathcal{X}_d$ acting on itself and let $0 < C < C'$. Suppose that $\mathcal{A}$ is a $G$-equivariant CSQ algorithm making $1 \leq q_d \ll (\log X_d)^{C}$ oracle calls with precision $(\log X_d)^{-C'} \ll \tau_d < \frac18(\log X_d)^{-C}$. Then for some $B  \gg_{C,C',p} 1$, all $A > 0$, and all $d \gg_{A,C,C'} 1$,
    \begin{equation*}
        \inf_{\mathcal{O}_\mu}\mathbb{P}_{R}[\norm{\mathcal{A}^{\mathcal{O}_\mu, R} - \mu}_2^2 \leq \norm{\mu}^2_2 - d^{-C}] \leq \begin{cases}
            (\log X_d)^{-A} & \text{ if } \mathcal{X}_d = \prod_{i=1}^d \mu_{p_i} \\
            \exp(-B (\log X_d)^{1/10}) & \text{ if } \mathcal{X}_d = \mu_p^{d_i} 
        \end{cases}
    \end{equation*}
    with infimum over all $\tau_d$-compatible oracle responses and $R$ denoting randomness in the algorithm. In particular, we need superpolynomial queries $q_d$ or precision $\tau_d^{-1}$ to weak-learn against adversarial noise. 
\end{thm}

We remark that the $G$-equivariance assumption is quite natural since $\mathcal{X}_d$ acts orthogonally on itself (see Lemma~\ref{lem:orthogonal_action}). In particular, any algorithm equivariant under an orthogonal transformation of the underlying dataset will be $G$-equivariant, and in fact, we only require equivariance with respect to a finite subgroup of the torus inside of $\mathrm{SO}_{2d}(\mathbb{R})$.

Theorem~\ref{thm.main_cyclic} and Theorem~\ref{thm.main_pgroup} are corollaries of Theorem~\ref{thm:main_NGD}. The proof of this theorem is based on the work of \cite{abbe2022nonuniversalitydeeplearningquantifying}, together with a few simplifications that can found in \cite{misiakiewicz2025shorttutorial}; see Theorem~\ref{thm:NGD_equivariant} for the learning lower bound. To deduce the first two theorems, one can use the general $G$-equivariance criteria established in Proposition~\ref{prop:G_inv_NGD_condition}, or one can use \cite[Lemma E.3]{abbe2022nonuniversalitydeeplearningquantifying} together with Lemma~\ref{lem:orthogonal_action}. Theorem~\ref{thm:main_kernel} is based on \cite[Proposition 11]{abbe2024mergedstaircasepropertynecessarynearly}; see Section \ref{ssec.kernel_methods} for the learning lower bound. The lower bound in Theorem~\ref{thm:main_CSQ} is based on \cite{misiakiewicz2025shorttutorial}, but to the best of our knowledge, has not been worked out carefully in the literature. We give an essentially complete derivation of this bound in Section~\ref{ssec:CSQ}. 

In practice our results say the following. If there exists an algorithm that can successful weak learn the Möbius function, it must break one of our assumptions. In particular, one must either:
\begin{enumerate}
    \item Use a different encoding: e.g. represent integers with something other than residues mod $p$ or a fixed base $p$-expansion given by cyclic embedding.
    \item Break the $G$-equivariance: i.e. the method must be able to detect rotations of the cyclic embedding.  
    \item Use a different method: the learning algorithm must not fall under one of the above methods. 
\end{enumerate}
It is not hard to break all of these assumptions. For example, one can consider a transformer with a base $p$ tokenization, learned embedding, and trained with stochastic gradient descent. Nevertheless, we expect that one can not learn the Möbius function in any meaningful sense. 

Finally, we note that it is relatively straightforward to prove exponential lower bounds on learning the hypothesis class of completely multiplicative binary-valued functions. Let
\begin{equation*}
    \mathcal{M}_d = \{f : [1, X_d] \cap \mathbb{Z} \to \{-1,1\} : f(mn)=f(m)f(n) \text{ for all } n,m \text{ with } mn \leq X_d\},
\end{equation*}
Notice that we have for the trivial approximation $\mathbb{E}_{h \sim \mu_\mathcal{H}}[h(n)] = \mathbbm{1}_{\square}(n)$, the indicator function of squares. If any of the above methods perform well on average over $\mathcal{M}_d$, it must have exponential complexity. Note that in this setting, we can take $\mathcal{X}_d$ of the general form in Equation~\ref{eq:general_space}, and we can drop the equivariance assumption. In fact, the proof relies only on the size of $X_d$. We give one such result here.
\begin{thm}\label{thm:main_binary_mul}
    Fix $C>0$ and $0 < \delta < 1/4$, and suppose that $f(\cdot; \theta_{T_d})$ is a parametric model trained by noisy gradient descent with hyperparameters satisfying
    \begin{equation*}
         T_d \left(\frac{R_d}{\tau_d}\right)^2 \ll X_d^{1/2-2\delta}.
    \end{equation*}
    Then for $d \gg 1$ we have
    \begin{equation*}
        \mathbb{P}_{h \sim \mathcal{M}_d,\theta_T}[\norm{f(\cdot;\theta_T) - h}_2^2 \leq \norm{h- \mathbbm{1}_{\square}(n)}^2_2 - (\log X_d)^{-C}] \ll X_d^{-\delta}.
   \end{equation*}
\end{thm}

\subsection{Number-theoretic input}
To prove the theorems above, we must bound the alignment in Equation~\eqref{eq:def_alignment}. Realizing this quantity as the norm of a certain covariance operator, we obtain a spectral expression
\begin{equation*}
    \mathrm{A}(f,G) = \max_{a \in \widehat{\mathcal{X}_d}} |\widehat{f}(a)|^2,
\end{equation*}
when $G = \mathcal{X}_d$. In particular, to prove superpolynomial bounds on learning, we must show
\begin{equation*}
    \max_{a \in \widehat{\mathcal{X}_d}}\left|\sum_{n < X_d} \mu(n)\chi_a(n)\right| \ll_A \frac{X_d}{(\log X_d)^{A}}
\end{equation*}
for any $A > 0$. Here we evaluate $\chi_a(n)$ by identifying $\Omega_d =[0, X_d) \cap \mathbb{Z}$ with $\mathcal{X}_d$, so the characters $\chi_a(n)$ are certain digital function on $\Omega_d$. We call these $\chi_a$ digital characters, see also Walsh--Vilenkin characters or the generalized Rademacher system. 

In the case where we take $\mathcal{X}_d = \prod_{i=1}^d \mu_{p_i}$, the $\mathcal{X}_d$ is a cyclic group and the character $\chi_a(n)$ can be identified with an additive character $e(n\theta)$. In this case, our bound follows from the following classical theorem of Davenport.
\begin{thm}\cite[Theorem 1]{10.1093/qmath/os-8.1.313}\label{thm.Davenport} For any $A > 0$, we have
    \begin{equation*}
        \sup_{\theta \in \mathbb{R}/\mathbb{Z}} \left|\sum_{x < X} \mu(x)e(x\theta)\right| \ll_A X (\log X)^{-A}.
    \end{equation*}
\end{thm}
The proof of this theorem uses Vinogradov's bilinear method to get a power saving in the minor arcs. The barrier to improving this result lies in the major arcs, where one expands $e(x\theta)$ using Dirichlet characters and applies the classical zero-free region for their $L$-functions. To improve this result would require proving that there are no Landau-Siegel zeros.

For binary expansions, we use the results of Green \cite{green2012notcomputingmobiusfunction} and Bourgain \cite{Bourgain2013}. For base $p$ expansions with $p$ an odd prime, we generalize the work of Green and Bourgain to obtain the following theorem. 
\begin{thm}\label{thm:main_math}
    Fix a prime $p$ and let $\mathcal{X}_d = \mu_p^d$. For all $d \gg 1$, we have
    \begin{equation*}
        \max_{\chi_a \in \widehat{\mu_p^{d}}} \left| \sum_{x < p^d} \mu(x)\chi_a(x) \right| < p^{d-d^{1/10}}.
    \end{equation*}
\end{thm}

We have not made any effort to optimize the exponent $1/10$. The proof of Theorem~\ref{thm:main_math} splits into cases based on the Hamming weight of $\chi_a$; see Section~\ref{sec.characters}. For small weights, we relate the coefficients coming from expanding in $\chi_a$ to traditional Fourier coefficients at sparse $p$-adic rationals. We then use classical bounds on the minor arcs, and we take advantage of the sparse $p$-adic structure to prove unconditional estimates on the major arcs. This is the same strategy as \cite{green2012notcomputingmobiusfunction}, and is based on an argument of Katai \cite{Katai1986}; see Section~\ref{sec.Katai} and Section~\ref{sec.small_weight}. For large weights, we employ Vinogradov's method of reducing to Type I and Type II sums, taking advantage of the carry structure to help localize our analysis. This follows the same strategy as \cite{Bourgain2013}, which is based on the work of Mauduit and Rivat \cite{MauduitRivat2010}; see the rest of Section~\ref{sec.main_proof}. The majority of the effort goes into generalizing results about the Fourier-Walsh functions, which are characters on $(\mathbb{Z}/2\mathbb{Z})^d$, to characters on general $\cal{X}_d$; see Section~\ref{sec.character_Fourier}.

We emphasize that the bounds we give go beyond the disjointness conjecture. The reason for this is that to give quantitative results on the complexity of these algorithms requires us to prove an effective rate, superpolynomial if we wish to rule out polynomial complexity. The disjointness in this setting is known, and in fact is established for all automatic sequences by Müllner \cite{M_llner_2017}. We also note that establishing the Type I/Type II bounds necessary to prove Theorem~\ref{thm:main_math} immediately establishes bounds of the form
\begin{equation*}
    \max_{\substack{a \in \widehat{\mathcal{X}_d} \\ |a| \gg d^{4/10}}} \left|\sum_{n \leq X_d} \Lambda(n)\chi_a(n)\right| < p^{d-d^{1/10}},
\end{equation*}
 where $\Lambda(n)$ is the Von-Mangoldt function. Consequently, our results also establish an effective digital prime number theorem of the following form.

\begin{thm}\label{thm:digital_PNT}
    Fix a prime $p$ and let $L : \mathbb{F}_p^d \to \mathbb{F}_p^m$ be a surjective linear map and fix $b \in \mathbb{F}_p^m$. For all $d \gg 1$, we have
    \begin{equation*}
        \sum_{\substack{n < p^d \\ L(n_0,\ldots, n_{d-1})=b}} \Lambda(n) = \mathfrak{S}_p(L,b) \frac{p^d}{p^m} + O\left(p^{d-d^{1/10}}\right),
    \end{equation*}
    where $n = \sum_{j=0}^{d-1} n_jp^j$, and if we let $e_0 \in (\mathbb{F}_p^d)^\vee$ be the functional $e_0(x) = x_0$, we have
    \begin{equation*}
         \mathfrak{S}_p(L,b) = \begin{cases}
            1 & \text{ if } e_0 \notin \mathrm{im } L^\vee, \\
            \frac{p}{p-1}& \text{ if } e_0 = L^\vee \lambda \text{ and } \lambda(b) \neq 0, \\
            0 & \text{ if } e_0 = L^\vee \lambda \text{ and } \lambda(b) = 0.
        \end{cases}
    \end{equation*}
\end{thm}

It follows that for any surjective linear map $L : \mathbb{F}_p^d \to \mathbb{F}_p^m$ and any $b \in \mathbb{F}_p^m$ we can count primes of the form
\begin{equation*}
    \pi_{L,b}(X_d) = \#\left\{\ell <X_d: \ell \text{ prime}, \quad \sum_{j=0}^{d-1} \ell_j p^j, \quad L(\ell_0,\ldots,\ell_{d-1}) = b   \right\},
\end{equation*}
provided that $m \leq B \log d$ for some constant $B$. Precisely, we have the following.
\begin{cor}\label{cor:digital_PNT}
    Fix a prime $p$ and let $X_d = p^d$. Let $L : \mathbb{F}_p^d \to \mathbb{F}_p^m$ be a surjective linear map where $m \leq B \log d$, and fix $b \in \mathbb{F}_p^m$ with $\mathfrak{S}_p(L,b) > 0$. For all $d \gg 1$,
    \begin{equation*}
        \pi_{L,b}(X_d) = \frac{\mathfrak{S}_p(L,b)}{p^m} \frac{X_d}{\log X_d}\left(1 + O\left(\frac{\log \log X_d}{\log X_d}\right)\right).
    \end{equation*}
\end{cor}

We conclude this section with the remark that to prove superpolynomial lower bounds on learning Möbius as a function on the general space $\mathcal{X}_d = \prod_{i=1}^r \mu_{p_i}^{d_i}$ would also establish such prime number theorems for counting primes with linear digital conditions in $r$ simultaneous bases. Currently, the strongest digital results in multiple bases are due to Dromota, Mauduit, and Rivat \cite{drmota2020prime}, and they are limited to handling two simultaneous bases. Once again, the difficulty lies in establishing unconditional results with sufficiently strong rates.

\subsection{Discussion}\label{sec:discussion}

First we note that in the above results, the Möbius function is interchangeable with the Liouville function
\begin{equation*}
    \lambda(n) = (-1)^{\Omega(n)},
\end{equation*}
where $\Omega(n)$ is the number of prime factors with multiplicity. Namely, the methods used to prove Theorem \ref{thm.Davenport} and Theorem~\ref{thm:main_math} can be applied to $\lambda$ as well.

It is also natural to ask about the more general problem of learning the function
\begin{equation*}
    \lambda_{S}(n) = (-1)^{\Omega_S(n)}
\end{equation*}
where $S$ is a subset of the set of primes $\mathcal{P}$ and $\Omega_S(n)$ counts prime factors in $S$ with multiplicity. While we have shown that the case where $S = \mathcal{P}$ is hard to learn, the case $S = \{2\}$ can easily be weak-learned from binary expansions. We point out that our techniques fail to produce strong lower bounds once the sifting density of $S$ drops below $1/2$. Namely, assume that 
\begin{equation*}
    \sum_{\substack{p \in S \\ p < X}} \frac{\log p}{p} = \frac{1-\kappa}{2} \log X + O(1),
\end{equation*}
for some $\kappa \in (0,1]$. Then by \cite[Theorem 5]{BorweinChoiCoons2010}, we have
\begin{equation*}
    \widehat{\lambda}_S(0) = \frac{1}{X}\sum_{n \leq X} \lambda(n) = (1+o(1))c_\kappa (\log X)^{\kappa -1},
\end{equation*}
and so 
\begin{equation*}
    \max_{a \in \widehat{\mathcal{X}}} |\widehat{\lambda}_{S}(a)|^2 \geq |\widehat{\lambda}_{S}(0)|^2 \gg (\log X_d)^{2\kappa -2} \gg (\log X_d)^{-2},
\end{equation*}
so at best we can prove quadratic lower bounds. In fact, \cite[Theorem 3.4]{abbe2022nonuniversalitydeeplearningquantifying} implies that such functions on the binary hypercube $\lambda_S \in L^2(\mu_2^d)$ can be weak learned by randomly initialized fully connected neural networks. However, the proof goes by giving a particular range of parameters and random initialization for which the function is weak learned in one iteration with high probability. In practice, we have found that models such as transformers can achieve slightly non-trivial accuracy, but still struggle to learn such functions. 

\section*{Acknowledgments}
The author is very grateful to Peter Sarnak for suggesting the problem of studying the complexity of Möbius through the lens of machine learning theory and for countless insightful discussions. They also acknowledge Francois Charton, Jordan Ellenberg, Boris Hanin, and Alex Negron for useful discussions and feedback on earlier drafts. 

\section{Background}

\subsection{Characters and the Möbius 
function}\label{sec.characters}
Here we introduce the space $\cal{X}_d$ and give an explicit description of its dual group $\widehat{\cal{X}}_d$. We then review relevant results on the Möbius function. Fix $r$ distinct primes $p_1<p_2<\cdots<p_r$ and some $d_i \geq 1$ for each $1 \leq i \leq r$. We consider the space
\begin{equation}\label{eq.def_calX}
    \cal{X}_d = \prod_{i=1}^r \mu_{p_i}^{d_i} \subseteq \mathbb{C}^{d} \cong \mathbb{R}^{2d},
\end{equation}
where $d = \sum_i d_i$. We put the uniform probability measure $\mu_\cal{X}$ on $\cal{X}$. We have an abelian group structure with a dual group
\begin{equation*}
    \widehat{\cal{X}}_d = \left\{ \chi_a : a \in \prod_{i=1}^r(\bb{Z}/p_i\bb{Z})^{d_i}\right\} \cong \prod_{i=1}^r(\bb{Z}/p_i\bb{Z})^{d_i},
\end{equation*}
where, if $x = (e_{p_i}(x_{i,j}))_{1 \leq i \leq r, 0 \leq j < d_i}$, the character $\chi_a$ is defined by
\begin{equation*}
    \chi_a(x) = \prod_{i=1}^r \prod_{j=0}^{d_i-1} (e_{p_i}(x_{i,j}))^{a_{i,j}}.
\end{equation*}
We will say that a character $\chi_{a}$ has weight $|a|$ where
\begin{equation*}
    |a| = \#\{(i,j) : a_{i,j} \neq 0\},
\end{equation*}
is the Hamming weight of the vector $a$.

We let $X_d = |\cal{X}_d|$ and we identify $\cal{X}_d$ with $\Omega_d = [0,X_d)\cap \mathbb{Z}$ via base $p_i$ expansions. For each integer $x < X_d$ and each prime $p_i$ write
\begin{equation*}
    x = \sum_{j=0}^{d_i-1} x_{i,j} p_i^{j},
\end{equation*}
where $0 \leq x_{i,j} <p_i$, and identify $x$ with $(e_{p_i}(x_{i,j})) \in \cal{X}_d$. This is a bijection by the Chinese remainder theorem, and the uniform measure on $\Omega_d$ agrees with the uniform measure on $\cal{X}_d$. This allows us to view arithmetic functions as functions on $\cal{X}_d$ via their restriction to $\Omega_d$.

The characters $\chi_a$ form an orthonormal basis for $L^2(\cal{X}_d)$. In particular, we have orthogonality relations
\begin{equation*}
    \frac{1}{|\cal{X}|} \sum_{x \in \cal{X}} \chi_a(x)\overline{\chi_b(x)} = \delta_{ab},
\end{equation*}
and a Fourier expansion on $\cal{X}_d$ given by
\begin{equation*}
    f(x) = \sum_{a \in \widehat{\cal{X}}} \widehat{f}(a)\chi_a(x), \qquad \widehat{f}(a)= \frac{1}{|\cal{X}|} \sum_{x \in \cal{X}} f(x)\overline{\chi_a(x)}.
\end{equation*}

Our goal is to show that the Möbius function does not correlate with the $\chi_a$ defined above. By viewing $\chi_a$ as a function on $\mathbb{Z}/X\mathbb{Z}$ and taking their Fourier expansions on $\mathbb{Z}/X\mathbb{Z}$, we can reduce this problem to bounding the Fourier coefficients of $\chi_a$ and the correlation of Möbius with additive characters. The strongest unconditional result in this direction is the Theorem~\ref{thm.Davenport} of Davenport from 1937. If one assumes GRH for Dirichlet $L$-functions, we have the following bound due to Baker and Harman:
\begin{thm}\cite[Theorem]{BakerHarman1991}\label{thm.Baker_Harman}
    Assume GRH for Dirichlet $L$-functions. Then for all $\varepsilon > 0$ we have
    \begin{equation*}
        \sup_{\theta \in \mathbb{R}/\mathbb{Z}} \left|\sum_{x < X} \mu(x)e(x\theta)\right| \ll X^{3/4+\varepsilon}.
    \end{equation*}    
\end{thm}
Under this assumption, we can handle weights as large as $|a| \ll d/\log d$ in the $\mathcal{X}_d = \mu_p^d$ case using the theory of Dirichlet $L$-functions. This can be used to get conditional bounds of the form
\begin{equation*}
    \max_{a \in \widehat{\mathcal{X}_d}}\left| \sum_{n < p^d} \mu(n)\chi_a(n)\right| \ll p^{d\left(1-c/(\log d)^2\right)},
\end{equation*}
see~\cite[Theorem 2]{Bourgain2013}.

\subsection{Statistical learning theory}
We give a brief introduction to statistical learning theory in order to orient an unfamiliar reader. In the following section, we describe learning with deep neural networks, which is a concrete and standard technique in practice. The general set up is as follows. You are given a data space $\mathcal{X}$ with probability measure $\mu_\mathcal{X}$, as well as a hypothesis class $\mathcal{H} \subseteq L^2(\mathcal{X})$ with probability measure $\mu_\mathcal{H}$. A secret target function $h \sim \mu_\mathcal{H}$ is then drawn at random, and you are given access to certain resources. With the limited resources, your goal is to produce a model $\widehat{h} \in L^2(\mathcal{X})$ which approximates $h$ as well as possible with probability as high as possible. One goal of the field is to understand upper and lower bounds on the resources needed to succeed.

The resources you have access to depend on the model of learning. Here are a few standard examples:
\begin{enumerate}
    \item In the probably almost correct (PAC) \cite{Valiant1984}, or Valiant model, you are given access to $n \in \mathbb{Z}_{ \geq 1}$ i.i.d. random samples $(x,h(x))$ drawn from $x \sim \mu_\mathcal{X}$. One can also consider a noisy variant in which you receive $(x, h(x) +\xi)$ where $\xi \sim N(0, \tau^2)$ is an i.i.d. Gaussian noise.
    \item In the statistical query model (SQ) \cite{Kearns1998}, you can request statistical information of the form $\mathbb{E}_{x \sim \mu_\mathcal{X}}[(\phi(x,h(x))]$ for any measurable $\phi : \mathcal{X} \times \mathcal{Y} \to [-1,1]$, and you receive a value $v \in \mathbb{R}$ satisfying
    \begin{equation*}
        |\mathbb{E}_{x \sim \mu_\mathcal{X}}[(\phi(x,h(x))] - v| \leq \tau.
    \end{equation*}
    It is strictly harder to learn in the SQ model then the PAC model.
\end{enumerate}
We will limit resources by considering specific classes of algorithms, e.g. kernel methods or gradient methods with noise, and prove lower bounds on the number of samples or the number of gradient descent steps needed to succeed. 

It is natural to consider the following notions of success. We say that you strongly learn if you perform arbitrarily well, i.e. for any $\varepsilon > 0$ and $\delta >0$ you achieve 
\begin{equation*}
    \mathbb{P}_{h \sim \mu_\mathcal{H}}\left[ \norm{h-\widehat{h}} < \varepsilon\right] > 1-\delta,
\end{equation*}
for some choice of norm. We will always use $L^2$ norm in this paper. On the other hand, we say that you weak learn if you improve on the trivial estimate $\overline{h} := \mathbb{E}_{h \sim \mu_\mathcal{H}}[h]$ in a nontrivial way. In this paper, this will mean a polynomial improvement, i.e. for some $C> 0$ you achieve
\begin{equation*}
    \mathbb{P}_{h \sim \mu_\mathcal{H}}\left[\norm{h-\widehat{h}} < \norm{h - \overline{h}} - d^{-C}\right] > d^{-C}
\end{equation*}
where $d$ is a complexity parameter, e.g. the number of digits. Since we are interested in lower bounds on functions that are believed to be ``high-complexity", we will focus on lower bounds for weak-learning.

While statistical learning theory generally considers a hypothesis class, we will prove bounds on learning a fixed function $f \in L^2(\mathcal{X})$ by taking advantage of a symmetry. In particular, if $G$ acts on $\mathcal{X}$ with $\mu_\mathcal{X}$ a $G$-invariant measure, and the method of learning is $G$-equivariant, then to succeed in learning a function $f$ implies you will succeed in learning the hypothesis class $\mathcal{H}_f = \{f \circ g : g \in G\}$. See Section~\ref{sec.lower_bounds} for more details on this reduction. To the best of our knowledge, this idea originates in the work of \cite{abbe2022nonuniversalitydeeplearningquantifying}.

\subsection{Learning with neural networks}\label{sec.neural_networks}
We begin by defining the neural networks and training algorithm for which our theorem applies. We end this section with some discussion about the theory of neural networks.
Let $L \in \mathbb{N}$ and $N_0, \ldots, N_L \in \mathbb{N}$. Then a (fully-connected, or feed-forward) neural network $f_{\mathrm{NN}}$ with activation function $\rho : \mathbb{R} \to \mathbb{R}$ is a map $f_{\mathrm{NN}}:\mathbb{R}^{N_0} \to \mathbb{R}^{N_L}$ defined by 
\begin{equation*}\label{d-NN}
    f_{\mathrm{NN}}(x;\theta) = \begin{cases}
    W_1(x),  & L = 1, \\
    W_2 \circ \rho \circ W_1(x), & L =2, \\
    W_L \circ \rho \circ W_{L-1} \circ \rho \circ \ldots \circ \rho \circ W_1(x), & L \geq 3,
\end{cases}
\end{equation*}
where
\begin{equation*}
    W_\ell(x) = A_\ell x + b_\ell \text{ with } A_\ell \in \mathbb{R}^{N_\ell \times N_{\ell -1}} \text{  and  } b_\ell \in \mathbb{R}^{N_\ell} \text{ for all }  \ell \in \{1,\ldots, L \},
\end{equation*}
and $\rho$ acts component-wise on vectors. The vector of parameters $\theta \in \mathbb{R}^p$ consists of all components in the $A_\ell$ and $b_\ell$. We will only consider the case $N_L = 1$ in this paper. 

Let $\mathcal{X} \subseteq \mathbb{R}^d$ be a probability space, and let $\cal{D}$ be a probability distribution on $\cal{X} \times \mathbb{R}$. We train a neural network $ f_{\mathrm{NN}}(\cdot; \theta) : \cal{X} \to \mathbb{R}$ by minimizing a loss function $\ell_\cal{D}(\theta)$. In this paper, we fix the loss function to be the mean square error:
\begin{equation*}
    \ell_\cal{D}(\theta) = \mathbb{E}_{\cal{D}}[(y-f(x;\theta))^2].
\end{equation*}
To minimize the loss, we use the noisy gradient descent algorithm (GD). The algorithm will randomly initialize $\theta_0$ according to some probability distribution $\mu_0$ on $\mathbb{R}^p$. It then iteratively updates the parameters with a step sizes $\eta_k > 0$ in a direction $G_\cal{D}(\theta_k)$ that approximates the gradient of the loss function, plus Gaussian noise $\xi^k \sim \cal{N}(0,\tau^2I)$:
\begin{equation*}
    \theta_{k+1} = \theta_k - \eta_k \left(G_{\cal{D}}(\theta_k) + \xi^k\right).
\end{equation*}
Here we have 
\begin{equation*}
    G_\mathcal{D}(\theta) := -\mathbb{E}_{\mathcal{D}}[(y-f(x;\theta)\Pi_{R}\nabla_\theta f(x;\theta)],
\end{equation*}
where $\Pi_{R}$ denote projection onto a ball of radius $R > 0$. This is just the direction of $\nabla \ell_{\cal{D}}(\theta)$, except we clip the gradient if it becomes too large. This is done in practice to avoid instability from exploding gradients \cite{zhang2020gradientclippingacceleratestraining}, and in theory to control the precision of the gradient. Notice that we have randomness coming from both initialization and noise. 

Neural networks are highly nonlinear functions that are well-suited for approximating a large class of functions. They are universal approximators in the sense that a sufficiently large network can approximate any continuous function on a compact domain arbitrarily well \cite{Cybenko1989}, as long as the activation function is not a polynomial. It is also known that any Boolean function in $n$ variables that can be implemented in time $T(n)$ can be computed by a neural network of size $O(T(n)^2)$ \cite{ShalevShwartzBenDavid2014}. In this sense, neural networks have very low approximation error. 

Another source of error in machine learning arises from sample complexity, which controls how many examples a learning algorithm needs before it can successfully approximate a target function. For neural networks, it is known that the sample complexity scales as a polynomial in the size of the network \cite{GolowichRakhlinShamir2020, BartlettFosterTelgarsky2017}.

The main source of error with neural networks is optimization error, that is, failing to minimize the loss function. As empirical risk minimization is known to be NP-hard, it is difficult to give training algorithms that guarantee low optimization error. This will be the source of our impossibility theorem, and it will be specific to our training algorithm. We note that our gradient descent algorithm falls under the statistical query \cite{Kearns1998} paradigm in statistical learning theory. On the other hand, stochastic gradient descent with batch size one falls under the Valiant paradigm \cite{Valiant1984}. It is known that learning in the statistical query paradigm is strictly harder than the Valiant paradigm. This suggests that the limitations of our learning algorithm may be circumvented by choosing a different learning algorithm. See \cite{Abbe2018ProvableLO, abbe2022nonuniversalitydeeplearningquantifying} for further discussion. 

\section{Lower bounds on equivariant learning algorithms}\label{sec.lower_bounds}

In this section, we present general lower bounds for learning a fixed function under three types of $G$-equivariant learning algorithms. Following the ideas in \cite{abbe2022nonuniversalitydeeplearningquantifying}, the key to proving lower bounds for a fixed function $h \in L^2(\mu_\mathcal{X})$ will be to use a symmetry group $G$ of the given learning algorithm to produce a hypothesis class $\mathcal{H} = \{h \circ g: g \in G\}$ on which the algorithm must have constant performance. One can then deduce lower bounds on learning from the complexity of the resulting hypothesis class. We begin by defining $G$-equivariant algorithms. 
\begin{defn}\label{def:G_equivariant}
    A randomized algorithm $\mathcal{A}$ that takes in a data distribution $\mathcal{D} \in \mathcal{P}(\cal{X} \times \cal{Y})$ and outputs a function $\mathcal{A}(\mathcal{D}) : \mathcal{X} \to \mathcal{Y}$ is said to be $G$-equivariant if for every $g \in G$ we have
    \begin{equation*}
        \mathcal{A}(\mathcal{D}) \overset{d}{=} \mathcal{A}(g \circ \mathcal{D}) \circ g
    \end{equation*}
    where $g \circ \mathcal{D}$ is the distribution $(g(x), y)$ for $(x,y) \sim \mathcal{D}$. Here $\overset{d}{=}$ denotes equality in distribution.
\end{defn}

We now discuss three types of learning algorithms for which we can prove lower bounds. For each regime, we give examples of $G$-equivariant algorithms and prove a lower bound in terms of $h$ and $G$. We will see a single quantity, which we call alignment, that bounds all three methods:
\begin{equation*}
    \mathrm{A}(h,G) := \sup_{\norm{\phi}_2=1} \mathbb{E}_{g \sim \mu_G}\left[|\langle U_gh, \phi\rangle|^2\right].
\end{equation*}
In certain cases, we will allow an arbitrary offset $h_* \in L^2(\mathcal{X})$,
\begin{equation*}
    \mathrm{A}(h,G; h_*) := \sup_{\norm{\phi}_2=1} \mathbb{E}_{g \sim \mu_G}\left[|\langle U_gh-h_*, \phi\rangle|^2\right].
\end{equation*}
More generally, given a hypothesis class $\mathcal{H} \subseteq L^2(\mu_\mathcal{X})$ with probability measure $\mu_\mathcal{H}$, we denote
\begin{equation*}
     \mathrm{A}(\mu_\mathcal{H}; h_*) := \sup_{\norm{\phi}_2=1} \mathbb{E}_{h \sim \mu_\mathcal{H}}\left[|\langle h-h_*, \phi\rangle|^2\right].
\end{equation*}

\subsection{Kernel methods}\label{ssec.kernel_methods}

We begin with a classical learning technique known as the kernel method. We take a general input space $\mathcal{X}$ with distribution $\mu_\mathcal{X}$ together the following data
\begin{enumerate}
    \item A feature space $\mathscr{F}$ which is a Hilbert space with inner product $\langle \cdot ,\cdot \rangle_{\mathscr{F}}$.
    \item A feature map $\Phi : \mathcal{X} \to \mathscr{F}$.
    \item A function space $\mathscr{H} = \{f(x;\theta) = \langle \theta , \Phi(x)\rangle : \theta \in \mathscr{F}\}$.
\end{enumerate}
The model class $ \mathscr{H}$ inherits the inner product from $\mathscr{H}$ via $\langle f(\cdot ; \theta_1), f(\cdot ; \theta_2)\rangle_{\mathscr{H}} = \langle \theta_1 ,\theta_2\rangle_{\mathscr{F}}$. We assume that $\theta \mapsto f(\cdot; \theta)$ is injective, passing to a quotient of $\mathscr{H}$ if necessary. The function space $\mathscr{H}$ also inherits the additional structure of a reproducing kernel Hilbert space (RKHS) with symmetric positive definite kernel $K : \mathcal{X} \times \mathcal{X} \to \mathbb{R}$ given by $K(x_1,x_2) = \langle \Phi(x_1), \Phi(x_2) \rangle_{\mathscr{F}}$. Note that feature maps and kernels give equivalent descriptions of the RHKS, and it is often most convenient to work only with the kernel $K$.

The kernel method is simply to apply empirical risk minimization (ERM) on the RHKS $\mathscr{H}$. That is, given data $(x_i,y_i) \in \mathcal{X} \times \mathbb{R}$ with $i \in [n]$, the learned predictor is defined by
\begin{equation*}
    \widehat{f} = \operatornamewithlimits{argmin}_{f \in \mathscr{H}}\left\{ \frac{1}{n} \sum_{i=1}^n \ell(f(x_i),y_i)) + \lambda \norm{f}^2_\mathscr{H} \right\}.
\end{equation*}
By the representer theorem, this reduces to a finite-dimensional optimization problem 
\begin{equation*}
    \widehat{a} =  \operatornamewithlimits{argmin}_{a \in \mathbb{R}^n}\left\{ \frac{1}{n} \sum_{i=1}^n \ell(a^Tk_n(x_i),y_i)) + \lambda a^T K_na \right\},
\end{equation*}
where 
\begin{equation*}
    k_n(x) = (K(x, x_i))_{i \in [n]}, \qquad K_n = K(x_i,x_j)_{i,j \in [n]}.
\end{equation*}
Taking $\ell$ to be convex in the first parameter, this is a convex optimization problem in $\mathbb{R}^n$, even when $\Phi$ is non-linear or $\mathscr{F}$ is infinite dimensional. 

Now suppose that a group $G$ acts on $\mathcal{X}$ such that $g_* \mu_\mathcal{X} = \mu_\mathcal{X}$ for all $g \in G$. Let $\mathscr{H}$ be an RHKS with a $G$-invariant kernel:
\begin{equation*}
    G \subseteq \mathrm{Aut}(K) := \{g : \mathcal{X} \to \mathcal{X} : K(gx,gy) = K(x,y) \text{ for all } x,y\in \mathcal{X}\}. 
\end{equation*}
A few examples of $G$-equivariant kernel methods include the following
\begin{enumerate}
    \item $K(x,y) = f(x \cdot y)$ as for linear regression gives us $\mathrm{Aut}(K) = \mathrm{O}_d(\mathbb{R})$.
    \item $K(x,y) = \psi(\norm{x-y}^2)$ as for Gaussian/RBF kernel gives us $\mathrm{Aut}(K) := \bb{R}^d \rtimes \mathrm{O}_d(\mathbb{R})$.
    \item Ridge regression with any $G$-invariant kernel.
    \item Support vector machines with any $G$-invariant kernel.
    \item For $K$ the infinite-width two-layer NTK \cite{jacot2018ntk} we have $\mathrm{Aut}(K) = \mathrm{O}_d(\mathbb{R})$.
\end{enumerate}
Let $\ell : \mathbb{R} \times \mathcal{Y} \to \mathbb{R}$ be a loss function, $\lambda > 0$, and for a sample $S \in (\mathcal{X} \times \mathcal{Y})^n$ consider the algorithm
\begin{equation}\label{eq.kernel_alg}
    \mathcal{A}(S) := \operatorname*{argmin}_{h \in \mathscr{H}_K}\left\{
    \frac{1}{n}\sum_{i=1}^n \ell(h(x_i),y_i) + \lambda \|h\|_{\mathscr{H}_K}^2
    \right\}.
\end{equation}
For each $g \in G$ denote by $U_g : \mathscr{H}_K \to \mathscr{H}_K$ the operator $h \mapsto h \circ g$.
\begin{lem}\label{lem.G_inv_kernel}
    For a $G$-invariant kernel $K$, the operator
    \begin{equation*}
        U_g : \mathscr{H}_K \to \mathscr{H}_K
    \end{equation*}
    is an isometry for all $g \in G$.
\end{lem}

\begin{proof}
    For $y \in X$ let $K_y = K(\cdot, y)$ and observe that $U_gK_y = K_{g^{-1}y}$ since
    \begin{equation*}
        (U_g K_y)(x) = K_y(gx) = K(gx,y) = K(x,g^{-1}y).
    \end{equation*}
    By $G$-invariance we have
    \begin{equation*}
        \langle U_g K_x, U_g K_y \rangle_{\mathscr{H}_K} = \langle K_x, K_y \rangle_{\mathscr{H}_K}.
    \end{equation*}
    Now observe that $V = \text{span}\{K_y : y \in X\}$ is dense since $V^\perp = \{0\}$:
    \begin{equation*}
        \langle f, K_y \rangle_{\mathscr{H}_K} =0 \implies f(y) = \langle f, K_y \rangle_{\mathscr{H}_K} = 0 \implies f = 0,
    \end{equation*}
    by the reproducing kernel property. We conclude the proof by linearity and density. 
\end{proof}

\begin{prop}\label{prop.G_inv_kernel}
    For a $G$-invariant kernel $K$ with a unique solution to Equation~\eqref{eq.kernel_alg}, the kernel method is $G$-equivariant: $  \mathcal{A}(g^{-1}S) = \mathcal{A}(S) \circ g$.
\end{prop}

\begin{proof}
    Let 
    \begin{equation*}
        J_S(h) := \frac{1}{n} \sum_{i=1}^n \ell(h(x_i), y_i) + \lambda \norm{h}^2
    \end{equation*}
    and observe that
    \begin{align*}
    J_{g^{-1}S}(U_g h)
        &= \frac{1}{n}\sum_{i=1}^n \ell\bigl((U_g h)(g^{-1}x_i),y_i\bigr) +\lambda \|U_g h\|_{\mathscr{H}_K}^2 \\
        &= \frac{1}{n}\sum_{i=1}^n \ell\bigl(h(x_i),y_i\bigr) + \lambda \|h\|_{\mathscr{H}_K}^2 \\
        &= J_S(h)
    \end{align*}
    since $U_g$ is an isometry. Therefore, $h \mapsto U_g h$ transports the objective function for $g^{-1}S$ to the objective function for $S$. Now take any $f \in \mathscr{H}_K$ and observe that
    \begin{equation*}
        J_{g^{-1}S}(f) = J_{g^{-1}S}(U_gU_{g^{-1}}f) = J_S( U_{g^{-1}}f) \geq J_S(\mathcal{A}(S)) = J_{g^{-1}S}(U_g \mathcal{A}(S)),
    \end{equation*}
    so $U_g \mathcal{A}(S)$ minimizes $J_{g^{-1}S}$, concluding the proof.
\end{proof}

\begin{thm}\label{thm:kernel_general_lb}
    Let $\mathcal{A}$ be a $G$-equivariant kernel method and let $\mu_\mathcal{X}$ be $G$-invariant with $G$ a finite group. For any $\varepsilon > 0$ and any $h \in L^2(\mathcal{X})$ with $\norm{h}_2 = 1$, suppose that kernel method $\mathcal{A}(S)$ on sample $S = (x_i,y_i)_{i \in [n]}$ achieves 
   \begin{equation*}
       \norm{\mathcal{A}(S)-h}_2^2 < \varepsilon.
   \end{equation*} 
    Then
    \begin{equation*}
        \min(n, \eta) \geq \frac{|G|}{\norm{\Gamma}_{\mathrm{op}}} (1-\varepsilon) = \frac{1-\varepsilon}{\mathrm{A}(h,G)}
    \end{equation*}
    where $\eta = \dim \mathscr{H}_K$ and $\Gamma = (\langle U_{g}h, U_{g'}h \rangle)_{g,g' \in G}$ is the Gram matrix of $U_Gh = \{U_gh : g \in G\}$.
\end{thm}

\begin{proof}
    Suppose that 
    \begin{equation*}
        \int_\mathcal{X} |\mathcal{A}(S) - h |^2d\mu_\mathcal{X} =   \norm{\mathcal{A}(S)-h}_2^2 < \varepsilon,
    \end{equation*}
    and note that for any $g \in G$, applying the kernel method to $U_gh$ yields $U_g \mathcal{A}(S)$ by Proposition~\ref{prop.G_inv_kernel}. By Lemma~\ref{lem.G_inv_kernel}, we have
    \begin{equation*}
       \norm{U_g\mathcal{A}(S) - U_{g}h}^2_2 = \norm{\mathcal{A}(S)-h}_2^2 < \varepsilon,
    \end{equation*}
    so we conclude that
    \begin{equation*}
        \frac{1}{|G|} \sum_{g \in G} \norm{U_g\mathcal{A}(S)-U_gh}_2^2 < \varepsilon.
    \end{equation*}
    The lower bound now follows from \cite[Proposition 11]{abbe2024mergedstaircasepropertynecessarynearly}.

    For the equality, consider the operator
    \begin{equation*}
        T : \ell^2(G) \to L^2(\mathcal{X}), \qquad T(c) := \frac{1}{\sqrt{|G|}} \sum_{g \in G} c_g U_gh.
    \end{equation*}
    It is easy to see that the adjoint is 
    \begin{equation*}
        T^* : L^2(\mathcal{X}) \to \ell^2(G), \qquad (T^*\phi)_g = \frac{1}{\sqrt{|G|}} \langle \phi, U_g h \rangle.
    \end{equation*}
    Now we compute
    \begin{equation*}
        TT^*(\phi) = \frac{1}{|G|} \sum_{g \in G} \langle \phi, U_g h \rangle U_gh = \mathbb{E}_{g \sim \mu_G}[\langle \phi, U_gh\rangle U_gh],
    \end{equation*}
    and
    \begin{equation*}
         (T^*Tc)_{g'} = \frac{1}{|G|} \sum_{g \in G} \langle U_{g'}h, U_g h \rangle c_{g'}.
    \end{equation*}
     Therefore $T^*T$ is represented by the matrix $\Gamma/|G|$, and we conclude
    \begin{equation*}
        \frac{\norm{\Gamma}_{\mathrm{op}}}{|G|} = \norm{T^*T}_{\mathrm{op}} = \norm{TT^*}_{\mathrm{op}} = \sup_{\norm{\phi}_2=1} \langle TT^*\phi,\phi \rangle =\sup_{\norm{\phi}_2=1} \mathbb{E}_{g \sim \mu_G}[|\langle U_gh, \phi \rangle |^2] = \mathrm{A}(h,G),
    \end{equation*}
    using the fact that $TT^*$ is a positive self-adjoint operator. 
\end{proof}

\subsection{Noisy gradient descent}
Suppose that $\mathcal{X} \subseteq \mathbb{R}^d$ and we have a parameterized model $f(\cdot ; \theta) : \mathcal{X} \to \mathbb{R}$ with parameters $\theta \in \mathbb{R}^\eta$ such that $\nabla_\theta f(\cdot; \theta)$ exists for a.e. $x \in \mathcal{X}, \theta \in \mathbb{R}^\eta$. We train the model using the following noisy gradient descent algorithm. Begin with a random initialization $\theta_0 \sim \mu_0$ according to an initialization distribution $\mu_0$. For a target $h \in L^2(\mathcal{X})$, we define the clipped population gradient
\begin{equation*}
    G_h(\theta) := -\mathbb{E}_{x \sim \mu_\mathcal{X}}[(h(x)-f(x;\theta)\Pi_{R}\nabla_\theta f(x;\theta)],
\end{equation*}
where $\Pi_{R}$ denote projection onto a ball of radius $R > 0$. We iteratively update the parameters via the rule
\begin{equation}\label{eq.NGD_rule}
    \theta_{k+1} = \theta_k -  \eta_k(G_h(\theta_k) + \xi_k),
\end{equation}
where $\eta_k > 0$ is a learning rate and $\xi_k \sim N(0, \tau^2I)$ is i.i.d. Gaussian noise. After $T \geq 1$ iterations, we output the model $f(\cdot ; \theta_T) : \mathcal{X} \to \mathbb{R}$. We give the following sufficient condition for such an algorithm to be $G$-equivariant. 

\begin{prop}\label{prop:G_inv_NGD_condition}
    Suppose that $G$ acts on $\mathcal{X}$ with $\mu_\mathcal{X}$ a $G$-invariant measure. Assume the following:
    \begin{enumerate}
        \item There is an orthogonal representation $\rho : G \to \mathrm{O}(\mathbb{R}^\eta)$ such that
        \begin{equation*}
            f(g(x); \rho(g)\theta) = f(x;\theta) \quad \forall x \in \mathcal{X}, g \in G, \theta \in \mathbb{R}^\eta.
        \end{equation*}
        \item The initial distribution is $\rho(G)$-invariant: $(\rho(g))_*\mu_0 = \mu_0$ for all $g  \in G$.
    \end{enumerate}
    Then the noisy gradient descent algorithm is $G$-equivariant. 
\end{prop}

\begin{proof}
    We use the more general result \cite[Theorem C.1]{li2021convnets}. It is clear that their conditions (1) and (3) follow from our assumptions. It remains to show that the update rule in Equation~\eqref{eq.NGD_rule} is $G$-equivariant. Since $\rho$ lands in $O(\mathbb{R}^\eta)$ and $\xi_k$ is Gaussian, the noise is invariant in distribution. Observe that condition (1) implies
    \begin{align}\label{eq.NGD_equivariant_step1}
    \begin{split}
        G_{U_{g^{-1}}h}(\rho(g)\theta)& = -\mathbb{E}_{x \sim \mu_\mathcal{X}}[(h(g^{-1}x)-f(x; \rho(g)\theta))\Pi_{R}\nabla_\theta f(x;\rho(g)\theta)]\\
        &= -\mathbb{E}_{x \sim \mu_\mathcal{X}}[(h(x)-f(x; \theta))\Pi_{R}\nabla_\theta f(gx;\rho(g)\theta)]
    \end{split}
    \end{align}
    Applying the chain rule to condition (1), we find
    \begin{equation*}
        \rho(g)^T \nabla_\theta f(gx; \rho(g)\theta) = \nabla_\theta f(x,\theta),
    \end{equation*}
    and since $\rho(g)$ is orthogonal, 
    \begin{equation}\label{eq.NGD_equivariant_step2}
        \nabla_\theta f(gx; \rho(g)\theta) =\rho(g) \nabla_\theta f(x,\theta).
    \end{equation}
    Using the fact that $\Pi_R$ commutes with orthogonal transformations, Equation~\eqref {eq.NGD_equivariant_step1} and Equation~\eqref {eq.NGD_equivariant_step2} imply
    \begin{equation*}
        G_{U_{g^{-1}}h}(\rho(g)\theta) = -\rho(g)\mathbb{E}_{x \sim \mu_\mathcal{X}}[(h(x)-f(x; \theta))\Pi_{R}\nabla_\theta f(x;\theta)] = \rho(g) G_h(\theta),
    \end{equation*}
    concluding the proof.
\end{proof}

We get that the following deep learning architectures are naturally $G$-equivariant.
\begin{enumerate}
    \item Fully connected neural networks (FCNN) with $G = O(\mathbb{R}^d)$ and $\rho :  O(\mathbb{R}^d) \to O(\mathbb{R}^p)$ moving rotation to the first layer weights if we use i.i.d. Gaussian initialization.
    \item Cylic convolutional neural networks (CCNN) with cyclic invariant readout, $G = \mathbb{Z}/d\mathbb{Z}$, and $\rho = I$ since the architecture is $G$ invariant.
    \item Transformer without positional embedding, global self-attention, and permutation invariant pooling with $G = S_n$ permuting the input sequence and $\rho = I$ if we use i.i.d. initialization. 
\end{enumerate}
Now we establish the lower bound on learning.

\begin{thm}\label{thm:NGD_equivariant}
    Fix a compact group $G$. Let $\mathcal{A}$ be a $G$-equivariant NGD algorithm corresponding to parametric model $f(\cdot;\theta)$ and let $\mu_\mathcal{X}$ be $G$-invariant. Fix any target function $h : \mathcal{X} \to \{-1,1\}$ and any baseline function $h_* : \mathcal{X} \to \{-1,1\}$ satisfying $h_* \circ g = h_*$ for all $g \in G$. For any $T \in \mathbb{Z}_{\geq 1}$, $\tau > 0$, $R > 0$, and $\varepsilon > 0$, we have
    \begin{equation*}
        \mathbb{P}_{\theta_T}[\norm{f(\cdot; \theta_h^T)-h}_2^2 \leq \norm{h_*-h}_2^2 - \varepsilon] \leq \frac{R}{2\tau}\sqrt{T\mathrm{A}} + \frac{\mathrm{A}}{\varepsilon}      
    \end{equation*}
    where $\mathrm{A} := \mathrm{A}(h,G;h_*)$
\end{thm}

\begin{proof}
Define a probability measure $\nu_h$ on $L^2(\mathcal{X})$ by taking the law of $U_g h$ when $g \sim \mu_G$. We first observe that the success event is constant along the $G$-orbit of $h$. Indeed, for each $g \in G$, $G$-equivariance of the NGD algorithm implies
\begin{equation*}
    f(\cdot;\theta_{U_g h}^T) \overset{d}{=} U_g f(\cdot;\theta_h^T).
\end{equation*}
Since $\mu_{\mathcal X}$ is $G$-invariant, the operator $U_g$ acts unitarily on $L^2(\mu_{\mathcal X})$, and hence
\begin{equation*}
    \|U_g \varphi\|_2 = \|\varphi\|_2
    \qquad \forall \varphi \in L^2(\mu_{\mathcal X}).
\end{equation*}
Therefore
\begin{equation*}
    \|f(\cdot;\theta_{U_g h}^T)-U_g h\|_2^2
    \overset{d}{=}
    \|U_g(f(\cdot;\theta_h^T)-h)\|_2^2
    =
    \|f(\cdot;\theta_h^T)-h\|_2^2.
\end{equation*}
Moreover, since $h_*$ is $G$-invariant, we have
\begin{equation*}
    \|h_* - U_g h\|_2^2
    =
    \|U_g h_* - U_g h\|_2^2
    =
    \|h_* - h\|_2^2.
\end{equation*}
It follows that for every $g \in G$,
\begin{equation*}
    \mathbb{P}_{\theta_T}
    \!\left[
        \|f(\cdot;\theta_{U_g h}^T)-U_g h\|_2^2
        \leq
        \|h_* - U_g h\|_2^2 - \varepsilon
    \right]
    =
    \mathbb{P}_{\theta_T}
    \!\left[
        \|f(\cdot;\theta_h^T)-h\|_2^2
        \leq
        \|h_* - h\|_2^2 - \varepsilon
    \right].
\end{equation*}
Averaging over $g \sim \mu_G$ gives
\begin{equation*}
\begin{aligned}
    \mathbb{P}_{\theta_T}
    \!\left[
        \|f(\cdot;\theta_h^T)-h\|_2^2
        \leq
        \|h_* - h\|_2^2 - \varepsilon
    \right]
    &=
    \mathbb{E}_{g \sim \mu_G}
    \mathbb{P}_{\theta_T}
    \!\left[
        \|f(\cdot;\theta_{U_g h}^T)-U_g h\|_2^2
        \leq
        \|h_* - U_g h\|_2^2 - \varepsilon
    \right] \\
    &=
    \mathbb{P}_{g \sim \mu_G,\theta_T}
    \!\left[
        \|f(\cdot;\theta_{U_g h}^T)-U_g h\|_2^2
        \leq
        \|h_* - U_g h\|_2^2 - \varepsilon
    \right].
\end{aligned}
\end{equation*}
using the independence of $\mu_G$ and the randomness in $\theta_T$ and Fubini's theorem. Now apply \cite[Theorem A.2]{abbe2022nonuniversalitydeeplearningquantifying} to the distribution $\nu_h$ on $L^2(\mathcal{X})$ with baseline function $h_*$. The corresponding alignment parameter is
\begin{equation*}
    \mathrm{A}
    :=
    \sup_{\|\phi\|_2=1}
    \mathbb{E}_{t \sim \nu_h}
    \big[\langle t-h_*,\phi\rangle^2\big]
    =
    \sup_{\|\phi\|_2=1}
    \mathbb{E}_{g \sim \mu_G}
    \big[\langle U_g h-h_*,\phi\rangle^2\big].
\end{equation*}
Hence we find
\begin{equation*}
    \mathbb{P}_{g \sim \mu_G,\theta_T}
    \!\left[
        \|f(\cdot;\theta_{U_g h}^T)-U_g h\|_2^2
        \leq
        \|h_* - U_g h\|_2^2 - \varepsilon
    \right]
    \leq
    \frac{R}{2\tau}\sqrt{T\mathrm{A}} + \frac{\mathrm{A}}{\varepsilon }.
\end{equation*}
Combining this with the previous identity proves the claim.
\end{proof}

\subsection{Correlational statistical queries}\label{ssec:CSQ}

In the statistical query (SQ) model of learning, we access data through an $\mathrm{Stat}(\tau)$ oracle as follows. The SQ algorithm $\mathcal{A}$ takes a data distribution $\mathcal{D}$ on $\mathcal{X} \times \mathcal{Y}$ and operates for $q$ rounds.
\begin{enumerate}
    \item For each $t \in [q]$, the algorithm issues a query $\phi_t : \mathcal{X} \times \mathcal{Y} \to [-1,1]$ and receives a response $v_t$ from $\mathrm{Stat}(\tau)$ satisfying
    \begin{equation*}
        |v_t -\mathbb{E}_{\mathcal{D}}[\phi_t(x,y)] \leq \tau,
    \end{equation*}
    and the query $\phi_t$ can be chosen adaptively, i.e. depending on $v_1,\ldots,v_{t-1}$.
    \item After $q$ rounds, the algorithm outputs $\mathcal{A}(\mathcal{D}) : \mathcal{X} \to \mathcal{Y}$ based on $v_1,\ldots,v_q$.
\end{enumerate}
As shown in \cite{Kearns1998}, SQ algorithms are strictly weaker than algorithms with access to i.i.d. samples. Moreover, success in the SQ model requires the algorithm to be robust against adversarial noise. We will consider a correlational statistical query $\mathrm{CSQ}(\tau)$ oracle, which is further restricted to functions of the form $\phi(x,y) = \phi(x)\cdot y$. Note that for binary-valued functions, one can simulate $\mathrm{Stat}(\tau)$ with $\mathrm{CSQ}(\tau)$, so the models are equivalent in the fixed distribution setting. Finally, we remark that gradient descent for $T$ steps and $p$ parameters can be implemented with $pT$ CSQ queries, each with tolerance $\tau \approx 1/\sqrt{n}$. Other loss functions can be implemented with SQ queries. 

In the rest of this section, we will develop lower bounds that parallel the story for NGD methods. This will be largely based on the ideas in \cite[Section 4]{misiakiewicz2025shorttutorial}. We will first prove lower bounds on a detection problem, where we ask for a CSQ algorithm to separate a null $\mathrm{H}_0$ from an alternative $\mathrm{H}_1$ hypothesis:
\begin{equation}\label{eq.detection_problem}
    \mathrm{H}_1 : \mathcal{D} \in \mathcal{H}, \qquad \mathrm{H}_0 : \mathcal{D} = \mathcal{D}_0.
\end{equation}
Next, we prove a reduction from detection to learning. Finally, given a bound on learning a hypothesis class, we impose a $G$-equivariance condition to give lower bounds for learning a fixed function.

\begin{lem}\label{lem:distinguish_lb}
Let $\mu_\mathcal{X}$ be an input distribution on $\mathcal{X}$, let $h_* \in L^2(\mathcal{X})$ with $\norm{h_*}_{\infty} \leq 1$, and let $\mathcal{H}\subseteq L^2(\mathcal{X})$ be a class of regression functions with prior $\mu_\mathcal{H}$ and $\norm{h}_{\infty} \leq 1$ for all $h \in \mathcal{H}$. For each $t\in \{h_*\}\cup \mathcal{H}$, let $\mathcal{D}_t$ be a distribution on $\mathcal{X}\times \mathcal{Y}$ with marginal $\mu_\mathcal{X}$ on $\mathcal{X}$ and regression function
\begin{equation*}
    \mathbb{E}_{\mathcal{D}_t}[y\mid x]=t(x).
\end{equation*}
Let $\mathcal{A}$ be a randomized $\mathrm{CSQ}(\tau)$ algorithm making at most $q$ queries. Denote
\begin{equation*}
    \alpha
    :=
    \sup_{\mathcal{O}_{h_*}}
    \mathbb{P}_{R}\!\left[\mathcal{A}^{\mathcal{D}_{h_*},\mathcal{O}_{h_*},R}\text{ outputs }\mathrm{H}_1\right],
\end{equation*}
with supremum is over all $\mathrm{CSQ}(\tau)$ oracle strategies for $\mathcal{D}_g$. For each $h\in \mathcal{H}$, denote
\begin{equation*}
    \beta(h)
    :=
    \inf_{\mathcal{O}_h}
    \mathbb{P}_{R}\!\left[\mathcal{A}^{D_h,\mathcal{O}_h,R}\text{ outputs }\mathrm{H}_1\right],
\end{equation*}
with infimum is over all $\mathrm{CSQ}(\tau)$ oracle strategies for $\mathcal{D}_h$, and set $\beta:= \mathbb{E}_{h\sim \mu_{\mathcal H}}[\beta(h)]$. Then
\begin{equation*}
    \beta-\alpha
    \leq
    \frac{q}{\tau^2} \mathrm{A}(\mu_\mathcal{H},h_*).
\end{equation*}
\end{lem}

\begin{proof}
Let $\mathcal{O}_{h_*}^\ast$ be the exact CSQ oracle for $\mathcal{D}_{h_*}$, so that on query $\phi$ it returns
\begin{equation*}
    \mathbb{E}_{\mathcal{D}_{h_*}}[y\phi(x)]
    =
    \mathbb{E}_{\mu_\mathcal{X}}[h_*(x)\phi(x)].
\end{equation*}
Let $R$ denote the internal random seed of $\mathcal{A}$ and fix a seed $r$. Running $\mathcal{A}$ against $\mathcal{D}_{h_*}$ with oracle $\mathcal{O}_{h_*}^\ast$ produces a deterministic sequence of queries $ \phi_1^{(r)},\ldots,\phi_q^{(r)}$, and corresponding oracle responses
\begin{equation*}
    v_t^{(r)}
    :=
    \mathbb{E}_{\mathcal{D}_{h_*}}[y\,\phi_t^{(r)}(x)]
    =
    \mathbb{E}_{\mu_\mathcal{X}}[h_*(x)\phi_t^{(r)}(x)].
\end{equation*}

For $h\in \mathcal H$, define the bad event
\begin{equation*}
    B_r(h)
    :=
    \left\{
        \exists t\in [q]:
        \left|
            \mathbb{E}_{\mathcal{D}_h}[y\,\phi_t^{(r)}(x)]
            -
            v_t^{(r)}
        \right|
        >
        \tau
    \right\}.
\end{equation*}
Now define a $\mathrm{CSQ}(\tau)$ oracle strategy $\mathcal{O}_h^\ast$ for $D_h$ as follows: whenever the algorithm asks a query $\phi$, the oracle attempts to answer with the exact null response $\mathbb{E}_{\mathcal{D}_{h_*}}[y\,\phi(x)]$. If this value is $\tau$-compatible with $\mathcal{D}_h$, the oracle returns it; otherwise it returns any arbitrary $\tau$-compatible value for $\mathcal{D}_h$.

For a fixed seed $r$, if $B_r(h)$ does not occur, then along the interaction path generated by the queries $\phi_1^{(r)},\ldots,\phi_q^{(r)}$, all the null responses $v_1^{(r)},\ldots,v_q^{(r)}$ are $\tau$-compatible with $\mathcal{D}_h$. Hence, when run with seed $r$, the interaction of $\mathcal{A}$ with $(\mathcal{D}_h,\mathcal{O}_h^\ast)$ produces exactly the same transcript as its interaction with $(\mathcal{D}_{h_*},\mathcal{O}_{h_*}^\ast)$. Therefore the outputs coincide. It follows that
\begin{equation*}
    \mathbf{1}_{\{\mathcal{A}^{D_h,\mathcal{O}_h^\ast,R}\text{ outputs }\mathrm{H}_1\}}
    -
    \mathbf{1}_{\{\mathcal{A}^{\mathcal{D}_{h_*},\mathcal{O}_{h_*}^\ast,R}\text{ outputs }\mathrm{H}_1\}}
    \leq
    \mathbf{1}_{B_R(h)}.
\end{equation*}
Taking expectation over $R$ gives
\begin{equation*}
    \mathbb{P}_{R}\!\left[\mathcal{A}^{\mathcal{D}_h,\mathcal{O}_h^\ast,R}\text{ outputs }\mathrm{H}_1\right]
    -
    \mathbb{P}_{R}\!\left[\mathcal{A}^{\mathcal{D}_{h_*},\mathcal{O}_{h_*}^\ast,R}\text{ outputs }\mathrm{H}_1\right]
    \leq
    \mathbb{P}_{R}(B_R(h)).
\end{equation*}
Since $\beta(h) \leq \mathbb{P}_{R}\!\left[\mathcal{A}^{\mathcal{D}_h,\mathcal{O}_h^\ast,R}\text{ outputs }\mathrm{H}_1\right]$ and $\alpha \geq \mathbb{P}_{R}\!\left[\mathcal{A}^{\mathcal{D}_{h_*},\mathcal{O}_{h_*}^\ast,R}\text{ outputs }\mathrm{H}_1\right]$, we obtain
\begin{equation*}
    \beta(h)-\alpha
    \leq
    \mathbb{P}_{R}(B_R(h)).
\end{equation*}
Averaging over $h\sim \mu_{\mathcal H}$ yields
\begin{equation*}
    \beta-\alpha
    \leq
    \mathbb{P}_{R,\,h\sim \mu_{\mathcal H}}\!\left(B_R(h)\right).
\end{equation*}

It remains to bound the right-hand side. Fix $r$ and apply the union bound to get
\begin{equation*}
    \mathbb{P}_{h\sim \mu_{\mathcal H}}(B_r(h))
    \leq
    \sum_{t=1}^q
    \mathbb{P}_{h\sim \mu_{\mathcal H}}
    \left(
        \left|
            \mathbb{E}_{\mathcal{D}_h}[y\,\phi_t^{(r)}(x)]
            -
            \mathbb{E}_{\mathcal{D}_{h_*}}[y\,\phi_t^{(r)}(x)]
        \right|
        >
        \tau
    \right).
\end{equation*}
Applying Markov's inequality to each summand yields
\begin{equation*}
    \mathbb{P}_{h\sim \mu_{\mathcal H}}(B_r(h))
    \leq
    \frac{1}{\tau^2}
    \sum_{t=1}^q
    \mathbb{E}_{h\sim \mu_{\mathcal H}}
    \left[
        \left(
            \mathbb{E}_{\mathcal{D}_h}[y\,\phi_t^{(r)}(x)]
            -
            \mathbb{E}_{\mathcal{D}_{h_*}}[y\,\phi_t^{(r)}(x)]
        \right)^2
    \right].
\end{equation*}
Since $\mathcal{A}$ is a bounded-output $\mathrm{CSQ}(\tau)$ algorithm, each query satisfies $\|\phi_t^{(r)}\|_{2}\leq 1$. By the definition of $\mathcal{D}_t$ we obtain
\begin{align*}
    \mathbb{P}_{h\sim \mu_{\mathcal H}}(B_r(h))
    &\leq
    \frac{q}{\tau^2}
    \sup_{\|\phi\|_{2}\leq 1}
    \mathbb{E}_{h\sim \mu_{\mathcal H}}
    \left[
        \left(
            \mathbb{E}_{\mu_\mathcal{X}}[(h(x)-h_*(x))\phi(x)]
        \right)^2
    \right] \\
    &=
    \frac{q}{\tau^2}
    \sup_{\|\phi\|_{2}\leq 1}
    \mathbb{E}_{h\sim \mu_{\mathcal H}}
    \bigl[
        |\langle h-h_*,\phi\rangle|^2
    \bigr] \\
    &=
    \frac{q}{\tau^2} \mathrm{A}(\mu_\mathcal{H},h_*).
\end{align*}
This bound is uniform in $r$, so averaging over $R$ gives
\begin{equation*}
    \mathbb{P}_{R,\,h\sim \mu_{\mathcal H}}\!\left(B_R(h)\right)
    \leq
    \frac{q}{\tau^2}\mathrm{A}(\mu_\mathcal{H},g).
\end{equation*}
Combining this with the previous inequality proves that
\begin{equation*}
    \beta-\alpha
    \leq
    \frac{q}{\tau^2}\mathrm{A}(\mu_\mathcal{H},g).
\end{equation*}
\end{proof}

We now prove the relatively obvious statement that, given an algorithm that can learn a hypothesis class $\mathcal{H}$, it can solve the detection problem~\eqref{eq.detection_problem}.

\begin{lem}
\label{lem:avg-randomized-learning-to-detection-csq}
With notation as in the previous lemma, assume there is a randomized $\mathrm{CSQ}(\tau)$ algorithm $\mathcal{A}$ such that, for every $h\in \mathcal{H}$ and every $\mathrm{CSQ}(\tau)$ oracle strategy $\mathcal{O}_h$ for $\mathcal{D}_h$, the algorithm makes at most $q$ queries and outputs a predictor
\begin{equation*}
    \mathcal{A}^{\mathcal{D}_h,\mathcal{O}_h,R}:\mathcal{X}\to[-1,1].
\end{equation*}
Assume moreover that
\begin{equation*}
    \mathbb{E}_{h\sim \mu_{\mathcal H}}
    \inf_{\mathcal{O}_h}
    \mathbb{P}_{R}\!\left[
        \|\mathcal{A}^{\mathcal{D}_h,\mathcal{O}_h,R}-h\|_{2}^2
        \leq
        \|h_*-h\|_{2}^2-\varepsilon
    \right]
    \geq p
\end{equation*}
for some $\varepsilon>0$ and $p\in (0,1]$. If $\tau<\varepsilon/8$, then there is a randomized bounded-output $\mathrm{CSQ}(\tau)$ algorithm $\mathcal{B}$ which makes at most $q+1$ queries and satisfies
\begin{equation*}
    \sup_{\mathcal{O}_{h_*}}
    \mathbb{P}_{R}\!\left[\mathcal{B}^{\mathcal{D}_{h_*},\mathcal{O}_{h_*},R}\text{ outputs }\mathrm{H}_1\right]
    =0 \quad \text{ and } \quad  \mathbb{E}_{h\sim \mu_{\mathcal H}}
    \inf_{\mathcal{O}_h}
    \mathbb{P}_{R}\!\left[\mathcal{B}^{\mathcal{D}_h,\mathcal{O}_h,R}\text{ outputs }\mathrm{H}_1\right]
    \geq p.
\end{equation*}
\end{lem}

\begin{proof}
Given bounded-output $\mathrm{CSQ}(\tau)$ access to an unknown distribution $\mathcal{D}_t$, where either $t={h_*}$ or $t\in \mathcal{H}$, and given a $\tau$-compatible oracle strategy $\mathcal{O}_t$, the algorithm $\mathcal{B}$ proceeds as follows. First, using the same random seed $R$ throughout, it simulates the learner $\mathcal{A}$ on the oracle $\mathcal{O}_t$ for $\mathcal{D}_t$ and obtains an output predictor
\begin{equation*}
    \widehat{t}:=\mathcal{A}^{\mathcal{D}_t,\mathcal{O}_t,R}:\mathcal{X}\to[-1,1].
\end{equation*}
Next, it defines the query function
\begin{equation*}
    \phi(x):=\frac{\widehat{t}(x)-g(x)}{2},
\end{equation*}
with $\norm{\phi}_\infty \leq 1$. The oracle returns a value $v$ satisfying
\begin{equation*}
    \left|v-\mathbb{E}_{\mathcal{D}_t}[y\phi(x)]\right|\leq \tau.
\end{equation*}
Because $ \mathbb{E}_{\mathcal{D}_t}[y\phi(x)] = \langle t,\phi\rangle$, this means $|v-\langle t,\phi\rangle|\leq \tau$. Since ${h_*}$ and $\mu_\mathcal{X}$ are known, the quantity $\langle {h_*},\phi\rangle$ is known to the algorithm, and hence it can form $ s:=v-\langle {h_*},\phi\rangle$. Then
\begin{equation*}
    |s-\langle t-{h_*},\phi\rangle|\leq \tau.
\end{equation*}
We have $\mathcal{B}$ output $\mathrm{H}_1$ if $s > \varepsilon/8$ and output $\mathrm{H}_0$ otherwise.

We now verify correctness. First, suppose $t=h_*$. Then $    \langle t-{h_*},\phi\rangle=0$, so
\begin{equation*}
    |s|\leq \tau<\frac{\varepsilon}{8}.
\end{equation*}
Hence $\mathcal{B}$ always outputs $\mathrm{H}_0$, regardless of the seed $R$ and regardless of the $\tau$-compatible oracle strategy $\mathcal{O}_{h_*}$. Therefore
\begin{equation*}
    \sup_{\mathcal{O}_{h_*}}
    \mathbb{P}_{R}\!\left[\mathcal{B}^{\mathcal{D}_{h_*},\mathcal{O}_{h_*},R}\text{ outputs }\mathrm{H}_1\right]
    =0.
\end{equation*}
Now suppose $t=h\in \mathcal{H}$, and fix a $\tau$-compatible oracle strategy $\mathcal{O}_h$ for $\mathcal{D}_h$. Define the success event
\begin{equation*}
    E_{h,\mathcal{O}_h}
    :=
    \left\{
        \|\mathcal{A}^{\mathcal{D}_h,\mathcal{O}_h,R}-h\|_2^2
        \leq
        \|h_*-h\|_2^2-\varepsilon
    \right\}.
\end{equation*}
Expanding the left-hand side around ${h_*}$ gives
\begin{equation*}
    \|\mathcal{A}^{\mathcal{D}_h,\mathcal{O}_h,R}-h\|_2^2
    =
    \|h-{h_*}\|_2^2
    +
    \|\mathcal{A}^{\mathcal{D}_h,\mathcal{O}_h,R}-{h_*}\|_2^2
    -
    2\langle h-{h_*},\mathcal{A}^{\mathcal{D}_h,\mathcal{O}_h,R}-{h_*}\rangle.
\end{equation*}
Thus, on the event $E_{h,\mathcal{O}_h}$, we have
\begin{equation*}
    2\langle h-h_*,\mathcal{A}^{\mathcal{D}_h,\mathcal{O}_h,R}-{h_*}\rangle
    \geq
    \|\mathcal{A}^{\mathcal{D}_h,\mathcal{O}_h,R}-{h_*}\|_2^2+\varepsilon
    \geq
    \varepsilon.
\end{equation*}
Since in this case
\begin{equation*}
    \phi=\frac{\mathcal{A}^{\mathcal{D}_h,\mathcal{O}_h,R}-{h_*}}{2},
\end{equation*}
it follows that
\begin{equation*}
    \langle h-{h_*},\phi\rangle
    =
    \frac{1}{2}\langle h-{h_*},\mathcal{A}^{\mathcal{D}_h,\mathcal{O}_h,R}-{h_*}\rangle
    \geq
    \frac{\varepsilon}{4}.
\end{equation*}
Therefore, on the event $E_{h,\mathcal{O}_h}$,
\begin{equation*}
    s
    \geq
    \frac{\varepsilon}{4}-\tau
    >
    \frac{\varepsilon}{8},
\end{equation*}
so $\mathcal{B}$ outputs $\mathrm{H}_1$. Thus
\begin{equation*}
    E_{h,\mathcal{O}_h}
    \subseteq
    \left\{
        \mathcal{B}^{\mathcal{D}_h,\mathcal{O}_h,R}\text{ outputs }\mathrm{H}_1
    \right\},
\end{equation*}
and hence
\begin{equation*}
    \mathbb{P}_{R}\!\left[
        \mathcal{B}^{\mathcal{D}_h,\mathcal{O}_h,R}\text{ outputs }\mathrm{H}_1
    \right]
    \geq
    \mathbb{P}_{R}(E_{h,\mathcal{O}_h}).
\end{equation*}
Taking the infimum over $\mathcal{O}_h$ and averaging over $h \sim \mu_\mathcal{H}$ yields
\begin{equation*}
    \mathbb{E}_{h\sim \mu_{\mathcal H}}
    \inf_{\mathcal{O}_h}
    \mathbb{P}_{R}\!\left[
        \mathcal{B}^{\mathcal{D}_h,\mathcal{O}_h,R}\text{ outputs }\mathrm{H}_1
    \right]
    \geq
    \mathbb{E}_{h\sim \mu_{\mathcal H}}
    \inf_{\mathcal{O}_h}
    \mathbb{P}_{R}(E_{h,\mathcal{O}_h})
    \geq
    p,
\end{equation*}
proving the claim.
\end{proof}

\begin{thm}
Suppose that $G$ acts on $\mathcal{X}$ and $\mu_\mathcal{X}$ is $G$-invariant. Let $\mathcal{A}$ be a randomized bounded-output $\mathrm{CSQ}(\tau)$ algorithm making at most $q$ queries, and suppose that $\mathcal{A}$ is $G$-equivariant in the following sense: for every $t \in L^2(\mathcal{X})$, every $g\in G$, and every $\mathrm{CSQ}(\tau)$ oracle strategy $\mathcal{O}_t$ for $\mathcal{D}_t$, there exists a $\mathrm{CSQ}(\tau)$ oracle strategy $g\cdot \mathcal{O}_t$ for $\mathcal{D}_{U_g t}$ such that
\begin{equation*}
    \mathcal{A}^{\mathcal{D}_{U_g t},\,g\cdot \mathcal{O}_t,\,R}
    \overset{d}{=}
    U_g\mathcal{A}^{\mathcal{D}_t,\mathcal{O}_t,R}.
\end{equation*}
Fix $h,h_* \in L^2(\mathcal{X})$, $\varepsilon>0$, and assume $\tau<\varepsilon/8$. Then
\begin{equation*}
    \inf_{\mathcal{O}_h}
    \mathbb{P}_{R}\!\left[
        \|\mathcal{A}^{D_h,\mathcal{O}_h,R}-h\|_{2}^2
        \leq
        \|h_*-h\|_{2}^2-\varepsilon
    \right]
    \leq
    \frac{q+1}{\tau^2} \mathrm{A}(h,G;h_*).
\end{equation*}
\end{thm}

\begin{proof}
First we establish the hypothesis class version of the statement:
\begin{equation*}
    p:=\mathbb{E}_{h\sim \mu_{\mathcal H}}
    \inf_{\mathcal{O}_h}
    \mathbb{P}_{R}\!\left[
        \|\mathcal{A}^{D_h,\mathcal{O}_h,R}-h\|_2^2
        \leq
        \|h_*-h\|_2^2-\varepsilon
    \right] \leq \frac{q+1}{\tau^2} \mathrm{A}(\mu_\mathcal{H}, h_*).
\end{equation*}
Assume that $p > 0$. By Lemma~\ref{lem:avg-randomized-learning-to-detection-csq}, there is a randomized bounded-output
$\mathrm{CSQ}(\tau)$ algorithm $\mathcal{B}$ making at most $q+1$ queries such that
\begin{equation*}
    \sup_{\mathcal{O}_{h_*}}
    \mathbb{P}_{R}\!\left[\mathcal{B}^{\mathcal{D}_{h_*},\mathcal{O}_{h_*},R}\text{ outputs }\mathrm{H}_1\right]
    =0 \quad \text{ and } \quad  \mathbb{E}_{h\sim \mu_{\mathcal H}}
    \inf_{\mathcal{O}_h}
    \mathbb{P}_{R}\!\left[\mathcal{B}^{\mathcal{D}_h,\mathcal{O}_h,R}\text{ outputs }\mathrm{H}_1\right]
    \geq p.
\end{equation*}
Applying Lemma~\ref{lem:distinguish_lb} to the algorithm $\mathcal{B}$ with the same null function $h_*$ and prior
$\mu_{\mathcal H}$ gives
\begin{equation}\label{eq:CSQ_hypothsis_bound}
    p
    \leq
    \frac{q+1}{\tau^2}\mathrm{A}(\mu_\mathcal{H}, h_*),
\end{equation}
as desired. Now we mimic the proof of Theorem~\ref{thm:NGD_equivariant}. 
Let $\nu_h$ denote the law of $U_g h$ when $g \sim \mu_G$, so that $\nu_h$ is a probability measure on the orbit
\begin{equation*}
    \mathcal{H}_h:=\{U_g h:g\in G\}\subseteq L^2(\mathcal{X}).
\end{equation*}
For each $t\in \mathcal{H}_h$, define
\begin{equation*}
    p(t)
    :=
    \inf_{\mathcal{O}_t}
    \mathbb{P}_{R}\!\left[
        \|\mathcal{A}^{\mathcal{D}_t,\mathcal{O}_t,R}-t\|_2^2
        \leq
        \|h_*-t\|_2^2-\varepsilon
    \right].
\end{equation*}

We claim that $p(U_g h)=p(h)$ for every $g\in G$. Fix $g\in G$. Let $\mathcal{O}_h$ be any $\tau$-compatible oracle strategy for $\mathcal{D}_h$, and let $g\cdot \mathcal{O}_h$ be the corresponding oracle strategy for $\mathcal{D}_{U_g h}$ given by equivariance, so
\begin{equation*}
    \mathcal{A}^{\mathcal{D}_{U_g h},\,g\cdot \mathcal{O}_h,\,R}
    \overset{d}{=}
    U_g\mathcal{A}^{\mathcal{D}_h,\mathcal{O}_h,R}.
\end{equation*}
Since $\mu_{\mathcal X}$ is $G$-invariant, $U_g$ acts unitarily on $L^2(\mathcal{X})$, and hence
\begin{equation*}
    \|U_g f-U_g t\|_2=\|f-t\|_2
    \qquad \forall f,t\in L^2(\mu_{\mathcal X}).
\end{equation*}
Because $h_*$ is $G$-invariant,
\begin{equation*}
    \|h_*-U_g h\|_2
    =
    \|U_g h_*-U_g h\|_2
    =
    \|h_*-h\|_2.
\end{equation*}
Therefore
\begin{align*}
    &\mathbb{P}_{R}\!\left[
        \|\mathcal{A}^{\mathcal{D}_{U_g h},\,g\cdot \mathcal{O}_h,\,R}-U_g h\|_2^2
        \leq
        \|h_*-U_g h\|_2^2-\varepsilon
    \right]  =
    \mathbb{P}_{R}\!\left[
        \|\mathcal{A}^{\mathcal{D}_h,\mathcal{O}_h,R}-h\|_2^2
        \leq
        \|h_*-h\|_2^2-\varepsilon
    \right].
\end{align*}
Taking the infimum over $\mathcal{O}_h$ gives $p(U_g h)\leq p(h)$. Applying the same argument with $g^{-1}$ in place of $g$ yields the reverse inequality, hence $p(U_g h)=p(h)$ for all $g \in G$ and averaging over $g\sim \mu_G$ we obtain
\begin{equation*}
    \mathbb{E}_{t\sim \nu_h}[p(t)] = p(h).
\end{equation*}

Now apply Equation~\ref{eq:CSQ_hypothsis_bound} to the hypothesis class $\mathcal{H}_h=\{U_g h:g\in G\}$ with prior $\nu_h$ and null function $h_*$. We find exactly
\begin{equation*}
    p(h)
    =
    \mathbb{E}_{t\sim \nu_h}[p(t)]
    \leq
    \frac{q+1}{\tau^2} \mathrm{A}(h,G;h_*),
\end{equation*}
which is the desired bound.
\end{proof}

\section{Alignment on roots of unity}\label{sec.alignment}
Recall that we will consider functions on 
\begin{equation*}
    \mathcal{X} = \prod_{i=1}^r \mu_{p_i}^{d_i}
\end{equation*}
with $\mu_\mathcal{X}$ the uniform measure. We must endow our space with a group action. We have quite a bit of freedom here with the general principal that larger groups will give us stronger bounds on the alignment $\mathrm{A}(f,G)$, but smaller groups will prove bounds for a less restricted class of learning algorithms. We begin with a large group, given by
\begin{equation*}
    G := \mathcal{X} \rtimes \Sigma := \prod_{i=1}^r \mu_{p_i}^{d_i} \rtimes \prod_{i=1}^r S_{d_i},
\end{equation*}
which will allow us to scramble each individual digit as well as the order of the digits. Namely, this group acts on each digit block by
\begin{equation*}
    (\zeta, \sigma)(x_1,\ldots,x_{d_i}) = (\zeta_1x_{\sigma(1)},\ldots,\zeta_{d_i}x_{\sigma(d_i)}).
\end{equation*}
First we prove that even this large action is respected by many learning algorithms. In particular, we have the following simple lemma. 
\begin{lem}\label{lem:orthogonal_action}
    The group $G = \mathcal{X} \rtimes \Sigma$ acts orthogonally on $\mathcal{X}$.
\end{lem}
\begin{proof}
    Decompose $\mathbb{C}^d = \bigoplus_{i=1}^r \mathbb{C}^{d_i}$. Observe that each $\mu_{p_i}^{d_i}$ gives a unitary action on $\mathbb{C}^{d_i}$ via component-wise multiplication, and each $S_{d_i}$ gives a unitary action on $\mathbb{C}^{d_i}$ by permuting basis vectors. In particular, the composite action
    \begin{equation*}
        (\zeta, \sigma)(x_1,\ldots,x_{d_i}) = (\zeta_1x_{\sigma(1)},\ldots,\zeta_{d_i}x_{\sigma(d_i)})
    \end{equation*}
    is unitary, and so we have $\mu_{p_i}^{d_i} \rtimes S_{d_i} \hookrightarrow U(d_i)$. Therefore we have
    \begin{equation*}
        G \hookrightarrow \bigoplus_{i=1}^r \mathrm{U}(d_i) \subseteq \mathrm{U}(d) \hookrightarrow \mathrm{SO}(2d)
    \end{equation*}
    by identifying $\mathbb{C}^d$ with $\mathbb{R}^{2d}$.
\end{proof}

As a corollary, by Proposition~\ref{prop.G_inv_kernel} and Proposition~\ref{prop:G_inv_NGD_condition}, standard machine learning algorithm such as SVM with Gaussian kernel and FCNN with Gaussian initialization are $\mathcal{X} \times \Sigma$-equivariant. Furthermore, since this contains the action of $\mathcal{X}$ on itself, we can compute for an arbitrary $f \in L^2(\mu_\mathcal{X})$ the alignment $\mathrm{A}(f,G)$ via a spectral decomposition.
\begin{defn}
    Suppose that $a \in \widehat{\cal{X}}$ and define the type of $a$ to be the tuple
    \begin{equation*}
        \mathrm{T}(a) = (m_{i,t})_{\substack{i \in [r] \\ t \in \mathbb{F}_{p_i}}},
    \end{equation*}
    where
    \begin{equation*}
        m_{i,t} = \#\{0 \leq j < d_i : a_{i,j} = t\}.
    \end{equation*}
\end{defn}

\begin{lem}\label{lem.G_alignemnt_equality}
    Let $f \in L^2(\mu_\mathcal{X})$ and let $G = \mathcal{X} \rtimes \Sigma$ as above. Then
    \begin{equation*}
        \mathrm{A}(f,G) = \max_{(m_{i,t})} \left(\prod_{i=1}^r \binom{d_i}{  m_{i,t} : t \in \mathbb{F}_{p_i}}\right)^{-1} \sum_{\substack{a \in \widehat{\cal{X}}\\\mathrm{T}(a) = (m_{i,t})}} |\widehat{f}(a)|^2,
    \end{equation*}
    with the max taken over all types in $\widehat{\cal{X}}$.
\end{lem}

\begin{proof}
    Let $\zeta \sim \cal{X}$ and $\sigma \sim \Sigma$ so that $(\zeta,\sigma) \sim G$, and let $x,x' \sim \cal{X}$ be independent. Then for any any $h : \cal{X} \to \mathbb{R}$, we have
    \begin{align*}
        \mathbb{E}_g[\mathbb{E}_x[f(g(x))h(x)]^2]  &= \mathbb{E}_{\zeta,\sigma}[\mathbb{E}_x[f((\zeta, \sigma)(x))h(x)]^2] \\
        &= \mathbb{E}_{\zeta,\sigma,x,x'}[f((\zeta, \sigma)(x))f((\zeta, \sigma)(x'))h(x)h(x')].
    \end{align*}
    Fourier expanding each $f$ and using that $\chi$ are multiplicative yeilds
    \begin{equation*}
        \mathbb{E}_g[\mathbb{E}_x[f(g(x))h(x)]^2] =\mathbb{E}_{\sigma,x,x'}\left[\sum_{a,a' \in \widehat{\cal{X}}}\hat{f}(a)\hat{f}(a')h(x)h(x')\chi_a(\sigma(x))\chi_{a'}(\sigma(x'))\right]\mathbb{E}_{\zeta}[\chi_a(\zeta)\chi_{a'}(\zeta)].
    \end{equation*}
    Observe that orthogonality implies $\mathbb{E}_{\zeta}[\chi_a(\zeta)\chi_{a'}(\zeta)] = \mathbbm{1}_{a + a' = 0}$, and since $f$ is real valued we have $\hat{f}(-a) = \overline{\hat{f}(a)}$
    , so we get
    \begin{equation*}
        \mathbb{E}_g[\mathbb{E}_x[f(g(x))h(x)]^2]  =\mathbb{E}_{\sigma,x,x'}\left[\sum_{a \in \widehat{\cal{X}}}|\hat{f}(a)|^2h(x)h(x')\chi_a(\sigma(x))\chi_{-a}(\sigma(x'))\right].
    \end{equation*}
    Use independence, $\chi_{-a}(x) = \overline{\chi_a(x)}$, and $h$ real valued to get
    \begin{align*}
        \mathbb{E}_{x,x'}[ h(x)h(x')\chi_a(\sigma(x))\chi_{-a}(\sigma(x')) ]
        &=\mathbb{E}_{x,x'}[ h(x)\overline{h(x')}\chi_a(\sigma(x))\overline{\chi_{a}(\sigma(x'))}] \\
        &= \mathbb{E}_x[h(x)\chi_a(\sigma(x))]\overline{\mathbb{E}_{x'} [h(x')\chi_a(\sigma(x'))]} \\
        &= |\mathbb{E}_x[h(x)\chi_a(\sigma(x))]|^2 \\
        &= |\hat{h}(\sigma^{-1}(a))|^2,
    \end{align*}
    so we have
    \begin{equation*}
        \mathbb{E}_g[\mathbb{E}_x[f(g(x))h(x)]^2] = \sum_{a \in \widehat{\cal{X}}} |\hat{f}(a)|^2 \mathbb{E}_\sigma |\hat{h}(\sigma^{-1}(a))|^2.
    \end{equation*}
    Now we observe that the orbits under $\Sigma$ are exactly
    \begin{equation*}
        \Sigma \cdot a = \{b \in \widehat{\cal{X}} : \mathrm{T}(b)=\mathrm{T}(a)\}
    \end{equation*}
    since permuting coordinates does not change the type. Since the Haar measure on $\Sigma$ is uniform, we have
    \begin{equation*}
        \mathbb{E}_\sigma|\hat{h}(\sigma^{-1}(a))|^2 = \frac{1}{\# (\Sigma \cdot a)} \sum_{b \in \Sigma \cdot a} |\hat{h}(b)|^2,
    \end{equation*}
    so we conclude that
    \begin{equation*}
        \mathbb{E}_g[\mathbb{E}_x[f(g(x))h(x)]^2] = \sum_{a \in \widehat{\cal{X}}} |\hat{f}(a)|^2 \frac{1}{\# (\Sigma \cdot a)}  \sum_{\mathrm{T}(b)=\mathrm{T}(a)} |\hat{h}(b)|^2.
    \end{equation*}
    By orbit-stabilizer, we have
    \begin{equation*}
        \# (\Sigma \cdot a) = \prod_{i=1}^r \frac{\# S_{d_i}}{\# \mathrm{Stab}_{S_{d_i}}(a^{(i)})} =\prod_{i=1}^r \frac{d_i!}{\prod_t m_{i,t}(a)!} = \prod_{i=1}^r \binom{d_i}{m_{i,0},\ldots,m_{i,p_i-1}},
    \end{equation*}
    since the action is independent across blocks and preserves the type. Therefore we have
    \begin{equation*}
        \mathbb{E}_g[\mathbb{E}_x[f(g(x))h(x)]^2] = \sum_{a \in \widehat{\cal{X}}} |\hat{f}(a)|^2 \left(\prod_{i=1}^r \binom{d_i}{m_{i,0},\ldots,m_{i,p_i-1}}\right)^{-1}  \sum_{\mathrm{T}(b)=\mathrm{T}(a)} |\hat{h}(b)|^2.
    \end{equation*}
    Noting that Plancheral implies that for every $a \in \widehat{\cal{X}}$,
    \begin{equation*}
        \sum_{\mathrm{T}(b)=\mathrm{T}(a)} |\hat{h}(b)|^2 \leq \norm{h}_2^2,
    \end{equation*}
    the supremum over all $h \in L^2(\cal{X})$ with $\norm{h}_2 = 1$  is obtained by taking $h = \chi_a$ or $h= \frac{1}{2}(\chi_a + \chi_{-a})$ for some $a$.
\end{proof}

With the group $G = \mathcal{X} \rtimes \Sigma$, to obtain polynomial bounds on learning it suffices to prove a uniform superpolynomial bound on the Fourier spectrum near the tails in the following sense.

\begin{lem}
    For any $\beta > 0$, there exists an explicit $D_{\beta, r} > 0$ such that if $d > D_{\beta,r}$ and $f \in L^2(\mathcal{X})$ satisfies $\norm{f}_2 \leq 1$ and $\mathrm{A}(f,G) > d^{-\beta}$, then there exists $a \in \widehat{\mathcal{X}}$ such that for some $i \in [r]$ and $t \in \mathbb{F}_{p_i}$, we have
    \begin{equation*}
        m_{i,t}(a) \leq \lfloor \beta \rfloor, \quad \text{ or } \quad m_{i,t}(a) \geq d_{i} - \lfloor \beta \rfloor
    \end{equation*}
    with $|\widehat{f}(a)| > d^{-\beta/2}$.
\end{lem}

\begin{proof}
Let $K=\lfloor \beta \rfloor$, and $k=K+1$. We claim we can take $D_{\beta,r} = \left\lceil (2rk)^{k/(k-\beta)} \right\rceil$. For a type $(m_{i,t})$, write
\begin{equation*}
    N(m_{i,t})
    :=
    \prod_{i=1}^r
    \binom{d_i}{m_{i,t}:t\in\mathbb F_{p_i}} .
\end{equation*}
By Lemma~\ref{lem.G_alignemnt_equality},
\begin{equation*}
    \mathrm A(f,G)
    =
    \max_{(m_{i,t})}
    \frac{1}{N(m_{i,t})}
    \sum_{\substack{a\in \widehat{\mathcal X}_{\mathbf d}\\
    \mathrm T(a)=(m_{i,t})}}
    |\widehat f(a)|^2 .
\end{equation*}
Let $(m_{i,t})$ be a type attaining this maximum, and suppose that $(m_{i,t})$ was not in the tail. Then for every $i\in [r]$ and every $t \in \mathbb{F}_{p_i}$,
\begin{equation}\label{eq:type_bound}
    k \leq m_{i,t}\leq d_i- k.
\end{equation}
Let $i_0$ be such that $d_{i_0}=\max_i d_i$, so $d_{i_0} \geq d/r$. Choose any $t_0\in\mathbb F_{p_{i_0}}$. Since $(m_{i,t})$ is balanced,
\begin{equation*}
    k\leq m_{i_0,t_0}\leq d_{i_0}-k.
\end{equation*}
Therefore
\begin{equation*}
      N(m_{i,t}) \geq \binom{d_{i_0}}{ m_{i_0,t}:t\in\mathbb F_{p_{i_0}}}
    \geq
    \binom{d_{i_0}}{m_{i_0,t_0}}
    \geq
    \binom{d_{i_0}}{k} \geq
    \left(\frac{d_{i_0}}{2k}\right)^k
    \geq
    \left(\frac{d}{2rk}\right)^k .
\end{equation*}
Since $d_{i_0} \geq 2k$ follows from Equation~\eqref{eq:type_bound}. Therefore, by Plancherel and $\norm{f}_2\leq 1$,
\begin{align*}
    \frac{1}{N(m_{i,t})}
    \sum_{\substack{a\in \widehat{\mathcal X}_{\mathbf d}\\
    \mathrm T(a)=(m_{i,t})}}
    |\widehat f(a)|^2
    \leq
    \frac{\norm{f}_2^2}{N(m_{i,t})}\leq
    (2rk)^k d^{-k}.
\end{align*}
Since $d>D_{\beta,r} = \left\lceil (2rk)^{k/(k-\beta)} \right\rceil$, we have $(2rk)^k d^{-k}<d^{-\beta}$. Thus any balanced type contributes strictly less than $d^{-\beta}$ to the alignment. This contradicts the assumption
\begin{equation*}
    \mathrm A(f,G)>d^{-\beta}.
\end{equation*}
Therefore the maximizing type $(m_{i,t})$ is a tail type. It follows easily that atleast one $a \in \widehat{\mathcal{X}}$ with $\mathrm{T}(a) = (m_{i,t})$ satisfies $|\widehat{f}(a)|^2 > d^{-\beta}$, and this proves the claim.
\end{proof}

In the direction of smaller groups, we can also quite easily consider any subgroup $G \leq \mathcal{X}$. The proof is nearly identical to the $G = \mathcal{X} \rtimes \Sigma$ case, and it also follows from Lemma~\ref{lem:orthogonal_action} that $G \leq \mathcal{X}$ also act orthogonally.

\begin{defn}
Let $G \leq \mathcal{X}_d$ be a subgroup, and define its annihilator by
\begin{equation*}
    \mathrm{Ann}(G) := \{a \in \widehat{\mathcal{X}_d} : \chi_a(g)=1 \text{ for all } g \in G\}.
\end{equation*}
\end{defn}

\begin{lem}
    Let $f \in L^2(\mu_\mathcal{X})$ and let $G \leq \mathcal{X}$ as above. Then
    \begin{equation*}
        \mathrm{A}(f,G)
        :=
        \sup_{\norm{\phi}_2=1} \mathbb{E}_{g \sim G}[\langle U_g f, \phi \rangle^2]
        =
        \max_{C \in \widehat{\mathcal{X}_d}/\mathrm{Ann}(G)}
        \sum_{a \in C} |\widehat{f}(a)|^2.
    \end{equation*}
\end{lem}

\begin{proof}
    Let $x,x' \sim \mathcal{X}_d$ be independent, and let $g \sim G$. For any real-valued
    $\phi : \mathcal{X}_d \to \mathbb{R}$, by following the previous proof we obtain
    \begin{align*}
        \mathbb{E}_{g}[\langle U_g f,\phi\rangle^2]
        &=
        \mathbb{E}_{g}\left[\left(\mathbb{E}_x[f(gx)\phi(x)]\right)^2\right] \\
        &=
        \mathbb{E}_{g,x,x'}[f(gx)f(gx')\phi(x)\phi(x')] \\
        &=         \mathbb{E}_{g,x,x'}
        \left[
            \sum_{a,a' \in \widehat{\mathcal{X}_d}}
            \widehat{f}(a)\widehat{f}(a')
            \chi_a(g)\chi_{a'}(g)\chi_a(x)\chi_{a'}(x')
            \phi(x)\phi(x')
        \right]\\
        &=
        \mathbb{E}_{x,x'}
        \Bigg[
            \sum_{\substack{a,a' \in \widehat{\mathcal{X}_d}\\ a+a' \in \mathrm{Ann}(G)}}
            \widehat{f}(a)\widehat{f}(a')
            \chi_a(x)\chi_{a'}(x')
            \phi(x)\phi(x')
        \Bigg] \\
        &= 
        \sum_{\substack{a,a' \in \widehat{\mathcal{X}_d}\\ a+a' \in \mathrm{Ann}(G)}}
        \widehat{f}(a)\widehat{f}(a')
        \overline{\widehat{\phi}(a)} \, \overline{\widehat{\phi}(a')}.
    \end{align*}
    For each coset $C \in \widehat{\mathcal{X}_d}/\mathrm{Ann}(G)$, define
    \begin{equation*}
        m_C(f) := \sum_{a \in C} |\widehat{f}(a)|^2,
        \qquad
        s_C(\phi) := \sum_{a \in C} |\widehat{\phi}(a)|^2,
        \qquad
        \beta_C := \sum_{a \in C} \widehat{f}(a)\overline{\widehat{\phi}(a)}.
    \end{equation*}
    Since the condition $a+a' \in \mathrm{Ann}(G)$ is equivalent to saying that $a'$ lies in the coset
    $-C$ whenever $a \in C$, we may rewrite the previous expression as
    \begin{equation*}
        \mathbb{E}_{g}[\langle U_g f,\phi\rangle^2]
        =
        \sum_{C \in \widehat{\mathcal{X}_d}/\mathrm{Ann}(G)} \beta_C \beta_{-C}.
    \end{equation*}
    Since $f$ and $\phi$ are real-valued, we have $ \widehat{f}(-a)=\overline{\widehat{f}(a)}$ and $\widehat{\phi}(-a)=\overline{\widehat{\phi}(a)}$, so we may wrtie
    \begin{equation*}
        \beta_{-C}
        =
        \sum_{a \in C} \hat{f}(-a)\overline{\widehat{\phi}(-a)}
        =
        \sum_{a \in C} \overline{\widehat{f}(a)}\widehat{\phi}(a)
        =
        \overline{\beta_C}.
    \end{equation*}
    Therefore
    \begin{equation*}
        \mathbb{E}_{g}[\langle U_g f,\phi\rangle^2]
        =
        \sum_{C \in \widehat{\mathcal{X}_d}/\mathrm{Ann}(G)} |\beta_C|^2,
    \end{equation*}
    and by Cauchy-Schwarz,
    \begin{equation*}
        |\beta_C|^2
        \leq
        \left(\sum_{a \in C} |\widehat{f}(a)|^2\right)
        \left(\sum_{a \in C} |\widehat{\phi}(a)|^2\right)
        =
        m_C(f)s_C(\phi).
    \end{equation*}
    Thus
    \begin{align*}
        \mathbb{E}_{g}[|\langle U_g f,\phi\rangle|^2]
        \leq
        \sum_{C \in \widehat{\mathcal{X}_d}/\mathrm{Ann}(G)} m_C(f)s_C(\phi) \leq
        \left(\max_{C \in \widehat{\mathcal{X}_d}/\mathrm{Ann}(G)} m_C(f)\right)
        \sum_{C \in \widehat{\mathcal{X}_d}/\mathrm{Ann}(G)} s_C(\phi).
    \end{align*}
    By Plancherel,
    \begin{equation*}
        \sum_{C \in \widehat{\mathcal{X}_d}/\mathrm{Ann}(G)} s_C(\phi)
        =
        \sum_{a \in \widehat{\mathcal{X}_d}} |\widehat{\phi}(a)|^2
        =
        \norm{\phi}_2^2
        =
        1.
    \end{equation*}
    Hence
    \begin{equation*}
        \mathbb{E}_{g}[\langle U_g f,\phi\rangle^2]
        \leq
        \max_{C \in \widehat{\mathcal{X}_d}/\mathrm{Ann}(G)}
        \sum_{a \in C} |\widehat{f}(a)|^2.
    \end{equation*}
    Taking the supremum over all $\phi$ with $\norm{\phi}_2=1$ gives the upper bound.

    For the reverse inequality, choose a coset $C_* \in \widehat{\mathcal{X}_d}/\mathrm{Ann}(G)$ attaining the maximum, and set $ M := \sum_{a \in C_*} |\hat{f}(a)|^2$. Since $f$ is real-valued, we also have $\sum_{a \in -C_*} |\hat{f}(a)|^2 = M$. If $C_*=-C_*$, define $\phi$ by
    \begin{equation*}
        \widehat{\phi}(a)
        :=
        \begin{cases}
            \hat{f}(a)/\sqrt{M}, & a \in C_*, \\
            0, & a \notin C_*.
        \end{cases}
    \end{equation*}
    Then $\phi$ is real-valued, $\norm{\phi}_2=1$, and
    \begin{equation*}
        \beta_{C_*}
        =
        \sum_{a \in C_*} \widehat{f}(a)\overline{\widehat{\phi}(a)}
        =
        \sqrt{M}.
    \end{equation*}
    Since all other $\beta_C$ vanish, we get $\mathbb{E}_{g}[|\langle U_g f,\phi\rangle|^2] = M$. If $C_* \neq -C_*$, define $\phi$ by
    \begin{equation*}
        \widehat{\phi}(a)
        :=
        \begin{cases}
            \widehat{f}(a)/\sqrt{2M}, & a \in C_* \cup (-C_*), \\
            0, & a \notin C_* \cup (-C_*).
        \end{cases}
    \end{equation*}
    and verify that we obtain equality. Thus
    \begin{equation*}
        \mathrm{A}(f,G)
        =
        \max_{C \in \widehat{\mathcal{X}_d}/\mathrm{Ann}(G)}
        \sum_{a \in C} |\widehat{f}(a)|^2.
    \end{equation*}
\end{proof}

An immediate corollary obtain by taking $G = \mathcal{X}$ is 
\begin{cor}\label{cor:alignment_X_d}
    For any $f \in L^2(\mu_\mathcal{X})$ and $G = \mathcal{X}$ we have
    \begin{equation*}
        \mathrm{A}(f,G) = \max_{a \in \widehat{\mathcal{X}}} |\widehat{f}(a)|^2.
    \end{equation*}
\end{cor}

\section{Harmonic analysis on roots of unity}\label{sec.harmonic}
In this section, we consider a general $\cal{X}$ of the form in Equation~\eqref{eq:general_space}, and we prove several lemmas related to harmonic analysis on $\cal{X}$. First, we relate a lower bound on Fourier coefficients for $a \in \widehat{\cal{X}}$ given by
\begin{equation*}
    \widehat{f}(a) = \frac{1}{X}\sum_{x \in \cal{X}} f(x)\overline{\chi_a(x)},
\end{equation*}
to a lower bound on the Fourier coefficients for $\theta \in \mathbb{R}/\mathbb{Z}$ given by
\begin{equation*}
    \widehat{f}(\theta) = \frac{1}{X} \sum_{x < X} f(x) e(-\theta x).
\end{equation*}
This is a generalization of \cite[Proposition 2]{green2012notcomputingmobiusfunction}, which is itself a generalization of \cite{Katai1986}. Second, we will give various $\ell^1$ and $\ell^\infty$ bounds on the discrete Fourier transforms of the $\chi_a$ themselves, given by
\begin{equation*}
    \widehat{\chi}_a(k) = \frac1X\sum_{x < X}\chi_a(x)e\left(\frac{-kx}{X}\right).
\end{equation*}
These bounds will be generalizations of the various Lemmas in \cite{Bourgain2013}. We will use the argument $a \in \widehat{\cal{X}}, \theta \in \mathbb{R}/\mathbb{Z}$, or $k \in \mathbb{Z}/X\mathbb{Z}$ to make it clear which Fourier transform we are referring to.

\subsection{An argument of Katai}\label{sec.Katai}
The key input to the low-weight case is that a large Fourier coefficient $\widehat{f}(a)$ implies a large Fourier coefficient $\widehat{f}(\theta)$. This is the content of the following proposition. 
\begin{prop}\label{prop.Katai_arg}
    Let $f : \Omega  \to [-1,1]$ be a function for which there exists some $a \in \widehat{\cal{X}}$ such that $|\hat{f}(a)| > \delta$ for some $0< \delta < 1/2$. Then there exists $\theta \in [0,1]$ such that
    \begin{equation*}
        |\hat{f}(\theta)| \geq \left(\frac{\delta}{10|a|\sqrt{p_r}}\right)^{4|a|}.
    \end{equation*}
\end{prop}

\begin{proof}
    First we write
    \begin{equation}\label{eq.psi_expansion}
        \widehat{f}(a) = \frac{1}{|\cal{X}|}\sum_{x \in \cal{X}} f(x) \overline{\chi_a(x)} = \frac{1}{|\cal{X}|} \sum_{x < X} f(x) \prod_{i=1}^r \prod_{j<d_i} \psi_{i,j}\left(\frac{x}{p_i^{j+1}}\right),
    \end{equation}
    where $\psi_{i,j}: \mathbb{R}/\mathbb{Z} \to \mu_{p_i}$ is the function
    \begin{equation*}
        \psi_{i,j}(t) = \overline{e_{p_i}(u)^{a_{i,j}}}
    \end{equation*}
    for $t \in [\frac{u}{p_i}, \frac{u+1}{p_i})$ and $0 \leq u < p_i$. 
    
    We approximate $\psi_{i,j}$ for $a_{i,j} \neq  0$ by a smoothed version $\widetilde{\psi}_{i,j} : \mathbb{R}/\mathbb{Z} \to \mathbb{C}$ with the property
    \begin{equation}\label{eq.smoothpsi}
        \mathbb{E}_{x \in \Omega}\left[\left|\psi_{i,j}\left(\frac{x}{p_i^{j+1}}\right) - \widetilde{\psi}_{i,j}\left(\frac{x}{p_i^{j+1}}\right)\right|\right] < \varepsilon
    \end{equation}
    for all $i,j$, and traditional Fourier shape
    \begin{equation*}
         \widetilde{\psi}_{i,j}(t) = \sum_{|r| \leq 200p_r^2\varepsilon^{-3}} \alpha_{i,r}e(rt)
    \end{equation*}
    with $|\alpha_{i,r}| \leq 1$ for all $r$.

    We use a standard smoothing argument. To ease notation, we fix $i,j$, let $p = p_i$ and $\psi = \psi_{i,j}$. Now we consider
    \begin{equation*}
        \psi_0 := \phi * \chi * \chi
    \end{equation*}
    where $\phi : \mathbb{R}/\mathbb{Z} \to \mathbb{C}$ is defined by $\phi(t) = \psi(t+ \frac{\varepsilon}{12p})$ and $\chi(t) = \frac{12p}{\varepsilon}\mathbbm{1}_{[-\varepsilon/24p, \varepsilon/24p]}(t)$. Note that $|\psi_0(t)| \leq 1$ and $\psi_0(t) = \psi(t)$ for $t$ outside of the intervals
    \begin{equation*}
        I_1 = \left[\frac{1}{p} - \frac{\varepsilon}{6p},\frac{1}{p}\right],  I_2 = \left[\frac{2}{p} - \frac{\varepsilon}{6p},\frac{2}{p}\right],  \ldots,  I_p = \left[1- \frac{\varepsilon}{6p}, 1\right]
    \end{equation*}
    Each interval $I_i$ contains at most $\varepsilon/6p$ of the rationals with denominator $p^j$ for any $j$. It follows that
    \begin{equation}\label{eq.psitopsi0}
         \mathbb{E}_{x \in \Omega}\left[\left|\psi\left(\frac{x}{p^{j+1}}\right) - \psi_0\left(\frac{x}{p^{j+1}}\right)\right|\right] < \frac{\varepsilon}{3}.
    \end{equation}

    The Fourier coefficients of $\psi_0$ are given by $\hat{\psi}_0(u) = \hat{\phi}(u)\hat{\chi}(u)^2$, and in particular we have the bound $|\hat{\psi}_0(u)| \leq |\hat{\chi}(u)|^2 \leq \min(1,12p/\varepsilon\pi|u|)^2$. It follows that for $U = 200p_r^2/\varepsilon^3$ we have 
    \begin{equation*}
        \sum_{|u| > U} | \hat{\psi}_0(u)| \leq \left(\frac{12p}{\varepsilon \pi}\right)^2 \sum_{|u| > U} \frac{1}{u^2} \leq \left(\frac{12p}{\varepsilon \pi}\right)^2  \frac{1}{U} <
        \frac{\varepsilon}{3}.
    \end{equation*}
    Now we define
    \begin{equation*}
        \psi_1(t) := \sum_{|u| <U} \hat{\psi}_0(u)e(ut),
    \end{equation*}
    which clearly has the desired Fourier shape. The preceding analysis shows
    \begin{equation}\label{eq.psi0topsi1}
        \norm{\psi_1-\psi_0}_\infty  \leq \sum_{|u| > U} |\hat{\psi}_0(u)| < \varepsilon/3.
    \end{equation}
    This also implies that $\norm{\psi_1}_\infty \leq \norm{\psi_0}_\infty +\frac{\varepsilon}{3} \leq 1 + \frac{\varepsilon}{3}$. Finally, we define
    \begin{equation*}
        \widetilde{\psi}(t) := \left(1 + \frac{\varepsilon}{3}\right)^{-1}\psi_1(t)
    \end{equation*}
    so that $| \widetilde{\psi}(t)| \leq 1$. Notice that $\widetilde{\psi}(t)$ has the correct Fourier shape, and we also have
    \begin{equation}\label{eq.psi1topsitilde}
        \norm{\widetilde{\psi} - \psi_1}_\infty = \norm{\psi_1}_\infty \left(1 - \frac{1}{1+\varepsilon/3}\right) \leq \frac{\varepsilon}{3},
    \end{equation}
    so combining Equations~\eqref{eq.psitopsi0}, \eqref{eq.psi0topsi1}, and \eqref{eq.psi1topsitilde} yields the smooth approximation in Equation~\eqref{eq.smoothpsi}.

    Now we insert the smooth approximation into Equation~\eqref{eq.psi_expansion} and apply Equation~\eqref{eq.smoothpsi} repeatedly. Fix some $(i',j')$ with $a_{i',j'} \neq 0$ that we replace and observe that we pick up an error term of size at most
    \begin{equation*}
       \mathbb{E}_{x \in \Omega}\left[  \left|f(x)\prod_{(i,j) \neq (i',j')} \psi_{i,j}\left(\frac{x}{p_i^{j+1}}\right)\right| \left|\widetilde{\psi}_{i',j'
        }\left(\frac{x}{p_{i'}^{j'+1}}\right) - \psi_{i',j'
        }\left(\frac{x}{p_{i'}^{j'+1}}\right)\right|\right] < \varepsilon,
    \end{equation*}
    since all of $f,\psi,\widetilde{\psi}$ are 1-bounded and we have Equation~\eqref{eq.smoothpsi}. Choosing $\varepsilon = \frac{\delta}{2|a|}$ yields
    \begin{equation}\label{eq.expectation_psi_tilde_lb}
        \left|\mathbb{E}_{x \in \Omega}\left[ f(x) \prod_{i=1}^r \prod_{j<d_i} \widetilde{\psi}_i\left(\frac{x}{p_i^{j+1}}\right)^{a_{i,j}}\right]\right| \geq \delta - \varepsilon |a| \geq \frac{\delta}{2}.
    \end{equation}

    Now we expand each $\widetilde{\psi}_{i,j}$ into its Fourier series. For each $(i,j)$ with $a_{i,j} \neq 0$ we have
    \begin{equation*}
         \widetilde{\psi}_{i,j}(t)
        = \sum_{|u|\le U} \alpha_{i,j,u}\,e(ut),
        \qquad |\alpha_{i,j,u}|\le 1,
    \end{equation*}
    so in particular
    \begin{equation*}
        \sum_u |\alpha_{i,j,u}|
        \le \sum_{|u|\le U} 1
        \le 3U.
    \end{equation*}
    For indices with $a_{i,j}=0$ we have $\psi_{i,j}= 1$ and we do not smooth, so they play no role in what follows. Using this expansion, we can write
    \begin{equation*}
         \prod_{i=1}^r \prod_{j<d_i} \widetilde{\psi}_{i,j}\!\left(\frac{x}{p_i^{j+1}}\right)
        = \prod_{\substack{1\le i\le r\\ j< d_i\\ a_{i,j}>0}}
          \widetilde{\psi}_{i,j}\!\left(\frac{x}{p_i^{j+1}}\right)
        = \sum_{\theta} c_\theta\, e(x\theta),
    \end{equation*}
    where the coefficients $c_\theta$ are obtained from the finitely supported families
    $\{\alpha_{i,j,u}\}_u$ by repeated convolution. By Young's convolution inequality we have
    \begin{equation}\label{eq.c_theta_UB}
        \sum_{\theta} |c_\theta|
        \;\le\;
        \prod_{\substack{i,j:\\ a_{i,j}>0}}
        \left(\sum_u |\alpha_{i,j,u}|\right)
        \;\le\;
        (3U)^{A}.
    \end{equation}
    
    Combining \eqref{eq.expectation_psi_tilde_lb} with \eqref{eq.c_theta_UB}, we obtain
    \begin{equation*}
        \frac{\delta}{2}
        \le
        \left|\mathbb{E}_{n \in \Omega}
          \left[ f(n) \prod_{i=1}^r \prod_{j<d_i}
          \widetilde{\psi}_{i,j}\!\left(\frac{n}{p_i^{j+1}}\right)\right]\right|
        = \left|\sum_\theta c_\theta \,\hat{f}(\theta)\right|
        \le \max_\theta |\hat{f}(\theta)| (3U)^A.
    \end{equation*}
    It follows that there exists some $\theta$ such that
    \begin{equation*}
        |\hat{f}(\theta)| \geq  \frac{\delta}{2}(3U)^{-|a|}\geq \frac{\delta}{2}\left(\frac{\varepsilon^3}{600 p_r^2}\right)^{|a|} > \left(\frac{\delta}{10|a|\sqrt{p_r}}\right)^{4|a|}.
    \end{equation*}
\end{proof}

\begin{rem}\label{rem:sparse_rational}
    In fact, the proof above shows that we can take $\theta \in \bb{Q}$ of the form
    \begin{equation*}
        \theta = \sum_{i=1}^{r} \sum_{j<d_i} \frac{s_{i,j}}{p_i^{j+1}},
    \end{equation*}
    where $s_{i,j} \neq 0$ for at most $|a|$ terms and $|s_{i,j}| \leq (40p_r)^2|a|^3/\delta^3$. We will take advantage of this fact in the $r=1$ case to resolve the potential contribution of Siegel zeros and to make our results unconditional. To this end, we will be able to estimate the major arc contribution with the help of the following lemma:
    
\begin{lem} \label{lem:sparse_padic_major_arc}
        Let $X = p^d$, $Q\geq 1$, and suppose that
        \begin{equation*}
             \theta = \frac{s_1}{p^{i_1}}  + \cdots + \frac{s_k}{p^{i_k}}, \qquad 0<i_1<\cdots<i_k\leq d,  \qquad  |s_j|\leq Q.
        \end{equation*}
        Further assume that $|\theta - a/q| \leq Q/(qX)$ with $q \leq Q$ and $(a,q)=1$, and that $p^{d/2k} > 8pQ^2$. Then $q$ is a power of $p$.
    \end{lem}
\end{rem}

\begin{proof}
Set $i_0 = 0$ and $i_{k+1} = d$. By the pigeonhole principle, there is an index $j \in \{0,\ldots, k\}$ such that $|i_j-i_{j+1}| \geq d/2k$. Now we note that $\theta$ is close to a rational with denominator $q'=p^{i_j}$. Precisely, write $a' = s_1p^{i_j-i_1} + \cdots +s_{j-1}p^{i_j-i_{j-1}} +s_j$, so then
\begin{equation*}
    \left|\theta - \frac{a'}{q'}\right| \leq Q(p^{-i_{j+1}} + p^{-i_{j+2}}+\cdots) \leq p^{-d/2k} \frac{pQ}{q'}
\end{equation*}
Note also that $q' \leq p^{-d/2k}X$ and since we assume $p^{d/2k} > 4Q^2$, we have $Q/X < 1/(4Qq')$. Thus
\begin{equation*}
    \left|\frac{a}{q} - \frac{a'}{q'}\right| \leq \left|\theta - \frac{a'}{q' }\right| + \left|\theta -\frac{a}{q}\right| \leq p^{-d/2k}\frac{pQ}{q'}+\frac{Q}{X} < \frac{1}{Qq'} < \frac{1}{qq'}
\end{equation*}
which implies $a/q = a'/q'$. Since $(a,q) = 1$ and $q \mid aq'$ we conclude that $q$ is a power of $p$.
\end{proof}

\subsection{Fourier coefficients of digital characters}\label{sec.character_Fourier}
We begin this section by observing the product structure of $\widehat{\chi}_a(k)$ and $\norm{\widehat{\chi}_a(k)}_1$. Notice that 
\begin{equation*}
    \chi_a(x) = \prod_{i=1}^r \chi^{(i)}_{a_i}(x_i),
\end{equation*}
where $\chi^{(i)}_{a_i}(x_i)$ only depends on $x_i \equiv x \bmod p_i^{d_i}$. Likewise, we have
\begin{equation}\label{eq.CRT_factorization}
    e\left(\frac{kx}{X}\right) = \prod_{i=1}^r e\left(\frac{k_it_ix_i}{p_i^{d_i}}\right),
\end{equation}
where $k_i \equiv k \bmod p_i^{d_i}$ and $t_i \equiv (X/p_i^{d_i})^{-1} \bmod p_i^{d_i}$ and the inverse is taken modulo $p_i^{d_i}$. The Chinese remainder theorem, together with the fact that the summand factors, implies that
\begin{equation}\label{eq.ell_infty_mult}
     \widehat{\chi}_a(k) = \frac1X\sum_{x < X}\chi_a(x)e\left(\frac{-kx}{X}\right) = \prod_{i=1}^r \frac{1}{p_i^{d_i}}\sum_{x_i < p_i^{d_i}} \chi_{a_i}^{(i)}(x_i)e\left(\frac{-k_it_ix_i}{p_i^{d_i}}\right) = \prod_{i=1}^r \widehat{\chi}_{a_i}^{(i)}(k_i'),
\end{equation}
where $k_i' = k_it_i \bmod p_i^{d_i}$. Likewise, we have
\begin{equation}\label{eq.ell_1_mult}
    \norm{\widehat{\chi}_a}_1 = \sum_{k < X} |\widehat{\chi}_a(k)| = \prod_{i=1}^r \sum_{k_i' < p_i^{d_i}} |\widehat{\chi}_{a_i}^{(i)}(k_i')| = \prod_{i=1}^r \norm{\widehat{\chi}_{a_i}^{(i)}}_1,
\end{equation}
so $\ell^1$ and $\ell^\infty$ norms reduce to norms of local factors. Now we study local factors.

If we fix a prime $p$ and an $a \in (\mathbb{Z}/p\mathbb{Z})^d$, we observe that $x = \sum_{j < d} x_jp^j$ and the definition of $\chi_a$ implies
\begin{align*}
    \widehat{\chi}_{a}(k) &= \frac{1}{p^d} \sum_{x < p^d} \chi_a(x)e\left(\frac{-kx}{p^d}\right) \\
    &= \prod_{j < d} \frac{1}{p} \sum_{x_j < p} e((a_jp^{-1} - kp^{j-d})x_j) \\
    &= \prod_{j < d} e\left(\frac{p-1}{2}(a_jp^{-1} - kp^{j-d})\right) \frac{\sin \pi p (a_jp^{-1} - kp^{j-d})}{p \sin \pi (a_jp^{-1} - kp^{j-d})}.
\end{align*}
Taking the absolute value, we get
\begin{equation}\label{eq.product_of_Gp}
    |\widehat{\chi}_a(k)| = \prod_{j < d} \frac{|\sin \pi p(a_jp^{-1} - kp^{j-d})|}{p |\sin \pi (a_jp^{-1} - kp^{j-d})|},
\end{equation}
which will be the starting point of many of our bounds. For $q \in \mathbb{Z}_{\geq 2}$ define the periodic even function $G_q : \mathbb{R}/\mathbb{Z} \to [-1,1]$ by
\begin{equation*}
    G_q(y) = \frac{\sin \pi q y}{q \sin \pi y}.
\end{equation*}
We will need a few basic properties of $G_q$. Using Taylor series, we find a pointwise bound
\begin{equation}\label{eq.Gq_bound}
            |G_q(y)| \leq \begin{cases}
                1-(qy)^2  & \text{ if } y \in [-1/q,1/q], \\
                1/3 & \text{ if } y \notin [-1/q,1/q].
            \end{cases}
\end{equation}
Using the expression
\begin{equation*}
     G_q(y - \ell q^{-1}) = e\left(\frac{1-q}{2}y\right)\frac{1}{q}\sum_{u < q} e(yu)e_{q}(-\ell u),
\end{equation*}
Plancherel implies that\footnote{Amusingly, the case $q = 2$ recovers the identity $\cos^2 y+ \sin^2 y= 1$.}
\begin{equation}\label{eq.sum_squares_one}
    \sum_{\ell < q} |G_q(y-\ell q^{-1})|^2 = q \sum_{u < q} \frac{1}{q^2} = 1.
\end{equation}
Using the definition of $G_q(y)$, we have
\begin{equation}\label{eq.Gq_product_formula}
    \prod_{m = 0}^{M-1} G_q(yq^m) = \frac{\sin q^M\pi y}{q^M \sin \pi y} = G_{q^M}(y).
\end{equation}

Now we study $\widehat{\chi}_a$ carefully. We begin with an  $\ell^1$ bound that comes from Fourier expanding each factor at the level of digits and applying the standard exponential sum bound. This generalizes \cite[Lemma 1]{Bourgain2013}.
\begin{lem}\label{lem.ell_1_A}
    We have
    \begin{equation*}
        \sum_{k < X} |\widehat{\chi}_a(k)| < \prod_{i=1}^r (Cd_i \log p_i)^{|a_i|}.
    \end{equation*}
\end{lem}
\begin{proof}
As in the proof of Proposition~\ref{prop.Katai_arg}, begin by writing
\begin{equation*}
    \chi_a(x) = \prod_{i=1}^r \prod_{j<d_i} \psi_{i,j}\left(\frac{x}{p_i^{j+1}}\right),
\end{equation*}
except that we do not include the conjugation in $\psi_{i,j}$. We bound $\psi_{i,j}$ for $a_{i,j} >0$ as follows.

We fix $i,j$ and write $p = p_i, d=d_i, \psi = \psi_{i,j}, a = a_{i,j}$. We have a Fourier expansion
\begin{equation*}
    \psi\left(\frac{x}{p^{j+1}}\right) = \sum_{k < p^{j+1}} \widehat{\psi}(k)e\left(\frac{xk}{p^{j+1}}\right),
\end{equation*}
where for $k \in \mathbb{Z}/p^{j+1}\mathbb{Z}$ we have
\begin{align*}
    \widehat{\psi}(k) &= \frac{1}{p^{j+1}}\sum_{v <p^{j+1}} \psi\left(\frac{v}{p^{j+1}}\right)e\left(-\frac{kv}{p^{j+1}}\right) \\
    &= \frac{1}{p^{j+1}} \sum_{u < p} e_{p}(au)\sum_{v< p^{j}}e\left(-\frac{k(up^{j-1}+v)}{p^{j+1}}\right)\\
    &= \frac{1}{p^{j+1}} \sum_{u < p} e_{p}\left((a-k)u\right)\sum_{v<p^{j}}e_{p^{j+1}}\left(-kv\right).
\end{align*}
The first sum evaluates to
\begin{equation*}
    \sum_{u=0}^{p-1} e_{p}\left((a-k)u\right) =  \begin{cases}
         p & \text{ if } k \equiv a \bmod p, \\
         0 & \text{ if } k \not\equiv a \bmod p,
     \end{cases} 
\end{equation*}
so $\widehat{\psi}(k) = 0$ unless $k \equiv a \bmod p$. We write $k = a+pt$ with $0 \leq t < p^{j}$ and obtain 
\begin{equation*}
    \widehat{\psi}(a+pt) = p^{-j} \sum_{v<p^{j}}e\left(- \frac{(a+pt)v}{p^{j+1}}\right) \ll \min\left(1, \frac{p^{-j}}{\norm{(a+pt)/p^{j+1}}}\right) 
\end{equation*}
using the standard exponential sum bound. Observe that for $0 < t < p^j$
\begin{equation*}
    \norm{(a+pt)/p^{j+1}} \geq \left|\norm{t/p^{j}} - \norm{a/p^{j+1}}\right| \geq \norm{t/p^{j}} - \frac{p-1}{p^{j+1}},
\end{equation*}
which implies that
\begin{equation*}
    \min\left(1, \frac{p^{-j}}{\norm{(a+pt)/p^{j+1}}}\right) \leq  \min\left(1, \frac{1}{p^j\norm{tp^{-j}} - 1}\right).
\end{equation*}
Since $p^j\norm{tp^{-j}} = \min(t, p^j-t)$ we conclude that 
\begin{equation*}
    \sum_{k < p^{j+1}} |\widehat{\psi}(k)| = \sum_{t < p^{j}}|\widehat{\psi}(a+pt)| \ll \sum_{0<t < p^{j}/2} \frac{1}{t} \ll j \log p,
\end{equation*}
so there is a $C > 0$ such that
\begin{equation}\label{eq.local_digit_ell_1_bound}
    \sum_{k < p^{j+1}} |\widehat{\psi}(k)| \leq C d\log p,
\end{equation}
uniformly in $j$. 

Now we take the discrete Fourier expansion of a local character
\begin{equation*}
    \chi_{a_i}^{(i)}(x) = \sum_{k < p_i^{d_i}} \widehat{\chi}_{a_i}^{(i)}(k)e\left(\frac{xk}{p_i^{d_i}}\right),
\end{equation*}
and we observe that Young's convolution inequality implies that
\begin{equation*}
     \sum_{k < p_i^{d_i}} |\widehat{\chi}_{a_i}^{(i)}(k)| < (Cd_i \log p_i)^{|a_i|}.
\end{equation*}
Using Equation~\eqref{eq.local_digit_ell_1_bound} and Equation~\eqref{eq.ell_1_mult}, we conclude the proof
\end{proof}

Next, we obtain an $\ell^\infty$ bound using digit compatibility and Equation~\eqref{eq.Gq_bound}. This generalizes \cite[Lemma 2]{Bourgain2013}.
\begin{lem}\label{lem.ell_infty}
    We have 
    \begin{equation*}
        \norm{\widehat{\chi}_a}_\infty < \prod_{\substack{i=1}}^r \left(1 - \frac{4}{9p_i^2}\right)^{\lfloor{|a_i|/2\rfloor}} \ll \prod_{i=1}^r \left(1 - \frac{4}{9p_i^2}\right)^{|a_i|/2}.
    \end{equation*}
\end{lem}

\begin{proof}
    Fix a $1 \leq i \leq r$ such that $d_i \geq 2$ and suppress the $i$ in our notation. We are given a $0 \leq a_{j} < p$ for each $j < d$ and a $0 \leq k < p^d$. Suppose that there is a $j \ge1 $ with $a_{j} > 0$ such that for some $\delta < 2/3$ we have
    \begin{equation*}
        |G_p(a_jp^{-1} -kp^{j-d})| > 1-\delta, \quad \text{ and } \quad   |G_p(a_{j-1}p^{-1}-kp^{j-d-1})| > 1-\delta. 
    \end{equation*}
    We set 
    \begin{equation*}
        \alpha_j = kp^{j-d}, \quad \alpha_{j-1} = kp^{j-1-d}, \quad \beta_j = \frac{a_j}{p} - \alpha_j, \quad \beta_{j-1} = \frac{a_{j-1}}{p} - \alpha_{j-1}
    \end{equation*}
    in $\mathbb{R}/\mathbb{Z}$. Then we have $\alpha_j = p\alpha_{j-1}$ and
    \begin{equation*}
        \norm{\beta_j}  < \frac{\sqrt{\delta}}{p}, \quad \text{ and } \quad \norm{\beta_{j-1}} < \frac{\sqrt{\delta}}{p},
    \end{equation*}
    by Equation~\eqref{eq.Gq_bound}. Hence
    \begin{equation*}
        \beta_j = \frac{a_j}{p} - \alpha_j = \frac{a_j}{p} - p\alpha_{j-1} = \frac{a_j}{p} - p\left(\frac{a_{j-1}}{p}-\beta_{j-1}\right) = \frac{a_j}{p} - a_{j-1} + p\beta_{j-1}.
    \end{equation*}
    Taking the distance to the nearest integer, we have
    \begin{equation*}
       \norm{\frac{a_j}{p} - a_{j-1}} = \norm{\beta_j - p\beta_{j-1}} \leq \norm{\beta_j} + p\norm{\beta_{j-1}} <  \left(1 + \frac{1}{p}\right)\sqrt{\delta} \leq \frac{3}{2}\sqrt{\delta},
    \end{equation*}
    by the triangle inequality. But since $a_{j-1} \in \mathbb{Z}$ and $1 \leq a_j \leq p-1$, we have
    \begin{equation*}
        \norm{\frac{a_j}{p} - a_{j-1}} \geq \frac{1}{p},
    \end{equation*}
    and we conclude that
    \begin{equation*}
        \delta > \frac{4}{9p^2}.
    \end{equation*}
    This is a contradiction if we choose
    \begin{equation*}
        \delta = \frac{4}{9p^2}.
    \end{equation*}

    Recalling Equation~\eqref{eq.product_of_Gp}, for each $j \geq 1$ with $a_{j} > 0$, we either get a factor
    \begin{equation*}
        |G_p(\beta_j)| < 1- \frac{4}{9p^2},
    \end{equation*}
    or a pair of factors
    \begin{equation*}
        |G_p(\beta_j)| \leq 1 \quad \text{ and } \quad |G_p(\beta_{j-1})| < 1 - \frac{4}{9p^2}.
    \end{equation*}
    By Equation~\eqref{eq.ell_infty_mult}, we conclude that
    \begin{equation*}
        |\widehat{\chi}_a(k)| < \prod_{i=1}^r \left(1 - \frac{4}{9p_i^2}\right)^{\lfloor |a_i|/2\rfloor}.
    \end{equation*}
\end{proof}

Now we give an $\ell^1$ bound in arithmetic progressions by averaging over shifts in two-digit blocks, which gives us a power saving over the trivial estimate. This generalizes \cite[Lemma 4]{Bourgain2013}. Note that we allow $\gamma_i = 0$ below, which generalizes \cite[Lemma 3]{Bourgain2013} aswell.
\begin{lem}\label{lem.ell_1_in_AP}
    For $d_i = 1$ let $c_i = 0$. For $d_i > 1$, all integers $0 \leq \gamma_i \leq d_i-2$ and $0 \leq b_i < p_i^{\gamma_i}$, there exists $c_i > 0$ such that 
    \begin{equation*}
        \sum_{\substack{k \leq N \\ k \equiv b_i(p_i^{\gamma_i})}} |\widehat{\chi}_a(k)|  \ll \prod_{i=1}^r p_i^{(1/2-c_i)(d_i-\gamma_i)}.
    \end{equation*}
\end{lem}

\begin{proof}
    We reduce to local characters via Equation~\eqref{eq.ell_1_mult}. Write $k = b + k'p^\gamma$ with $b < p^\gamma$ fixed and $k' < p^{d-\gamma}$, and let
    \begin{equation*}
        \beta_j = a_jp^{-1} - (b+k'p^\gamma)p^{j-d}.
    \end{equation*}
    Since $|G_p(y)| \leq 1$,
    \begin{equation}\label{eq.AP_drop_digits}
        \sum_{\substack{k < p^d \\ k \equiv b (p^\gamma)}} |\widehat{\chi}_a(k)| \leq \sum_{k' < p^{d-\gamma}} \prod_{j < d-\gamma}|G_p(\beta_j)|.
    \end{equation}
    
    Consider the shifts
    \begin{equation*}
        k'\rightarrow k' + \omega p^{d-\gamma-2} + \theta p^{d-\gamma-1} \iff \beta_{j} \rightarrow \beta_j - \omega p^{j-2} - \theta p^{j-1},
    \end{equation*}
    for $\omega,\theta \in \bb{Z}/p\mathbb{Z}$, so we have
    \begin{equation*}
         \sum_{k' < p^{d-\gamma}} \prod_{j<d-\gamma} |G_p(\beta_j)| = \frac{1}{p^2} \sum_{\omega,\theta} \sum_{k' < p^{d-\gamma}} \prod_{j= 2}^{d-\gamma-1}|G_p(\beta_j)| \prod_{j=0}^1|G_p(\beta_j - \omega p^{j-2} - \theta p^{j-1})|
    \end{equation*}
    Since we average over all translates of size $p^{-1}$, the factor
    \begin{equation*}
        F = \frac{1}{p^2} \sum_{\omega,\theta}\prod_{j=0}^1|G_p(\beta_j - \omega p^{j-2} - \theta p^{j-1})|
    \end{equation*}
    no longer depends on the $a_j$. We set $a_0=a_1=0$ so that $\beta_j = p\beta_{j-1}$ and by Equation~\eqref{eq.Gq_product_formula}
    \begin{equation*}
        F = \frac{1}{p^2} \sum_{\omega,\theta} |G_{p^2}(\beta_j - \omega p^{j-2} - \theta p^{j-1})|.
    \end{equation*}
    Combine the $(\omega, \theta) \in (\mathbb{Z}/p\mathbb{Z})^2$ parameters into a single parameter $\ell \in \mathbb{Z}/p^2\mathbb{Z}$:
    \begin{equation*}
        F = \frac{1}{p^2} \sum_{\ell < p^2}|G_{p^2}(\beta_j - \ell p^{-2})|.
    \end{equation*}

    Let $q = p^2$ and $y = -(b+k'p^\gamma)p^{-d}$, and consider
    \begin{equation*}
        F(y) =  \frac{1}{q} \sum_{\ell  < q} |G_q(y-\ell q^{-1})|.
    \end{equation*}
    as a function on $\mathbb{R}/\mathbb{Z}$. Using Cauchy–Schwarz and Equation~\eqref{eq.sum_squares_one},
    \begin{equation*}
        F(y)^2 \leq \frac{1}{q^2} \left(\sum_{\ell=0}^{q-1}1 \right)\left(\sum_{\ell=0}^{q-1} |G_q(y - \ell q^{-1})|^2\right) = \frac{1}{q},
    \end{equation*}
    with equality if and only if $|G_q(y-\ell q^{-1})|$ is constant in $\ell$. Since it is not, 
    \begin{equation*}
        \varepsilon_p(y):=\frac{1}{p} - F(y),
    \end{equation*}
    is a continuous positive function with compact domain $\mathbb{R}/\mathbb{Z}$. Thus
    \begin{equation*}
        \varepsilon_p := \inf_{y \in \mathbb{R}/\mathbb{Z}}\varepsilon_p(y) > 0,
    \end{equation*}
    and we conclude that 
    \begin{equation*}
        \norm{F}_\infty \leq \frac{1}{p} - \varepsilon_p.
    \end{equation*}
    
    Now we observe that the same two-digit averaging procedure works for the next two-digit block. In particular, averaging over the shifts
    \begin{equation*}
        k' \rightarrow k' +\omega p^{d-\gamma-4} + \theta p^{d-\gamma-3} \iff \beta_j \rightarrow \beta_j - \omega p^{j-4}-\theta p^{j-3}
    \end{equation*}
    produces the same $F(y)$ factor for the $j = 2,3$ block. Iterating this process yields 
    \begin{equation*}
        \sum_{k' < p^{d-\gamma}} \prod_{j<d-\gamma}\frac{|\sin \pi p\beta_{j}|}{p_i|\sin \pi \beta_{j}|} \leq  p^{d-\gamma}\left(\frac{1}{p} - \varepsilon_{p}\right)^{\lfloor (d-\gamma)/2\rfloor} \ll p^{(\frac12-c)(d-\gamma)}.
    \end{equation*}
    where $c > 0$ depends on $p$. Together with Equation~\eqref{eq.AP_drop_digits}, we conclude the proof.
\end{proof}

Now we prove an approximation property. This generalizes \cite[Lemma 5]{Bourgain2013}. Before doing so, we give a quick estimate that we will use repeatedly in the following proof. 

\begin{lem}\label{lem.product_trick}
    For any real number $D > 1$ and integer $q \geq 2$, there exists a $C> 0$ such that 
    \begin{equation*}
        \prod_{v = 0}^\infty \left(1 + Dq^{-v/2}\right) \ll C^\frac{(\log D)^2}{\log q}.
    \end{equation*}
\end{lem}

\begin{proof}
    Let $V = \lceil 4 \log_q D \rceil$, so that for $v \geq V$ we have   
    \begin{equation*}
        Dq^{-v/2} \leq D^{-1} < 1.
    \end{equation*}
    Now observe that
    \begin{equation*}
        \log\left(\prod_{v \leq V} (1+Dq^{-v/2})\right) \leq \log\left((1+D)^V\right) \ll V \log D \ll \frac{(\log D)^2}{\log q}
    \end{equation*}
    For the tail, we use $\log(1+x) \leq x$ to get
    \begin{equation*}
        \log\left(\prod_{v=V}^{\infty}\left(1+Dq^{-v/2}\right)\right)   \leq D\sum_{v=V}^{\infty} q^{-v/2} = \frac{Dq^{-V/2}}{1-q^{-1/2}} \ll D^{-1} \leq 1
    \end{equation*}
    where the constant is absolute. Therefore, absorbing constants into $C$, we get 
    \begin{equation*}
        \prod_{v = 0}^\infty \left(1 + Dq^{-v/2}\right)  \ll C^{\frac{(\log D)^2}{\log q}}.
    \end{equation*}
\end{proof}

\begin{lem}\label{lem.Psi_approximation}
    Let $a \in \widehat{\cal{X}}$ satisfy $a_{i,j} = 0$ for $j \leq d_i- J_i$ and assume $J_i \geq 2$. Then there exist $C_i,c_i > 0$ such that
    \begin{equation}\label{eq.support_bound}
        \sum_{k \leq N} |\widehat{\chi}_a(k)| < \prod_{i=1}^r C_i^{(\log D_i)^2/\log p_i} p_i^{(1/2-c_i)J_i+1/2} .
    \end{equation}
    where $D_i = (p_i-1)(d_i \log p_i)$. Moreover, there exists a $3^r$-bounded function $\Psi_a$ on $\Omega$ and $c_i'>0$ satisfying $|\widehat{\Psi}_a| \leq |\widehat{\chi}_a|$,
    \begin{equation*}
        \frac{1}{X}\sum_{x \leq X} |\Psi_a(x) - \chi_a(x)|^2 \ll  r \sum_{i=1}^r p_i^{-c_i't_i},
    \end{equation*}
    and
    \begin{equation*}
        \widehat{\Psi}_a(k) = 0 \quad  \text{ if } \quad |k_i'| > p_i^{J_i +t_i} \quad \text{ for some } \quad i\in [r],
    \end{equation*}
    where 
    \begin{equation*}
        k_i :\equiv k \bmod p_i^{d_i}, \quad t_i' :=  (X/p_i^{d_i})^{-1}\bmod p_i^{d_i}, \quad k_i' = t_i'k_i \bmod p_i^{d_i},
    \end{equation*}
    and $t_i \in \mathbb{Z}$ are parameters satisfying
    \begin{equation*}
        C(\log D_i)^2/\log p_i < t_i < \frac12 (d_i-J_i).
    \end{equation*}
\end{lem}

\begin{proof}
    We reduce to a local character by Equation~\eqref{eq.ell_1_mult}. Write $k = k_0 + k_1p^{J}$ with $k_0 < p^J$ and $k_1 < p^{d-J}$ and
    \begin{equation}\label{eq.original_factorization}
        |\widehat{\chi}_a(k)| = \prod_{j<d-J} |G_p((k_0 + p^Jk_1)p^{j-d})| \prod_{d-J\leq j<d} |G_p(a_jp^{-1} - k_0p^{j-d})|.
    \end{equation}
    Reindexing, we see that the second factor is
    \begin{equation*}
        \prod_{d-J\leq j<d}  |G_p(a_jp^{-1} - k_0p^{j-d})| = \prod_{j<J} |G_p(a_{j+d-J}p^{-1} - k_0p^{j-J})| = |\widehat{\chi}_{a-d+J}(k_0)|
    \end{equation*}
    where $(a-d+J)_j = a_{j+d-J}$. 
    
    For the first factor, we fix $k_0$ and consider $F : \mathbb{Z}/p^{d-J}\mathbb{Z} \to \mathbb{C}$ given by
    \begin{equation*}
        F(k_1) =  \prod_{j<d-J} |G_p((k_0 + p^Jk_1)p^{j-d})|
    \end{equation*}
    We have a Lipschitz bound for all $k_0 < p^J$:
    \begin{equation*}
        F(k_1) \leq \prod_{j<d-J} \left(|G_p(k_1p^{j-d+J})| + \frac{\pi}{2}p^{j-d+J+1}\right).
    \end{equation*}
    Now we set $y_1 = k_1p^{j-d+J}$ and we consider the set 
    \begin{equation*}
        B(k_1) = \left\{j < d-J : |G_p(y_1)| < p^{(j-d+J)/2} \right\},
    \end{equation*}
    If $j \notin B(k_1)$ we have
    \begin{equation*}
        |G_p(y_1)|+ \frac{\pi}{2}p^{j-d+J+1} \leq \left(1 + \frac{\pi}{2}p^{(j-d+J+2)/2}\right)|G_p(y_1)|
    \end{equation*}
    For $j \in B(k_1)$, we recall Equation~\eqref{eq.sum_squares_one} and observe that this implies that there exists a $b_j \in \mathbb{Z}$ with $1 \leq b_j \leq p-1$, depending on $y_1$, such that 
    \begin{equation*}
        |G_p(b_jp^{-1}+y_1)| > p^{-1/2} \geq p^{(j-d+J)/2}.
    \end{equation*}
    This implies that
    \begin{equation*}
        |G_p(y_1)|+ \frac{\pi}{2}p^{j-d+J+1}  < p^{(j-d+J+1)/2}\left(1+ \frac{\pi}{2}p^{(j-d+J+2)/2}\right) |G_p(b_jp^{-1}+y_1)| ,
    \end{equation*}
    so we have the bound
    \begin{align*}
        F(k_1) < &\prod_{j \in B(k_1)}p^{(j-d+J+1)/2}\left(1+\frac{\pi}{2}p^{(j-d+J+2)/2}\right)|G_p(b_jp^{-1}+y_1)|  \\
        &\times\prod_{j \notin B(k_1)}\left(1 + \frac{\pi}{2}p^{(j-d+J+2)/2}\right)|G_p(y_1)|,
    \end{align*}
    where $B(k_1)$ and $b_j$ depend on $k_1$. For any $b \in (\mathbb{Z}/p\mathbb{Z})^{d-J}$ define $B = \{j : b_j > 0\}$ and sum the preceding bound over all choices of $b$. Noting that
    \begin{equation*}
        \prod_{j<d-J}\left(1 + \frac{\pi}{2}p^{(j-d+J+2)/2}\right) < \prod_{v = -1}^\infty \left(1+\frac{\pi}{2}p^{-v/2}\right) \ll p^{1/2},
    \end{equation*}
    with an absolute constant, we have
    \begin{equation*}
        F(k_1) \ll p^{1/2}\sum_{b} p^{\sum_{j \in B}(j-d+J+1)/2}\prod_{j \notin B}|G_p(y_1)| \prod_{j \in B}|G_p(b_jp^{-1}+y_1)|,
    \end{equation*}
    where
    \begin{equation*}
        \prod_{j \notin B}|G_p(y_1)| \prod_{j \in B}|G_p(b_jp^{-1}+y_1)| = \prod_{j<d-J} |G_p(b_jp^{-1}+y_1)| =|\widehat{\chi}_b(k_1)|.
    \end{equation*}

    Returning to Equation~\eqref{eq.original_factorization}, we now have
    \begin{equation*}
        \sum_{k < p^d} |\widehat{\chi}_a(k)| \ll p^{1/2}\sum_{b} p^{\sum_{j \in B}(j-d+J+1)/2} \sum_{k_1 < p^{d-J}}|\widehat{\chi}_b(k_1)|   \sum_{k_0 < p^J} |\widehat{\chi}_{a-d+J}(k_0)|.
    \end{equation*}
    By Lemma~\ref{lem.ell_1_in_AP}, the sum over $k_0$ yields
    \begin{equation*}
        \sum_{k_0 < p^J} |\widehat{\chi}_{a-d+J}(k_0)| \ll p^{(1/2-c)J}.
    \end{equation*}
    By Lemma~\ref{lem.ell_1_A}, the sum over $k_1$ yields
    \begin{equation*}
        \sum_{k_1 < p^{d-J}}|\widehat{\chi}_b(k_1)| \leq (Cd\log p)^{|b|}.
    \end{equation*}
    By Lemma~\ref{lem.product_trick}, the sum over $b$ and $k_1$ yields
    \begin{align}
    \begin{split}\label{eq.bk_computation}
        \sum_{b} p^{\sum_{j \in B}(j-d+J+1)/2} (Cd \log p)^{|b|}  &= \sum_{B} (p-1)^{|B|}p^{\sum_{j \in B}(j-d+J+1)/2} (Cd \log p)^{|B|} \\
        &= \sum_B \prod_{j \in B} (p-1)(Cd \log p)p^{(j-d+J+1)/2}\\
        &= \prod_{j<d-J}\left(1+(p-1)(Cd \log p)p^{(j-d+J+1)/2}\right) \\
        & \ll C^{(\log D)^2/\log p},
    \end{split}
    \end{align}
    where $D :=(p-1)(Cd \log p)$, and we conclude equation~\eqref{eq.support_bound}.

    Next we establish the approximation claim. We begin by reducing the global statement to the local statement. Use the identification $k \leftrightarrow (k_1',\ldots,k_r')$, as in Equation~\eqref{eq.CRT_factorization}, to define $\eta$ by 
    \begin{equation*}
        \eta(k) = \prod_{i=1}^r \eta_i(k_i'),
    \end{equation*}
    and $\Psi_a$ by
    \begin{equation*}
        \Psi_a(x) = \sum_{k} \eta(k) \widehat{\chi}_a(k)e\left(\frac{kx}{X}\right) = \prod_{i=1}^r \sum_{k_i' < p_i^{d_i}} \eta_i(k_i')\widehat{\chi}^{(i)}_{a_i}(k_i')e\left(\frac{k_i'x}{p_i^{d_i}}\right)
    \end{equation*}
    where $\eta_i : \mathbb{R} \to [0,1]$ is trapezoidal with $\eta_i(k_i) = 1$ for $|k_i'| < K_ip_i^{J_i}$, $\eta(k_i) = 0$ for $|k_i'| > 2K_ip_i^{J_i}$. Notice that $\norm{\Psi_a}_\infty \leq 3^r$ and $\widehat{\Psi}_a(k) = \widehat{\chi}_a(k)$ when $|k_i'| \leq K_ip_i^{J_i}$ for all $i$, and $\widehat{\Psi}_a(k) = 0$ if $|k_i'| > 2K_ip_i^{J_i}$ for some $i$. By Plancherel,
    \begin{equation*}
        \norm{\Psi_a - \chi_a}_2^2 = \sum_{k< X} |(1-\eta(k))\widehat{\chi}_a(k)|^2 = \sum_{(k_1',\ldots,k_r')} \left|1-\prod_{i=1}^r \eta_i(k_i')\right|^2 \prod_{\ell=1}^r |\widehat{\chi}^{(\ell)}_{a_\ell}(k_\ell')|^2,
    \end{equation*}
    and
    \begin{equation*}
         \left|1-\prod_{i=1}^r \eta_i(k_i')\right|^2 \leq \left(\sum_{i=1}^r(1-\eta_i(k_i')) \right)^2 \leq r \sum_{i=1}^r (1-\eta_i(k_i'))^2.
    \end{equation*}
    Therefore
    \begin{align}\label{eq.approx_2_norm}
    \begin{split}
        \norm{\Psi_a - \chi_a}_2^2 &\leq r \sum_{i=1}^r  \sum_{(k_1',\ldots,k_r')} (1-\eta_i(k_i'))^2 |\widehat{\chi}^{(i)}_{a_i}(k_i')|^2\prod_{\ell\neq i}^{r}|\widehat{\chi}^{(\ell)}_{a_\ell}(k_\ell')|^2 
        \\
        &= r\sum_{i=1}^r \sum_{k_i' <p_i^{d_i}} (1-\eta_i(k_i'))^2 |\widehat{\chi}^{(i)}_{a_i}(k_i')|^2 \prod_{\ell\neq i}^{r}\sum_{k_\ell' < p_\ell^{d_\ell}}|\widehat{\chi}^{(\ell)}_{a_\ell}(k_\ell')|^2 
        \\
        & = r\sum_{i=1}^r \sum_{k_i' < p_i^{d_i}} (1-\eta_i(k_i'))^2|\widehat{\chi}_{a_i}^{(i)}(k_i')|^2,
    \end{split}
    \end{align}
    since Parseval gives us
    \begin{equation*}
        \sum_{k_\ell' < p_\ell^{d_\ell}}|\widehat{\chi}^{(\ell)}_{a_\ell}(k_\ell')|^2 = 1.
    \end{equation*}

    Now we establish the local bound. We take $C (\log D)^2/\log p < \rho < \frac12(d-J)$ and we consider the two cases $\min B < d-J-\rho$ and $B \subseteq [d-J-\rho, d-J]$. In the first case we estimate
    \begin{align*}
        \sum_{k_1<p^{d-J}} \sum_{\substack{b \in (\mathbb{Z}/p\mathbb{Z})^{d-J} \\ \min B < d-J-\rho}} p^{\sum_{j \in B}(j-d+J+1)/2} |\widehat{\chi}_b(k_1)| &\ll \sum_{\min B < d-J-\rho} \prod_{j \in B} Dp^{(j-d+J+1)/2} 
    \end{align*}
    as in the computation in Equation~\eqref{eq.bk_computation}. This time we get the bound
    \begin{equation*}
        \ll \left(\prod_{j<d-J-\rho}\left(1+Dp^{(j-d+J+1)/2}\right)-1\right)\prod_{d-J-\rho \leq j < d-J}\left(1+Dp^{(j-d+J+1)/2}\right) 
    \end{equation*}
    since we must subtract off all terms whose factors lie in $j \geq d-J-\rho$. For the first product, we use $C(\log D)^2/\log p < \rho$ to get
    \begin{equation*}
        \sum_{j<d-J-\rho}Dp^{(j-d+J+1)/2} < D\sum_{v=\rho}^{\infty} p^{-v/2} \ll Dp^{-\rho/2} \ll p^{-\rho/4}
    \end{equation*}
    and $\log (1+x) \leq x$ to get
    \begin{equation*}
        \prod_{j<d-J-\rho}\left(1+Dp^{(j-d+J+1)/2}\right) \ll \exp(p^{-\rho/4}) \ll 1+p^{-\rho/4}.
    \end{equation*}
    For the second product, we use Lemma~\ref{lem.product_trick} to get
    \begin{equation*}
        \prod_{d-J-\rho \leq j < d-J}\left(1+Dp^{(j-d+J+1)/2}\right)  = \prod_{v=0}^{\rho-1}\left(1 + Dp^{-v/2}\right) \ll C_1^\frac{(\log D)^2}{\log p}.
    \end{equation*}
    Since $C (\log D)^2/\log p < \rho$, we conlude the bound
    \begin{equation*}
        \sum_{k_1<p^{d-J}} \sum_{\substack{b \in (\mathbb{Z}/p\mathbb{Z})^{d-J} \\ \min B < d-J-\rho }} p^{\sum_{j \in B}(j-d+J+1)/2} |\widehat{\chi}_b(k_1)|  \ll C_1^\frac{(\log D)^2}{\log p}p^{-\rho/4} \ll p^{-\rho/4}.
    \end{equation*}
    Since
    \begin{align}\label{eq.quick_bound1}
    \begin{split}
         \sup_{k_1<p^{d-J}} \sum_{\substack{b \in (\mathbb{Z}/p\mathbb{Z})^{d-J} \\ \min B < d-J-\rho }} p^{\sum_{j \in B}(j-d+J+1)/2} |\widehat{\chi}_b(k_1)|& \leq  \sum_{\min B < d-J-\rho}\prod_{j \in B}(p-1) p^{(j-d+J+1)/2} \\
        & \leq \prod_{j<d-J} \left(1+p^{(j-d+J+3)/2}\right) \ll p^{3/2},
    \end{split}
    \end{align}
    the preceding equation implies
    \begin{equation}\label{eq.small_B_bound}
        \sum_{k_1<p^{d-J}}\Bigg(\sum_{\substack{b \in (\mathbb{Z}/p\mathbb{Z})^{d-J} \\ \min B < d-J-\rho }} p^{\sum_{j \in B}(j-d+J+1)/2}|\widehat{\chi}_b(k_1)|\Bigg)^2 \ll p^{3/2-\rho/4}.
    \end{equation}

    Now we bound $\widehat{\chi}_b(k_1)$ for $B \subseteq [d-J-\rho, d-J]$. We write
    \begin{equation*}
        |\widehat{\chi}_b(k_1)| = \prod_{j<d-J-\rho} |G_p(k_1p^{j-d+J})|\prod_{d-J-\rho\leq j < d-J} |G_p(b_jp^{-1}+k_1p^{j-d+J})|,
    \end{equation*}
    bound the second factor trivially, and use Equation~\eqref{eq.Gq_product_formula} to obtain
    \begin{equation*}
         |\widehat{\chi}_b(k_1)| \leq \prod_{j < d-J-\rho}|G_p(k_1p^{j-d+J})| = |G_{p^{d-J-\rho}}(k_1p^{-(d-J)})|.
    \end{equation*}
    Now the standard bounds for $\sin$ imply
    \begin{equation}\label{eq.G_digit_bound}
        |G_{p^{d-J-\rho}}(k_1p^{-(d-J)})| \leq \frac{1}{p^{d-J-\rho}2k_1p^{-(d-J)}} = \frac{p^{\rho}}{2k_1} < \frac12 k_1^{-1/2},
    \end{equation}
    since $p^{2\rho} < k_1$. Let us denote $w_B = p^{\sum_{j \in B}(j-d+J+1)/2} (p-1)^{|B|}$. For $p^{2\rho} \leq K_1 <p^{d-J}$ we have
    \begin{align}\label{eq.large_B_bound}
    \begin{split}
        &\qquad\sum_{K_1 < |k_1| < p^{d-J}}\left(\sum_{B \subseteq [d-J-\rho,d-J]}w_B|\widehat{\chi}_b(k_1)|\right)^2 \\
         &\leq \sum_{B \subseteq [d-J-\rho,d-J]}w_B
        \sum_{B \subseteq [d-J-\rho,d-J]}\sum_{K_1 < |k_1| < p^{d-J}}w_B|\widehat{\chi}_b(k_1)|^2 \\
        &\ll \sum_{B \subseteq [d-J-\rho,d-J]}w_B
        K_1^{-1/2}\sum_{B \subseteq [d-J-\rho,d-J]}w_B\norm{\widehat{\chi}_b}_1 \\
        &\ll \sum_{B \subseteq [d-J-\rho,d-J]}w_BK_1^{-1/2} C^{(\log D)^2/\log p} \\
        &\ll p^{3/2}K_1^{-1/2} C^{(\log D)^2/\log p}
    \end{split}
    \end{align}
    where we used Cauchy-Schwarz, Equation~\eqref{eq.G_digit_bound}, Equation~\eqref{eq.bk_computation}, and Equation~\eqref{eq.quick_bound1}. 
    
    Returning to our factorization in Equation~\eqref{eq.original_factorization} and using Plancherel we obtain
    \begin{align*}
        \sum_{|k| > K_1p^{J}} |\widehat{\chi}_a(k)|^2         &\leq \sum_{k_0 < p^J} |\widehat{\chi}_{a-d+J}(k_0)|^2 \sum_{K_1 < |k_1|<p^{d-J}}\left(p^{1/2}\sum_{b} p^{\sum_{j \in B}(j-d+J+1)/2}|\widehat{\chi}_b(k_1)|\right)^2\\
        &\leq \sum_{K_1 < |k_1|<p^{d-J}}\left(p^{1/2}\sum_{b} p^{\sum_{j \in B}(j-d+J+1)/2}|\widehat{\chi}_b(k_1)|\right)^2.
    \end{align*}
    Now we split the sum over $b$ into $\min B < d-J-\rho$ and $B \subseteq [ d-J-\rho, d-J]$. For small $B$ we use Equation~\eqref{eq.small_B_bound}, and for large $B$ we use Equation~\eqref{eq.large_B_bound} to obtain
    \begin{align*}
        &\sum_{|k| > K_1p^{J}} |\widehat{\chi}_a(k)|^2 \ll  p^{5/2}\left( p^{-\rho/4} + K_1^{-1/2} C^{(\log D)^2/\log p}\right)
    \end{align*}
    Now we choose $K_1= p^{t-1}$ and $\rho = \frac12(t-1)$ to get
    \begin{equation*}
        \sum_{|k| > K_1p^{J}} |\widehat{\chi}_a(k)|^2 \ll p^{-c't},
    \end{equation*}
    for some $c' > 0$ depending on $p$. By Equation~\eqref{eq.approx_2_norm}, this completes the proof.
\end{proof}

We conclude this section with the following generalization of \cite[Lemma 6]{Bourgain2013}
\begin{lem}\label{lem.interval}
    If $I \subseteq \Omega $ is an interval, there is $c>0$ such that
    \begin{equation*}
        \sum_{k \in I} |\widehat{\chi}_a(k)| \ll (p_r|I|)^{\frac12 - c}.
    \end{equation*}
\end{lem}

\begin{proof}
    Choose $Q \mid X$ such that $|I| \leq Q \leq p_r|I|$ and write $Q = \prod_{i=1}^r p_i^{m_i}$. Then
    \begin{equation*}
        |\widehat{\chi}_a(k)|  \leq  \prod_{i=1}^r\prod_{d_i-m_i \leq j<d_i}|G_{p_i}(a_{i,j}p_i^{-1}-kp_i^{j-d_i})|
        = \prod_{i=1}^r\prod_{j<m_i}|G_{p_i}(a_{i,j}'p_i^{-1}-kp_i^{j-m_i})| 
    \end{equation*}
    where $a'_{i,j} = a_{i,j+d_i-m_i}$. The right-hand side is $|\widehat{\chi}_{a'}(k)|$ where $\chi_{a'}$ is a character on $\prod_{i=1}^r \mu_{i}^{m_i}$. In particular, since $|\widehat{\chi}_{a'}(k)|$ is $Q$-periodic and $|I| \leq Q$,
    \begin{equation*}
        \sum_{k \in I} |\widehat{\chi}_a(k)| \leq \sum_{k \in I} |\widehat{\chi}_{a'}(k)| \leq \sum_{k < Q} |\widehat{\chi}_{a'}(k)| \ll \prod_{i=1}^r p_i^{(\frac12-c_i)m_i}
    \end{equation*}
    by Lemma~\ref{lem.ell_1_A}. Now we observe that 
    \begin{equation*}
        \prod_{i=1}^r p_i^{(\frac12-c_i)m_i} = Q^{\frac12-c} \quad \text{ where } \quad c = \frac{\sum_{i=1}^r c_i m_i \log p_i}{\sum_{i=1}^r m_i \log p_i} \geq \min_{1 \leq i \leq r} c_i.
    \end{equation*}
    We conclude that
    \begin{equation*}
        \sum_{k \in I} |\widehat{\chi}_a(k)| \ll Q^{(\frac12 - c)} \leq (p_r|I|)^{(\frac12 - c)}.
    \end{equation*}
\end{proof}

\section{Proof of number-theoretic input} \label{sec.main_proof}

\subsection{Cyclic case}
The cyclic case is where $d_i = 1$ for all $i \in [r]$. Observe that the characters $\chi_a$ reduce to the usual additive characters on $\mathbb{Z}/X\mathbb{Z}$ in this case. In particular, we have
\begin{equation*}
    \cal{X} = \prod_{i=1}^r \mu_{p_i} \cong \prod_{i=1}^r \mathbb{Z}/p_i\mathbb{Z} \cong \mathbb{Z}/X\mathbb{Z},
\end{equation*}
and our characters are just
\begin{equation*}
    \chi_a(x) = \prod_{i=1}^r e_{p_i}(a_ix_i) = e\left(\sum_{i=1}^r \frac{a_ix_i}{p_i} \right)  = e\left(\frac{1}{X}\sum_{i=1}^r \frac{ Xa_ix_i}{p_i} \right) = e\left(\frac{tx}{X}\right),
\end{equation*}
where $t = \sum_{i=1}^r \frac{Xa_i}{p_i}$. Theorem~\ref{thm.Davenport} implies that for all $A > 0$,
\begin{equation*}
    \sup_{a \in \widehat{\cal{X}}}\left|\sum_{x < X} \mu(x)\chi_a(x) \right|= \sup_{t < X}\left|\sum_{x < X} \mu(x)e\left(\frac{tx}{X}\right)\right| \leq \sup_{\theta \in \mathbb{R}/\mathbb{Z}} \left| \sum_{x < X} \mu(x)e(\theta x)\right| \ll X (\log X)^{-A}.
\end{equation*}
Since $d \log 2 \leq \sum_i \log p_i = \log X$,
\begin{equation}\label{eq.cyclic_bound}
    |\widehat{\mu}(a)| \ll (\log X)^{-A} \ll d^{-A},
\end{equation}
uniformly in $a$. Now Theorem~\ref{thm.main_cyclic} and the $\mathcal{X}_d = \prod_{i=1}^r \mu_{p_i}$ case of Theorems~\ref{thm:main_kernel}, \ref{thm:main_NGD} and ~\ref{thm:main_CSQ} follow from Section~\ref{sec.lower_bounds} and Corollary~\ref{cor:alignment_X_d}.

\subsection{Small weight case}\label{sec.small_weight}
In this section, we prove Theorem~\ref{thm:main_math} for the case where $|a| \ll d^{1/2-\varepsilon}$. By Proposition~\ref{prop.Katai_arg}, this will involve controlling the contribution of major arcs. Following the ideas in \cite{green2012notcomputingmobiusfunction}, we take advantage of Remark~\ref{rem:sparse_rational} to resolve the contribution of a potential Siegel zero. This is the content of the following lemma.

\begin{prop}\label{prop:p_power_major_arc_mu_Lambda}
Fix an odd prime $p$. There exist constants $c,\eta > 0$ depending on $p$ such that the following holds. Let $X \geq 3$, $Q=\exp(\eta\sqrt{\log X})$, and $q = p^s \leq Q$. Suppose that $(a,q) = 1$ and 
\begin{equation*}
    \theta = \frac{a}{q}+\beta, \qquad |\beta| \leq \frac{Q}{qX}
\end{equation*}
Then
\begin{equation*}
    \sum_{n\leq X}\mu(n)e(n\theta) \ll_p X\exp(-c\left(\log X\right)^{1/2}),
\end{equation*}
and
\begin{equation*}
    \sum_{n\leq X}\Lambda(n)e(n\theta)
    =
    \frac{\mu(q)}{\varphi(q)}
    \int_0^X e(\beta t)\,dt
    +
    O_p\left(X\exp(-c\left(\log X\right)^{1/2})\right).
\end{equation*}
\end{prop}

\begin{proof}
We begin with Dirichlet characters. We claim that there are constants $c,\eta>0$ such that, for every Dirichlet character
\(\chi\) of conductor $p^j\leq \exp(\eta\left(\log X\right)^{1/2})$, one has
\begin{equation}\label{eq:mu_character_p_power}
    \sum_{n\leq X}\mu(n)\chi(n) \ll_p X\exp(-c\left(\log X\right)^{1/2}),
\end{equation}
and
\begin{equation}\label{eq:Lambda_character_p_power}
    \sum_{n\leq X}\Lambda(n)\chi(n) = \mathbbm{1}_{\chi=\chi_0}X + O_p\left(X\exp\left(-c\left(\log X\right)^{1/2} \right)\right).
\end{equation}
These are the standard consequences of Perron's formula and the classical zero-free region for Dirichlet $L$-functions. The only possible obstruction is an exceptional zero of a primitive real character. In the present setting, this obstruction is harmless. Since $p$ is odd, every real character of $p$-power conductor factors through modulo $p$. Thus, the only primitive non-principal real character of $p$-power conductor is the quadratic character $\chi_p$ modulo $p$. This character is fixed once $p$ is fixed. Since $L(1,\chi_p) \neq 0$, there is a constant $\delta_p > 0$ such that $L(s,\chi_p) \neq 0$ for real $s \in [1-\delta_p,1]$ by continuity. Therefore, no real zero of a $p$-power conductor character can approach $s=1$ as $X\to\infty$. After decreasing $c$ and $\eta$, the usual zero-free region gives \eqref{eq:mu_character_p_power} and \eqref{eq:Lambda_character_p_power} uniformly for $p^j\leq \exp(\eta\left(\log X\right)^{1/2})$.

We now pass from character twists to additive twists. First consider $\mu$. Let $q=p^s\leq Q$. For $(n,p) =1 $, expanding in Dirichlet characters modulo $q$ gives
\begin{equation*}
    e\left(\frac{an}{q}\right)
    =
    \frac{1}{\varphi(q)}
    \sum_{\chi\bmod q}
    \tau_q(\overline{\chi})\chi(a)\chi(n),
\end{equation*}
where $\tau_q(\overline{\chi})$ is the Gauss sum. Using the trivial bound $|\tau_q(\overline{\chi})|\leq \varphi(q)$ and \eqref{eq:mu_character_p_power}, we obtain
\begin{equation*}
    \sum_{\substack{n\leq X\\(n,p)=1}}
    \mu(n)e\left(\frac{an}{q}\right)
    \ll_p
    \varphi(q)X\exp(-c\left(\log X\right)^{1/2}).
\end{equation*}
For $p \mid n$ and $\mu(n) \neq 0$ we $n = pm$ with $(m,p) =1$ and $\mu(n) = -\mu(m)$. Thus their contribution is
bounded in the same way with denominator $q/p$. Therefore
\begin{equation*}
    \sum_{n\leq X}\mu(n)e\left(\frac{an}{q}\right)
    \ll_p
    \varphi(q)X\exp(-c\left(\log X\right)^{1/2}).
\end{equation*}
Since $q\leq Q=\exp(\eta\left(\log X\right)^{1/2})$, choosing $\eta$ sufficiently small
and decreasing $c$ gives
\begin{equation}\label{eq:mu_exact_p_power}
    \sum_{n\leq X}\mu(n)e\left(\frac{an}{q}\right)
    \ll_p
    X\exp(-c\left(\log X\right)^{1/2}).
\end{equation}

Now consider $\Lambda$. Since the terms with $p \mid n$ contribute only $O_p((\log X)^2)$, we may restrict to $(n,p) =1$. Expanding in Dirichlet characters yields
\begin{align*}
    \sum_{n\leq X}\Lambda(n)e\left(\frac{an}{q}\right)
    &=
    \frac{1}{\varphi(q)}
    \sum_{\chi\bmod q}
    \tau_q(\overline{\chi})\chi(a)
    \sum_{n\leq X}\Lambda(n)\chi(n)
    +
    O_p((\log X)^2).
\end{align*}
The principal character contribution is
\begin{equation*}
    \frac{\mu(q)}{\varphi(q)}X + O_p\left(X \exp(-c \left(\log X\right)^{1/2})\right)
\end{equation*}
The non-principal characters are bounded using
\eqref{eq:Lambda_character_p_power}. As above, the total cost of summing over
characters is at most a factor $\varphi(q)\leq Q$, which is absorbed by taking
$\eta$ sufficiently small. Thus
\begin{equation}\label{eq:Lambda_exact_p_power}
    \sum_{n\leq X}\Lambda(n)e\left(\frac{an}{q}\right)
    =
    \frac{\mu(q)}{\varphi(q)}X
    +
    O_p\left(X\exp(-c\left(\log X\right)^{1/2})\right).
\end{equation}

It remains to pass from $a/q$ to $\theta=a/q+\beta$ for $|\beta| \leq \frac{Q}{qX}$ using partial
summation. From Equation~\eqref{eq:mu_exact_p_power}, uniformly for $Y\leq X$,
\[
    \sum_{n\leq Y}\mu(n)e\left(\frac{an}{q}\right)
    \ll_p
    X\exp(-c\left(\log X\right)^{1/2}),
\]
after enlarging the implicit constant to handle small $Y$. Therefore
\begin{align*}
    \sum_{n\leq X}\mu(n)e(n\theta)
    &=
    \sum_{n\leq X}\mu(n)e\left(\frac{an}{q}\right)e(\beta n) \\
    &\ll_p
    X\exp(-c\left(\log X\right)^{1/2})(1+|\beta|X).
\end{align*}
Since $|\beta|X\leq Q/q\leq Q$, this is $\ll_p X\exp(-c\left(\log X\right)^{1/2})$ after decreasing $c$. Similarly, applying partial summation to Equation~\eqref{eq:Lambda_exact_p_power} yields
\begin{align*}
    \sum_{n \leq X} \Lambda(n)e(n\theta) &= \sum_{n \leq X} \Lambda(n)e\left(\frac{an}{q}\right)e(\beta n) \\
    &= \frac{\mu(q)}{\varphi(q)}\int_0^X e(\beta t)dt + O_p\left(X \exp(-c\left(\log X\right)^{1/2})(1+|\beta|X)\right) \\
    &= \frac{\mu(q)}{\varphi(q)}\int_0^X e(\beta t)dt + O_p\left(X \exp(-c\left(\log X\right)^{1/2})\right)
\end{align*}
after decreasing $c$.
\end{proof}

Now Theorem~\ref{thm:main_math} for $|a| \ll d^{1/2-\varepsilon}$ is deduced from the following generalization of \cite[Proposition 1]{green2012notcomputingmobiusfunction}. 

\begin{prop}\label{prop.smallA}
    There is a $c>0$ such that for any $a \in \widehat{\cal{X}}$
    \begin{equation*}
        \widehat{\mu}(a) \ll |a|\sqrt{p_r}\exp\left(-\frac{c\sqrt{d}}{|a|}\right).
    \end{equation*}
\end{prop}

\begin{proof}
    Suppose that this proposition fails. Then there is some $a \in \widehat{\cal{X}}$ with
    \begin{equation*}
        |\hat{\mu}(a)| > 10|a|\sqrt{p_r}\exp\left(-\frac{c\sqrt{d}}{|a|}\right),
    \end{equation*}
    for all $c > 0$. By Proposition~\ref{prop.Katai_arg} and Remark~\ref{rem:sparse_rational} there is a $\theta \in [0,1]$ which is a sparse $p$-adic rational satisfying
    \begin{equation}\label{eq:small_weight_contradiction}
        |\hat{\mu}(\theta)| \geq \exp\left(-4c\sqrt{d}\right) \geq \exp\left(-\frac{4c}{\log 2}(\log X)^{1/2}\right),
    \end{equation}
    since $\log X = \sum_{i=1}^r d_i \log p_i \geq d \log 2$.

    We take the usual major and minor arc decomposition: let $Q = \exp(\eta (\log X)^{1/2})$ with $\eta > 0$ and for $1 \leq q \leq Q$ and $(b,q) = 1$, 
    \begin{equation}\label{eq:major_minor_arcs_def}
        \mathfrak{M}(b,q) := \left\{\theta \in \mathbb{R}/\mathbb{Z} : \left|\theta - \frac{b}{q}\right| \leq \frac{Q}{qX}\right\}, \qquad \mathfrak{M}= \bigcup_{q \leq Q} \bigcup_{\substack{b \bmod q \\ (b,q)=1}}, \qquad \mathfrak{m} = (\mathbb{R}/\mathbb{Z}) \setminus \mathfrak{M}.
    \end{equation}
    If $\theta \in \mathfrak{m}$, we get a $C>0$ such that
    \begin{equation*}
        \sum_{n < X} \mu(n)e(n\theta) \ll X \exp\left(-C (\log X)^{1/2}\right),
    \end{equation*}
    from classical minor-arc estimates, for example \cite[Theorem 13.9]{IwaniecKowalski2004}, which contradicts Equation~\eqref{eq:small_weight_contradiction} by choosing $c$ small enough. Likewise, we get a contradiction for $\theta \in \mathfrak{M}$, by Lemma~\ref{lem:sparse_padic_major_arc} and Proposition~\ref{prop:p_power_major_arc_mu_Lambda}, concluding the proof.
\end{proof}

Next we complete the proof for the fixed $p$ case. This generalizes \cite[Theorem 1]{Bourgain2013}.

\subsection{Reduction to Type I/II}

We begin with the general technique of reducing to Type I/II sums for estimating
\begin{equation*}
    \sum_{x \leq X} \mu(x) g(x).
\end{equation*}
Fix a $q \in \mathbb{Z}_{\geq 2}$ and take a geometric decomposition:
\begin{equation*}
    \sum_{x \leq X} \mu(x) \chi_a(x) = \sum_{x} \sum_{x/q < n \leq x}\mu(n) \chi_a(n),
\end{equation*}
where the sum over $x$ has $O(\log X)$ terms. We focus on the inner sum
\begin{equation*}
    S(x) = \sum_{x/q < n \leq x}\mu(n) g(n).
\end{equation*}

We use the following inclusion-exclusion identity for Möbius \cite[Proposition 13.5]{IwaniecKowalski2004}. For any $y,z \geq 1$ and $m > \max(y,z)$ we have
\begin{equation*}
    \mu(m) = -\sum_{\substack{bc \mid m \\ b \leq y, c \leq z}} \mu(b)\mu(c) + \sum_{\substack{bc \mid m \\ b >y , c> z}} \mu(b)\mu(c).
\end{equation*}
Taking $y = x^{\gamma_1}, z= x^{\gamma_2}$ for some $0 < \gamma_i < 1$, and assuming $x/q > x^{\max_i \gamma_i}$, we get
\begin{equation*}
    S(x) = -\sum_{x/q <n \leq x}\sum_{\substack{bc \mid n \\ b,c \leq x^{\gamma_1},x^{\gamma_2}}} \mu(b)\mu(c)g(n) + \sum_{x/q <n \leq x} \sum_{\substack{bc \mid n \\ b,c > x^{\gamma_1},x^{\gamma_2}}} \mu(b)\mu(c)g(n).
\end{equation*}
Switch the order of summation for the remaining terms:
\begin{equation*}
    S(a,x) = -\sum_{b,c \leq x^{\gamma_1},x^{\gamma_2} } \mu(b)\mu(c) \sum_{\frac{x}{qbc}< m \leq \frac{x}{bc}} g(bcm) +\sum_{b,c > x^{\gamma_1},x^{\gamma_2}} \mu(b)\mu(c) \sum_{\frac{x}{qbc}< m \leq \frac{x}{bc}} g(bcm)
\end{equation*}
We identify the first sum as a Type I contribution: 
\begin{align*}
    \sum_{b,c \leq x^{\gamma_1},x^{\gamma_2} } \mu(b)\mu(c) \sum_{\frac{x}{qbc}< m \leq \frac{x}{bc}} g(bcm) &=  \sum_{k \leq x^{\gamma_1+\gamma_2}}\left( \sum_{\substack{bc = k \\ b,c \leq x^{\gamma_1},x^{\gamma_2}}} \mu(b)\mu(c)\right) \sum_{\frac{x}{qk} < m \leq \frac{x}{k}} g(km) \\
    &\leq \sum_{k \leq x^{\gamma_1+\gamma_2}} \tau(k)\left| \sum_{\frac{x}{qk} < m \leq \frac{x}{k}} g(km)\right| \\
    &= \sum_{K} \sum_{K < k \leq qK} \tau(k) \left| \sum_{\frac{x}{qk} < m \leq \frac{x}{k}} g(km)\right|,
\end{align*}
where $qK \leq x^{\gamma_1+\gamma_2}$ and the sum over $K$ has $O(\log x)$ terms. We identify the second sum as a Type II contribution:
\begin{align*}
    \sum_{b,c > x^{\gamma_1},x^{\gamma_2}} \mu(b)\mu(c) \sum_{\frac{x}{qbc}< m \leq \frac{x}{bc}} g(bcm)&= \sum_{b > x^{\gamma_1}} \mu(b) \sum_{\frac{x}{qb} < k \leq \frac{x}{b}} \Bigg(\sum_{\substack{c \mid k \\ c>x^{\gamma_2}}}\mu(c)\Bigg) g(kb) \\
    &=  \sum_B \sum_{\substack{B < b \leq qB }} \mu(b) \sum_{\substack{ \\\frac{x}{qb} < k \leq \frac{x}{b}}} \Bigg(\sum_{\substack{c \mid k \\ c>x^{\gamma_2}}}\mu(c)\Bigg) g(kb) \\
    &\leq \sum_B \Bigg|\sum_{\substack{B < b \leq qB }}  \sum_{\substack{ \\\frac{x}{qb} < k \leq \frac{x}{b}}} \alpha_k\beta_b g(kb)\Bigg|,
\end{align*}
where $\beta_b = \mu(b)\mathbbm{1}_{(x^{\gamma_1}, \infty)}(b)$ and $\alpha_k=\sum_{\substack{c \mid k \\ c>x^{\gamma_2}}}\mu(c)$. It is clear that $|\beta_b| \leq 1$ and $|\alpha_k| \leq \tau(k)$. This time we have $x^{\gamma_1}< B < x^{1-\gamma_2} $ and the sum over $B$ has $O(\log x)$ terms. Therefore our constraints on $\gamma_i$ are $\gamma_1+\gamma_2 < 1$. In particular, we find the following reduction to Type I and Type II sums.
 \begin{lem}\label{lem.TypeIandII_reduction}
     Suppose that $q \in \mathbb{Z}_{\geq 2}$, and take $\gamma_1,\gamma_2 > 0$ such that $\gamma_1+\gamma_2 < 1$. Assume that $x^{1-\max_i\gamma_i} > q$ and let $g$ be an arithmetic  function. Suppose that for every $M \leq X$ and every sequence of complex numbers $\alpha_m, \beta_n$ with $|\alpha_m| \leq \tau(m)$, $|\beta_n| \leq 1$, we have
     \begin{equation*}
         \sum_{M < m \leq qM}  \tau(m)\left| \sum_{\frac{X}{qm} < n \leq \frac{X}{m}} g(mn)\right| \leq U \quad \text{ for } qM \leq X^{\gamma_1 + \gamma_2} \quad \text{(Type I)},
     \end{equation*}
     and
     \begin{equation*}
         \left|\sum_{M < m \leq qM} \sum_{\substack{ \\\frac{X}{qm} < n \leq \frac{X}{m}}} \alpha_m\beta_n g(mn) \right| \leq U \quad  \text{ for } X^{\gamma_1} < M < X^{1-\gamma_2} \quad \text{(Type II)}.
     \end{equation*}
     Then we have
     \begin{equation*}
         \left| \sum_{x \leq X} \mu(x)g(x) \right| \ll_{\gamma_i} (\log X)^{2} U.
     \end{equation*}
 \end{lem}

Now we use Cauchy–Schwarz to give a further reduction for the Type II sums following the same strategy as \cite[Section 6]{MauduitRivat2010}. Fix ranges $m \asymp M$ and $n \asymp N$ such that $M \leq N$ with $MN \asymp X$. We estimate
\begin{equation}\label{eq.TypeII_estimate}
    \left|\sum_{\substack{m \asymp M \\ n \asymp N}} \alpha_m\beta_n \chi_a(nm) \right| \leq \sum_{m \asymp M}\tau(m)\left|\sum_{n \asymp N} \beta_n \chi_a(mn) \right|.
\end{equation}

We use a slight variation of \cite[Lemma 4]{MauduitRivat2010}, which is just the van der Corput trick. 

\begin{lem}\label{lem.van-der-Corput}
    Let $(z_n)_{n \asymp N}$ be complex numbers supported on an interval of length $\ll N$, and let $L,K \geq 1$ be integers such that $LK \leq N$. Then
    \begin{equation*}
        \left|\sum_{n \asymp N} z_n\right|^2 \ll \frac{N}{L} \sum_{|\ell| < L} \left(1 - \frac{|\ell|}{L}\right) \sum_{n \asymp N} z_n \overline{z_{n+\ell K}}.
    \end{equation*}
\end{lem}

\begin{proof}
    Extend to a sequence $(z_n)_{n \in \mathbb{Z}}$ setting $z_n = 0$ outside of the original interval. Then
    \begin{equation*}
        L \sum_n z_n =\sum_n \sum_{\ell=1}^L z_n = \sum_n \sum_{\ell=1}^L z_{n+\ell K}.
    \end{equation*}
    Cauchy-Schwarz yields
    \begin{equation*}
        \left| \sum_n z_n  \right|^2 = \frac{1}{L^2}\left| \sum_{n} \sum_{\ell=1}^L z_{n+\ell K}  \right|^2 \ll  \frac{N}{L^2}\sum_{n} \left| \sum_{\ell=1}^L z_{n+\ell K}  \right|^2
    \end{equation*}
    since there are $\ll N + LK \ll N$ values of $n$ for which the $\ell$ sum is nonzero. Expanding the square, we get
    \begin{equation*}
        \sum_{n} \left| \sum_{\ell=1}^L z_{n+\ell K}  \right|^2 = \sum_n \sum_{\ell_1,\ell_2=1}^L z_{n+\ell_1K} \overline{z_{n+\ell_2K}}.
    \end{equation*}
    Now we set $m = n+\ell_1 K$ and $\ell = \ell_2 - \ell_1$, so $n+\ell_2k = m+\ell K$. For $1 \leq \ell_1,\ell_2 \leq L$ we get $|\ell| < L$, and exactly $L - |\ell|$ choices of $\ell_1,\ell_2$ for each $\ell$. Therefore
    \begin{equation*}
        \sum_{n} \left| \sum_{\ell=1}^L z_{n+\ell K}  \right|^2 = \sum_{|\ell| < L} (L-|\ell|) \sum_{m} z_m \overline{z_{m+\ell K}}.
    \end{equation*}
    We conclude that
    \begin{equation*}
       \left| \sum_n z_n  \right|^2 \ll \frac{N}{L} \sum_{|\ell| < L} \left(1-\frac{|\ell|}{L}\right)\sum_{m} z_m \overline{z_{m+\ell K}} 
    \end{equation*}
\end{proof}
Thus for any choice of $L,K$ with $LK\leq N$, Lemma~\ref{lem.van-der-Corput}, implies that 
\begin{equation*}
    \left|\sum_{n \asymp N}  \beta_n \chi_a(mn)\right|^2 \ll \frac{N}{L} \sum_{|\ell| < L}\sum_{n \asymp N} \beta_n\overline{\beta_{n+\ell K}}  \chi_a(mn)\overline{ \chi_a(m(n+\ell K))}.
\end{equation*}
By Equation~\eqref{eq.TypeII_estimate} and another application of Cauchy-Schwarz, we conclude that
\begin{align}\label{eq.TypeII_shifted_char_cor_bound}
\begin{split}
    \Bigg|\sum_{\substack{m \asymp M \\ n \asymp N}} \alpha_m \beta_n \chi_a(mn)\Bigg|^2 &\ll \frac{N}{L}\left(\sum_{m \asymp M} \tau(m)^2\right) \sum_{|\ell| < L} \sum_{n \asymp N} \left|\sum_{m \asymp M} \chi_a(mn)\overline{\chi_a(m(n+\ell K))}\right| \\
    & \ll \frac{MN(\log X)^3}{L}\sum_{|\ell| < L} \sum_{n \asymp N} \left|\sum_{m \asymp M} \chi_a(mn)\overline{\chi_a(m(n+\ell K))}\right|
\end{split}
\end{align}
since $\sum_{m < M} \tau(m)^2 \ll M (\log M)^3 \leq M (\log X)^3$.

\subsection{Type II analysis}
Now we assume that $r=1$ and begin to analyze the contribution of the Type II sums. Note that we allow the constants $c$ and $C$ can change from line to line in this section. Choose parameters
\begin{equation*}
    q = p, \quad M = p^\mu, \quad N = p^\nu, \quad L= p^\rho, \quad K = p^{\kappa}
\end{equation*}
with $\mu+\nu = d$ and $\rho < \mu/100$ and $LK < N$. Observe that in base $p$, the integers $mn$ and $m(n+\ell K)$ must have all the same digits for $j < \kappa$. Furthermore, for most pairs $(m,n)$, sufficiently large digits of $mn$ will agree with $m(n+\ell K)$. We make this precise with the following digit stability lemma. 
\begin{lem}\label{lem.digit_stability}
Let $q\ge 2$ be an integer and fix $\delta > 0$. For every $\varepsilon>0$ there exists a constant $c(\varepsilon)$ such that for all integers $\mu,\nu,\rho, \kappa$ with $\mu>0$, $\nu>0$, $\kappa < \rho$, and $0\le \rho\le \nu/2$, and for every integer $r\in\mathbb{Z}$ with $|r|<q^{\rho + \kappa}$ and $q^\kappa \mid r$, we have the following. For $m,n \in \mathbb{Z}$ in the range
\begin{equation*}
    q^{\mu-1}<m\le q^{\mu},\qquad q^{\nu-1}<n\le q^{\nu},
\end{equation*}
write $a = mn$ and $b = m(n+r)$ and let
\begin{equation*}
    a = \sum_{j \geq 0} a_j q^j, \qquad b = \sum_{j \geq 0} b_j q^j,
\end{equation*}
with $0 \leq a_j,b_j < q$. The number $E(r,\mu,\nu,\rho, \kappa, \delta)$ of pairs of integers $(m,n)$ in the above range such that there exists $j \geq \mu + \kappa+(1+\delta)\rho$ with $a_j \neq b_j$ satisfies 
\begin{equation*}
    E(r,\mu,\nu,\rho, \kappa, \delta) \leq C(\varepsilon)q^{(\mu + \nu)(1+\varepsilon) - \delta\rho}
\end{equation*}
\end{lem}

\begin{proof}
    First we treat the case $0 \leq r < q^\rho$ and $\kappa = 0$; the case $-q^{\rho} < r < 0$ is analogous. This implies
    \begin{equation*}
        0 \leq mr < q^{\mu + \rho},
    \end{equation*}
    so if addition of $mr$ changes a digit at $j \geq \mu + \rho$ it must come from a carry. If a digit at $j \geq \mu + (1+\delta)\rho$ changes, it must come from a sequence of carries starting at some digit $j < \mu + \rho$. This implies that all intermediate digits 
    \begin{equation*}
        \mu + \rho \leq j < \mu + (1+\delta)\rho
    \end{equation*}
    are maximal, so $a_j = q-1$ in this range. Now we count $(m,n)$ such that $a = mn$ has this property. Let $ \chi(a)$ be the indicator function of this property, and observe that we have at most $\tau(a)$ pairs $(m,n)$ for each $a$. Since $q^{\mu+\nu-2} < a \leq q^{\mu + \nu}$ we have
    \begin{equation*}
        E(r,\mu,\nu,\rho, \kappa, \delta) \leq \sum_{q^{\mu+\nu-2} < a \leq q^{\mu + \nu}} \tau(a)\chi(a)
    \end{equation*}
    The condition $\chi(a) =1$ fixes $\delta\rho$ digits and we are in an interval with at most $\mu+\nu$ digits, so
    \begin{equation*}
        \#\{q^{\mu+\nu-2} < a \leq q^{\mu + \nu} : \chi(a) = 1\} \ll q^{\mu+\nu-\delta\rho}.
    \end{equation*}
    On the other hand, we have the standard divisor bound
    \begin{equation*}
        \tau(a) \ll a^{\varepsilon} \ll q^{(\mu+\nu)\varepsilon}.
    \end{equation*}
    We conclude that
    \begin{equation*}
        E(r,\mu,\nu,\rho, \kappa, \delta)  \ll q^{(\mu + \nu)(1+\varepsilon)-\delta\rho}.
    \end{equation*}
    The case of negative $r$ is the same argument but with minimal digit $a_j = 0$ instead of maximal digit $a_j = q-1$. For $\kappa > 0$, we can apply the $\kappa = 0$ case to
    \begin{equation*}
        a' = \left\lfloor \frac{a}{q^\kappa}\right\rfloor = \sum_{j \geq 0} a_{j+\kappa} q^j, \quad b' = \left\lfloor \frac{b}{q^\kappa}\right\rfloor = \sum_{j \geq 0} b_{j+\kappa} q^j, \quad r' = \frac{r}{q^\kappa}.
    \end{equation*}
\end{proof}

Using Lemma~\ref{lem.digit_stability}, we can assume that $\chi_a(mn) = \chi_a(m(n+\ell K))$ in each factor for digits
\begin{equation*}
    j< \kappa \quad \text{and} \quad j \geq \mu + \kappa + (1+\delta)\rho,
\end{equation*}
where $\delta > 0$ such that $\delta\rho \in \mathbb{Z}_{\geq 1}$ remains will be specified later. This will introduce an error in Equation~\eqref{eq.TypeII_shifted_char_cor_bound} of size
\begin{equation*}
    O\left(\frac{MN (\log X)^3}{L}\sum_{|\ell|<L}(MN)^{1+\varepsilon} L^{-\delta}\right) = O\left((X)^{2+\varepsilon} L^{-\delta}\right)
\end{equation*}
where $\varepsilon > 0$ is arbitrary. Equivalently, up to the above error, we have
\begin{equation*}
    \chi_a(mn)\overline{\chi_a(m(n+\ell K))} = \chi_{a'}(mn)\overline{\chi_{a'}(m(n+\ell K))},
\end{equation*}
where
\begin{equation*}
    a'_{j} = \begin{cases}
        a_{j} & \text{ if } \kappa \leq j < \kappa + \mu + \rho', \\
        0 & \text{ otherwise,}
    \end{cases}
\end{equation*}
and $\rho' = (1+\delta)\rho$. We will consider $\kappa = 0$ or $\mu - \rho \leq \kappa < d -\mu-\rho$ so that as we vary $\kappa$, the intervals $[\kappa, \kappa+\mu+\rho)$ cover $[0, d-1]$. We will later choose $\kappa$ so that
\begin{equation}\label{eq.kappa_choice}
    |a'| \geq \max_{|J| =\mu}  |a \cap J| \gg \frac{\mu}{d} |a|.
\end{equation}

For $\kappa \neq 0$ we approximate $\chi_{a'}$ by $\Psi_{a'}$ given by Lemma~\ref{lem.Psi_approximation} with $d$ replaced by $\kappa + \mu+\rho'$ and $J$ by $\mu+\rho'$. Take $t = \delta\rho$ so that
\begin{equation*}
    \frac{\mu}{100} > \rho \gg (\log D)^2/\log p.
\end{equation*}
Then Equation~\eqref{eq.TypeII_shifted_char_cor_bound} becomes
\begin{align}\label{eq.TypeII_substitute}
\begin{split}
      \Bigg|\sum_{\substack{m \asymp M \\ n \asymp N}} \alpha_m \beta_n \chi_a(mn)\Bigg|^2 &\ll \frac{X(\log X)^3}{L}\sum_{|\ell|<L}\sum_{n \asymp N} \left|\sum_{m \asymp M} \Psi_{a'}(mn)\overline{\Psi_{a'}(m(n+\ell K))}\right| \\
     &+ X (\log X)^3 \sum_{n \asymp N} \sum_{m \asymp M}\left|\Psi_{a'}(mn)-\chi_{a'}(mn)\right|^2 \\
     &+(X)^{2+\varepsilon} L^{-\delta},
\end{split}
\end{align}
where, by Lemma~\ref{lem.Psi_approximation}, the middle term satisfies
\begin{align*}
    X (\log X)^3\sum_{n \asymp N} \sum_{m \asymp M}\left|\Psi_{a'}(mn)-\chi_{a'}(mn)\right|^2  &< X^{1+\varepsilon}  \left(\sum_{x < X} \left|\Psi_{a'}(x)-\chi_{a'}(x)\right|^2 \right)^{\frac12}\left(\sum_{x \leq X} \tau(x)^2\right)^{\frac12} \\
    &\ll X^{2+\varepsilon} (\log X)^C p^{-c\delta\rho} \ll X^{2+\varepsilon} L^{-c\delta},
\end{align*}
for arbitrary $\varepsilon > 0$. We also observe that 
\begin{equation}\label{eq.psia_prime_expansion}
    \Psi_{a'}(x) = \sum_{k < p^{\mu+\rho'+t}} \widehat{\Psi}_{a'}(k)e\left(\frac{kx}{p^{\mu+\rho'+\kappa}}\right),
\end{equation}
where by Lemma~\ref{lem.ell_infty},
\begin{equation*}
    \norm{\widehat{\Psi}_{a'}}_\infty \leq \norm{\widehat{\chi}_{a'}}_\infty \ll p^{-c|a'|},
\end{equation*}
and by Lemma~\ref{lem.Psi_approximation} with our choice of $\rho$,
\begin{equation}\label{eq.1norm_Psi_a_prime}
    \norm{\widehat{\Psi}_{a'}}_1 \leq \norm{\widehat{\chi}_{a'}}_1 \ll C^{(\log D)^2/\log p}p^{(1/2-c)(\mu+\rho)} \ll p^{(1/2-c)(\mu+\rho)}.
\end{equation}

For $\kappa = 0$ we simply use
\begin{equation}\label{eq.chi_aprime_expansion}
    \chi_{a'}(x) = \sum_{k < p^{\mu+\rho'}}\widehat{\chi}_{a'}(k)e\left(\frac{kx}{p^{\mu+\rho'}}\right),
\end{equation}
where by Lemma~\ref{lem.ell_infty} with $d$ replaced by $\mu+\rho'$,
\begin{equation*}
    \norm{\widehat{\chi}_{a'}}_\infty \ll p^{-c|a'|},
\end{equation*}
and by Lemma~\ref{lem.ell_1_in_AP}
\begin{equation}\label{eq.1norm_chi_a_prime}
    \norm{\widehat{\chi}_{a'}}_1 \ll p^{(1/2-c)(\mu+\rho')} \ll p^{(1/2-c)(\mu+\rho)},
\end{equation}
for $\delta$ sufficiently small and taking $c > 0$ slightly smaller. Denote by $w$ either $\chi_{a'}$ when $\kappa = 0$ or $\Psi_{a'}$ when $\mu+\rho \leq \kappa < d-\mu-\rho$. Substitute Equation~\eqref{eq.chi_aprime_expansion} and Equation~\eqref{eq.psia_prime_expansion} into the first term of Equation~\eqref{eq.TypeII_substitute} to get
\begin{align*}
    \frac{MN (\log X)^3}{L}\sum_{|\ell|<L}\sum_{n\asymp N}\left| \sum_{k,k'}\widehat{w}(k)\widehat{w}(k')\sum_{m \asymp M}e\left(\frac{m}{p^{\mu+\rho'+\kappa}}\left(kn - k'(n+\ell K) \right)\right)\right|.
\end{align*}
Now we use the triangle inequality, smooth the $m$ sum, and apply Poisson summation to obtain
\begin{equation}\label{eq.m_smoothed_expansion}
    \ll \frac{M^2N (\log X)^3}{L}\sum_{|\ell|\ll L}\sum_{n\asymp N}\sum_{k,k'}|\widehat{w}(k)||\widehat{w}(k')|\mathbbm{1}_{\norm{\frac{kn}{p^{\mu+\rho'+\kappa}} - \frac{k'(n+\ell K)}{p^{\mu+\rho'+\kappa}}}< \frac{1}{M_1}},
\end{equation}
where $M_1=M^{1-\varepsilon_1}$ with $\varepsilon_1 > 0$ arbitrary. This is justified by the following lemma:
\begin{lem}\label{lem:msmooth}
Let $Y\ge1$ and let $\Upsilon\in C_c^\infty(\mathbb{R})$ be a fixed non–negative function with
\begin{equation*}
    \Upsilon(u)=1  \quad \text{for } 1\le u\le q \quad\text{ and }
    \quad \Upsilon(u)=0  \quad \text{for } u\notin[1/2,q +1/2],
\end{equation*}
and for $\alpha \in \mathbb{R}$ define
\begin{equation*}
    S_Y(\alpha):= \sum_{m} \Upsilon\left(\frac{m}{Y}\right)e(\alpha m).
\end{equation*}
Then for every $\varepsilon > 0$, we have for all $\alpha\in\mathbb{R}$ and all $A>0$
\begin{equation*}
    \bigl|S_Y(\alpha)\bigr|
     \ll\ 
    Y\mathbbm{1}_{\{\norm{\alpha}<1/Y^{1-\varepsilon}\}}
    +Y^{-A},
    \end{equation*}
    where the implied constant depends $\Upsilon,\varepsilon$ and $A$.
\end{lem}

\begin{proof}
    We apply Poisson summation to the function
    \begin{equation*}
        F(x)=\Upsilon\!\left(\frac{x}{Y}\right)e(\alpha x).
    \end{equation*}
    Noting that
    \begin{equation*}
        \widehat{F}(\xi)
    =\int_{\mathbb{R}} \Upsilon\left(\frac{x}{Y}\right)e(\alpha x)e(-\xi x)dx
    =Y\int_{\mathbb{R}} \Upsilon(u)e\bigl((\alpha-\xi)Yu\bigr)du
    =Y\widehat{\Upsilon}\bigl(Y(\alpha-\xi)\bigr),
    \end{equation*}
    we find
    \begin{equation*}
        S_Y(\alpha)
    =\sum_{r\in\mathbb{Z}}\widehat{F}(r)
    =Y\sum_{r\in\mathbb{Z}}\widehat{\Upsilon}\!\bigl(Y(\alpha-r)\bigr).
    \end{equation*}
    Since $\Upsilon$ is compactly supported and smooth,  
    its Fourier transform is rapidly decaying. Hence for any $B > 0$, 
    \begin{equation}\label{eq:SY-main}
    |S_Y(\alpha)|
    \ll_B
    Y\sum_{r\in\mathbb{Z}}
    (1+Y|\alpha-r|)^{-B}.
    \end{equation}
    
    Let $r_0\in\mathbb{Z}$ satisfy $|\alpha-r_0|=\|\alpha\|$, and put $\delta:=\|\alpha\|$.
    Split \eqref{eq:SY-main} into the $r=r_0$ term and the rest:
    \begin{equation*}
        \sum_{r\in\mathbb{Z}}(1+Y|\alpha-r|)^{-B}
    \le (1+Y\delta)^{-B}
    +\sum_{r\neq r_0}(1+Y|\alpha-r|)^{-B}.
    \end{equation*}
    For $r\neq r_0$ we have 
    \(|r-\alpha|\ge \frac12|r-r_0|\), hence
    \begin{equation*}
        \sum_{r\neq r_0}(1+Y|\alpha-r|)^{-B}
    \ll \sum_{t\neq0}(1+Y|t|)^{-B}
    \ll Y^{-B}.
    \end{equation*}
    Therefore
    \begin{equation}\label{eq:SY-final}
    |S_Y(\alpha)|
    \ll_B
    Y(1+Y\delta)^{-B}
    +
    Y^{1-B}.
    \end{equation}
    
    Now fix $\varepsilon>0$ and observe that  if $\delta < 1/Y^{1-\varepsilon}$, then $|S_Y(\alpha)| \ll Y$ which gives the first term. If $\delta \ge 1/Y^{1-\varepsilon}$, then $Y\delta \ge  Y^{\varepsilon}$. Substituting into \eqref{eq:SY-final}, we obtain,
    \begin{equation*}
        |S_Y(\alpha)| \ll_B Y(Y^\varepsilon)^{-B} = Y^{1-B\varepsilon}.
    \end{equation*}
    Choosing $B>(A+2)/\varepsilon$ gives $|S_Y(\alpha)| \ll Y^{-A}$.
\end{proof}

Now we study the various contributions arising from the condition
\begin{equation}\label{eq.main_condition}
    \norm{\frac{kn}{p^{\mu+\rho'+\kappa}} - \frac{k'(n+\ell K)}{p^{\mu+\rho'+\kappa}}}< \frac{1}{M_1}.
\end{equation}
We will make use of the following easy lemma repeatedly. 

\begin{lem}\label{lem.linear_cong_ineq_counts}
    Let $\delta > 0$ and $\alpha,\beta \in \mathbb{R}$. Suppose that $I \subseteq \mathbb{R}$ be an interval of length $N$. Then 
    \begin{equation*}
        \#\{n \in I : \norm{\alpha n - \beta} < \delta\} \ll 1 + \min\left(N, \frac{\delta}{\norm{\alpha}}\right).
    \end{equation*}
\end{lem}

\begin{proof}
Observe that if $\alpha \in \mathbb{Z}$ then $\norm{\alpha n -\beta} = \norm{\beta}$ so we either have zero or $N$ solutions. Thus, we can assume that $0<\norm{\alpha}\leq \frac12$. Divide $I$ into blocks of length $\norm{\alpha}^{-1}$. In each block, the values $\alpha n \bmod 1$ are $\norm{\alpha}$-seperated, so an interval of length $2\delta$ can have at most $\ll 1 + \delta \norm{\alpha}^{-1}$ such points. There are $\ll 1 + N\norm{\alpha}$ blocks, so the total number of solutions is
\begin{equation*}
    \ll \left(1 + \delta \norm{\alpha}^{-1}\right)\left(1 + N\norm{\alpha}\right) \ll 1+N + \frac{\delta}{\norm{\alpha}}.
\end{equation*}
Taking the minimum with the trivial bound $N$ completes the proof. 
\end{proof}

First, note that the $\ell = 0$ contribution to Equation~\eqref{eq.TypeII_shifted_char_cor_bound} is at most $X^2L^{-1}(\log X)^3$. Now we analyze the diagonal contribution $k=k'$ with $\ell \neq 0$:
\begin{equation*}
    \frac{X^2}{L} \sum_{|\ell| < L} \sum_{k<p^{\mu+\rho'+t}} |\widehat{w}(k)|^2\mathbbm{1}_{\norm{\frac{k \ell}{p^{\mu+\rho'}}} < \frac{1}{M_1}}.
\end{equation*}
The condition $\norm{\frac{k \ell}{p^{\mu+\rho'}}} < \frac{1}{M_1}$ says that there exists $v \in \mathbb{Z}$ such that $|k\ell-vp^{\mu+\rho'}| < p^{\mu+\rho'}/M_1$. Since $k < p^{\mu+\rho'+t}$ we will get 
\begin{equation*}
    \ll 1+\frac{p^{\mu+\rho'+t}|\ell|}{p^{\mu+\rho'}} \ll p^t|\ell|
\end{equation*}
possible values of $v$ for each $\ell$. By Lemma~\ref{lem.linear_cong_ineq_counts}, for each $v$ and $\ell$ we have 
\begin{equation*}
    \ll 1 + \frac{p^{\mu + \rho'}}{M_1 |\ell|} \ll \frac{M^{\varepsilon_1} p^{\rho'}}{|\ell|}
\end{equation*}
possibilities for $k$. Therefore, the total number of $k$ that survive is $\ll p^{\rho' + t} M^{\varepsilon_1}$, and we conclude that the diagonal contribution is at most
\begin{equation*}
    M^{2+\varepsilon_1}(\log X)^3N^2 p^{\rho'+t} \norm{\widehat{w}}_\infty^2 < X^2(\log X)^3L^2\norm{\widehat{w}}_\infty^2 <X^2(\log X)^3L^2p^{-c|a'|},
\end{equation*}
by choosing $\varepsilon_1$ small enough relative to $\delta\rho$.


Now we study the off-diagonal contribution $k \neq k'$ and $\ell \neq 0$, first considering the case $\kappa = 0$. Observe that we can partition the $k,k'$ sum as
\begin{equation*}
    \sum_{k \neq k'} = \sum_{r \geq 0} \sum_{\substack{k \neq k' \\ \nu_p(k-k') =r}} = \sum_{r \geq 0} \sum_{v < p^r} \sum_{\substack{k \equiv v \bmod p^r \\ k' \equiv v \bmod p^r}},
\end{equation*}
where there are $\ll \log X$ relevant values of $r$. Now we need to estimate the contribution for
\begin{equation*}
    (k-k', p^{\mu+\rho'}) = p^r,
\end{equation*}
so $k-k' = k_1p^r$ with $(k_1,p) = 1$, and condition~\eqref{eq.main_condition} becomes
\begin{equation}\label{eq.contribution_in_AP}
    \norm{\frac{k_1n}{p^{\mu+\rho'-r}} - \frac{k'\ell}{p^{\mu+\rho'}}} < \frac{1}{M_1},
\end{equation}
implying that
\begin{equation}\label{eq.k_counter}
    \norm{\frac{k'\ell}{p^{r}}} < \frac{p^{\mu+\rho'}}{M_1p^r} \leq \frac{M^{\varepsilon_1}L^{1+\delta}}{p^r} < \frac{L^{1+2\delta}}{p^r},
\end{equation}
since we choose $\varepsilon_1$ small compared to $\delta \rho$. It follows that there are at most $L^{1+2\delta}$ possibilities for $k' \bmod p^r$, and hence for $(k,k') \bmod p^r$. For fixed $k,k',\ell$, condition~\eqref{eq.contribution_in_AP} determines $n \bmod p^{\mu+\rho'-r}$ up to $1+L^{1+2\delta}p^{-r}$ possibilities, hence $n$ up to $\frac{Np^r}{ML}(1+L^{1+2\delta}p^{-r})$ possibilities. Thus, the corresponding contribution to Equation~\eqref{eq.m_smoothed_expansion} is at most
\begin{align*}
    &\frac{M^2N(\log X)^3}{L} \sum_{|\ell| \ll L} L^{1+2\delta} \frac{Np^r}{ML}(1+L^{1+2\delta}p^{-r}) \max_v \sum_{\substack{k \equiv v \bmod p^r \\ k' \equiv v \bmod p^r}}|\widehat{w}(k)||\widehat{w}(k')|  \\
    &\ll M(\log X)^3N^2(L+p^r)L^{2\delta} \max_v \Bigg(\sum_{\substack{k <p^{\mu+\rho'} \\ k \equiv v \bmod p^r}}|\widehat{w}(k)|\Bigg)^2 \\
    &\ll M(\log X)^3N^2(L+p^r)L^{2\delta}p^{(1-c)(\mu+\rho'-r)} \\
    & = X^2(\log X)^3(L^2p^{-r} + L)(MLp^{-r})^{-c}L^{3\delta}.
\end{align*}
If we assume that $MLp^{-r}  > L^{C}$ we can obtain the bound $X^2(\log X)^3L^{-1}$, so suppose that $MLp^{-r}<L^C$. From Equation~\eqref{eq.k_counter} and $k < p^{\mu+\rho'}$, there are at most $L^{1+4\delta}(MLp^{-r})^2 < L^C$ possibilities for $(k,k')$. This gives a contribution 
\begin{equation*}
    M^2(\log X)^3N^2L^C \norm{w}_\infty^2 < L^C X^2(\log X)^3p^{-c|a'|}.
\end{equation*}
Summing over the $\ll \log X$ values of $r$ for which $\nu_p(k-k') = r$ is nonempty, we conclude that for $\kappa = 0$ we get the bound
\begin{equation*}
    X^2(\log X)^3(L^{-1} + L^C p^{-c|a'|}),
\end{equation*}
for some $C > 0$.

Now consider the off-diagonal contribution for $\kappa  \geq \mu-\rho$. This time we partition the $k,k'$ sum as
\begin{equation*}
    \sum_{k \neq k'} = \sum_{\Delta k} \sum_{\substack{k \neq k' \\ |k-k'| \asymp \Delta k}}
\end{equation*}
where the sum over $\Delta k$ has $\ll \log X$ terms. Fix $k,k',\ell$ such that $|k-k'| \asymp \Delta k < ML^2$. Letting $n$ range over an interval of length $\frac{MLK}{\Delta k}$, the number of possible $n$ in that interval is
\begin{equation*}
    \ll 1 + \frac{L^{1+2\delta}K}{\Delta k}
\end{equation*}
using Lemma~\eqref{lem.linear_cong_ineq_counts} and the bounds in Equation~\eqref{eq.k_counter}. If we assume that $N \gtrsim
 \frac{MLK}{\Delta k}$, then since $\kappa \geq \mu-\rho$ implies
\begin{equation*}
    LK \geq M > \frac{\Delta k}{L^2},
\end{equation*}
we find that, assuming $\delta < 1$, there are at most 
\begin{equation*}
    \frac{N\Delta k}{MLK}\left(1 + \frac{L^{1+2\delta}K}{\Delta k}\right) < \frac{N}{M}L^2
\end{equation*}
values of $n$ satisfying Condition~\eqref{eq.main_condition}. This gives a contribution to Equation~\eqref{eq.TypeII_substitute} of size
\begin{equation}
    L^2 MN^2(\log X)^3 \norm{w}_1^2 < L^2MN^2(\log X)^3(ML)^{1-c} < X^2(\log X)^3L^3M^{-c}.
\end{equation}
Now assume that $N \ll \frac{MLK}{\Delta k}$. Then for fixed $k,k',\ell$ there are at most 
\begin{equation*}
    1 + \frac{KL^3}{\Delta k} \sim \frac{KL^3}{\Delta k} 
\end{equation*}
values of $n$ satisfying Condition~\eqref{eq.main_condition}. Moreover, we have
\begin{equation*}
    \norm{\frac{k'\ell}{p^{\mu+\rho'}}} < \frac{1}{M_1} + \frac{N\Delta k}{Mp^{\rho'+\kappa}}.
\end{equation*}
Since $|k'\ell| < p^{\mu+\rho}L^2$, there is some integer $\ell_1$ with $|\ell_1| <L^2$ and
\begin{equation*}
    \left|\frac{k'\ell}{p^{\mu+\rho'}}-\ell_1\right|< \frac{1}{M_1} + \frac{N\Delta k}{Mp^{\rho'+\kappa}}.
\end{equation*}
Therefore
\begin{equation*}
    \left|k'-\ell_1 \frac{p^{\mu+\rho'}}{\ell}\right| < L^{1+2\delta} + \frac{N\Delta k}{K},
\end{equation*}
restricting $k'$ to at most $L^2$ intervals of size $L^{1+2\delta} + \frac{N\Delta k}{K}$. Using Lemma~\ref{lem.interval} in the fixed $p$ regime, we obtain the following contribution to Equation~\eqref{eq.m_smoothed_expansion}:
\begin{align*}
    &M^2(\log X)^3NL^2\left(L^{1+2\delta} + \frac{N\Delta k}{K}\right)^{1-c} \frac{KL^3}{\Delta k}\\
    &\ll \frac{M^2(\log X)^3NL^7K}{\Delta k} + M^2(\log X)^3N^2L^5\left(\frac{N\Delta k}{K}\right)^{-c} \\
    &< X^2(\log X)^3L^7\left(\frac{K}{N\Delta k}\right)^c
\end{align*}
If we assume $\frac{N\Delta k}{K} > L^C$ we obtain the bound $X^2 (\log X)^3L^{-1}$, so suppose that $\frac{N\Delta k}{K}<L^C$. Exactly as in the $\kappa = 0$ case, this restricts $k$ to $L^C$ values and the correspoding contribution to Equation~\eqref{eq.m_smoothed_expansion} is
\begin{equation*}
    M^2(\log X)^3N^2L^C \norm{\widehat{w}}_\infty^2  < X^2(\log X)^3L^Cp^{-c|a'|},
\end{equation*}
uniformly in $\Delta k$. Summing the $\ll \log X$ possibilities we conclude a contribution of size
\begin{equation*}
     X^2(\log X)^4(L^{-1} + L^3M^{-c} + L^C p^{-c|a'|}).
\end{equation*}

Collecting the previous bounds for both $\kappa = 0$ and $\kappa \geq \mu-\rho$, we conclude that Equation~\eqref{eq.m_smoothed_expansion} is bounded by
\begin{equation*}
     O\left(X^2(\log X)^4\left(L^{-1} + L^3M^{-c} + L^Cp^{-c|a'|}\right)\right).
\end{equation*}
Recalling Equation~\eqref{eq.TypeII_substitute}, we conclude that 
\begin{equation}\label{eq.typeII_fin_bound_prime}
     \Bigg|\sum_{\substack{m \asymp M \\ n \asymp N}} \alpha_m\beta_n \chi_a(nm) \Bigg| \ll X(\log X)^4(L^{-c\delta} + L^2M^{-c} + L^Cp^{-c|a'|}).
\end{equation}
Recalling Equation~\eqref{eq.kappa_choice}, we obtain
\begin{equation}\label{eq.typeII_fin_bound}
    \Bigg|\sum_{\substack{m \asymp M \\ n \asymp N}} \alpha_m\beta_n \chi_a(nm) \Bigg| \ll X(\log X)^4(L^{-c\delta} + L^2M^{-c} + L^Cp^{-c\frac{\mu}{d}|a|}),
\end{equation}
where $L = p^\rho$ is a parameter satisfying $\frac{\mu}{100} > \rho \gg (\log D)^2/\log p$. For $|a| \leq d^{1/2}H^{-1}$ with $H \gg 1$ a parameter, we obtain from Proposition~\ref{prop.smallA}
\begin{equation*}
    \left|\sum_{x \leq X} \mu(x)\chi_a(x)\right| \ll X (\log X) e^{-cH}.
\end{equation*}
Thus we may assume $|a| > d^{1/2}H^{-1}$. Taking $L = p^H$, it follows from Equations~\eqref{eq.typeII_fin_bound_prime} and~\eqref{eq.typeII_fin_bound},
\begin{equation}\label{eq.TypeII_final_bound}
    \Bigg|\sum_{\substack{m \asymp M \\ n \asymp N}} \alpha_m\beta_n \chi_a(nm) \Bigg|  \ll  X (\log X)^3p^{-c\delta H},
\end{equation}
assuming 
\begin{equation}\label{eq.remaining_cases}
    M > p^{CH^2d^{1/2}} \quad \text{ or } \quad M > C^H \text{ and } |a'|> CH \text{ with }|a'| \gg \frac{\mu}{d}|a|.
\end{equation}
Since we only need estimates for Type II sums in the range 
\begin{equation*}
    M > X^{\gamma_1} > p^{CH^2d^{1/2}},
\end{equation*}
this completes the Type II analysis if we take $H = d^{1/10}$. 

\subsection{Type I analysis}
For Type I sums, we have to estimate
\begin{equation*}
    \sum_{m \asymp M} \tau(m)\left|\sum_{n \asymp N} \chi_a(mn)\right|
\end{equation*}
where $MN \asymp X$. Viewing this as a special case of Equation~\eqref{eq.TypeII_estimate} with $\beta_n = 1$, our Type II analysis handles the case $M > p^{CH^2d^{1/2}}$. Therefore, it remains to bound this under the assumption $M < p^{CH^2d^{1/2}}$ and $|a| > d^{1/2}H^{-1}$. 

Fourier expanding $\chi_a$ with a suitable smoothing in the $n$-sum (e.g. Lemma~\ref{lem:msmooth}), we obtain
\begin{align}
\begin{split}\label{eq.TypeI_analysis}
      \sum_{m \asymp M}\tau(m) \left|\sum_{n \asymp N} \chi_a(mn)\right| &\leq \sum_{m \asymp M}\tau(m) \sum_{k < X} |\widehat{\chi}_a(k)|\left|\sum_{n \asymp N} e\left(\frac{kmn}{X}\right)\right| \\
     &\ll N \sum_{m \asymp M}\tau(m)\sum_{k < X} |\widehat{\chi}_a(k)|\mathbbm{1}_{\norm{\frac{km}{X}}< \frac{d^2}{N}} + o(1) \\
     &\ll N Md^2 \norm{\widehat{\chi}_a}_\infty\sum_{m \asymp M} \tau(m)  \\
     &\ll N M^2d^2  \log M\norm{\widehat{\chi}_a}_\infty \\
     &\ll XM (\log X)^3 p^{-cd^{1/2}H^{-1}}.
\end{split}
\end{align}
by Lemma~\eqref{lem.ell_infty}. If we take $H = d^{1/10}$, then this completes the proof in the case $M < C^H$. Now we assume that $M > C^H$, so $\mu > H$. Recalling Equation~\eqref{eq.remaining_cases}, we may assume that $\max |A \cap J | < CH$ with the max taken over intervals of length $|J| = \mu$ and $A = \{a_j : j \neq 0\}$. We will use these assumptions to further exploit the second bound above.

Write $A = A_1 \cup A_2$ where $A_1 = A \cap [0,d-2\mu]$ and $A_2 = A \cap [d-2\mu,d]$. By our assumption, $|A_2| < CH$. Therefore, by truncating the Fourier expansion of $\psi_j$, we obtain
\begin{equation*}
    \chi_{a_2}(x) = \prod_{j \in A_2} \psi_j\left(\frac{x}{p^{j+1}}\right) = \sum_{k_2 \in \cal{A}_2} \widehat{\chi}_{a_2}(k_2)e\left(\frac{k_2x}{p^d}\right) + O_{L_1}(Hp^{-H}),
\end{equation*}
where we can take $|\cal{A}_2| < p^{H|A_2|} < C^{H^2}$. On the other hand, since we have
\begin{equation*}
    \chi_{a_1}(x) = \sum_{k_1 < p^{d-2\mu}} \widehat{\chi}_{a_1}(k_1)e\left(\frac{k_1x}{p^{d-2\mu}}\right),
\end{equation*}
we obtain
\begin{equation*}
    \chi_a(x) =  \sum_{k_1 < p^{d-2\mu}} \sum_{k_2 \in \cal{A}_2} \widehat{\chi}_{a_1}(k_1)\widehat{\chi}_{a_2}(k_2)e\left(\frac{p^{2\mu}k_1+k_2}{p^d}x\right) + O_{L^1}(Hp^{-H}).
\end{equation*}
This allows us to refine Equation~\eqref{eq.TypeI_analysis} to a bound of the form 
\begin{align}
\begin{split}\label{eq.typeI_analysis_refined}
    &\leq N \sum_{m \asymp M} \tau(m)\sum_{k_1<p^{d-2\mu}} \sum_{k_2\in \cal{A}_2}|\widehat{\chi}_{a_1}(k_1)||\widehat{\chi}_{a_2}(k_2)| \mathbbm{1}_{\norm{\frac{p^{2\mu}k_1+k_2}{p^d}m}<\frac{d^2}{N}} \\
     &< N |\cal{A}_2| \norm{\widehat{w}_{a_1}}_\infty \max_{k_2} \sum_{m \asymp M}\tau(m)\left|\left\{k_1<p^{d-2\mu} : \norm{\frac{p^{2\mu}k_1+k_2}{p^d}m}<\frac{d^2}{N}\right\}\right| \\
     &= N |\cal{A}_2| \norm{\widehat{w}_{a_1}}_\infty \sum_{m \asymp M}\tau(m)\left|\left\{k_1<p^{d-2\mu} : k_1m \equiv 0 \bmod  p^{d-2\mu}\right\}\right|.
\end{split}
\end{align}
Write $m = p^zm_0$ with $(p,m_0) = 1$ so that
\begin{equation*}
    \left|\left\{k_1<p^{d-2\mu} : k_1m \equiv 0 \bmod  p^{d-2\mu}\right\}\right| = \left|\left\{k_1<p^{d-2\mu} : k_1 \equiv 0 \bmod  p^{d-2\mu-z}\right\}\right| = p^z,
\end{equation*}
and see that the sum over $m$ is
\begin{equation*}
    \sum_{m\leq  M} \tau(m) p^{\nu_p(m)} = \sum_{0 \leq z \leq \mu} p^z \sum_{\substack{m_0 \leq M/p^z\\ (m_0,p) = 1}} \tau(m_0p^z)\ll  \sum_{0 \leq z \leq \mu} p^z (z+1) \frac{M}{p^z} \log M\ll \mu^2 M \log M.
\end{equation*}
Since $|A_1| \gg \frac{d^{1/2}}{H}$ and $|\cal{A}_2| < C^{H^2}$, Lemma~\ref{lem.ell_infty} implies that 
\begin{equation}\label{eq.TypeI_final_bound}
    \sum_{m \asymp M} \tau(m) \left|\sum_{n \asymp N} \chi_a(mn)\right| \ll C^{H^2} p^{-cd^{1/2}H^{-1}}MN(\log X)^3  \ll X (\log X)^3p^{-cd^{1/2}H^{-1}}
\end{equation}
by taking $H=d^{1/10}$. Combining our Type I estimates (Equation~\eqref{eq.TypeI_final_bound}) and Type II estimates (Equation~\eqref{eq.TypeII_final_bound}), we conclude by Lemma~\ref{lem.TypeIandII_reduction} that for some $c  > 0$,
\begin{equation}
    \left|\sum_{x < X}\mu(x)\chi_a(x)\right| \ll Xp^{-cd^{1/10}}.
\end{equation}
The above argument works for $H = d^{1/10+\varepsilon}$ for $\varepsilon$ small enough, so this becomes a strict equality with $c = 1$ and for large enough $d$. This completes the proof of Theorem~\ref{thm:main_math}. By the lower bounds in Section~\ref{sec.lower_bounds} and Corollary~\ref{cor:alignment_X_d}, Theorems~\ref{thm.main_cyclic}, \ref{thm.main_pgroup}, \ref{thm:main_kernel}, \ref{thm:main_NGD}, and \ref{thm:main_CSQ} all follow.
\subsection{Digital prime number theorem}
Fix a prime $p$ and let $X = p^d$. We now use the preceding analysis to understand sums of the form
\begin{equation*}
    \sum_{n < X} \Lambda(n)\chi_a(n),
\end{equation*}
for any $a \in \mathbb{F}_p^d$. 

\begin{prop}\label{prop:balanced_Lambda_digital}
Fix a prime $p$ and let $X = p^d$. For all $d\gg 1$ and all $a \in \widehat{\mathcal{X}}_d$, we have
\begin{equation*}
    \sum_{n<X}\bigl(\Lambda(n)-\nu_p(n)\bigr)\chi_a(n)
    \ll Xp^{-d^{1/10}},
\end{equation*}
where 
\begin{equation*}
    \nu_p(n) = \frac{p}{p-1} \mathbbm{1}_{p \nmid n}.
\end{equation*}
\end{prop}

\begin{proof}
    First suppose that $|a| \gg d^{1/2}H^{-1}$ where $H = d^{1/10}$. Observe that $p \nmid n$ is equivalent to $n_0 \neq 0$ so we can write
    \begin{align*}
        \sum_{n < X} \nu_p(n)\chi_a(n) &= \frac{p}{p-1}\sum_{\substack{n_0 \in \mathbb{F}_p^\times \\ n_1,\ldots,n_{d-1} \in \mathbb{F}_p}} e_p\left(\sum_{j=0}^{d-1}a_jn_j\right) \\
        &= \frac{p}{p-1}\left(\sum_{n_0 \in \mathbb{F}_p^\times} e_p(a_0n_0)\right)\prod_{j=1}^{d-1}\left(\sum_{n_j \in \mathbb{F}_p}e_p(a_jn_j)\right)
    \end{align*}
    Since $|a| \gg d^{1/2}H^{-1} > 1$, atleast one $a_j \neq 0$ with $j > 0$ and so some digit factor must vanish by orthogonality. Then by Vaughan’s identity and the Type I/II analysis above, see for example \cite[Lemma 1]{MauduitRivat2010}, we find that 
    \begin{equation*}
        \left|\sum_{n < X} \Lambda(n) \chi_a(n)\right| \ll X p^{-d^{1/10}}.
    \end{equation*}

    It remains to treat $|a| < d^{1/2}H^{-1}$, using again the argument of Kátai and this time invoking standard major and minor arc estimates for $\Lambda$. First we note that $|a| = 0$ follows from the usual prime number theorem, so we assume $|a| > 0$. Suppose for the sake of contradiction that
    \begin{equation*}
        \left| \frac{1}{X \log X}\sum_{n \leq X} (\Lambda(n)-\nu_p(n))\chi_a(n)\right| >  \Delta.
    \end{equation*}
    Then by Proposition~\ref{prop.Katai_arg}, we find a $C>0$ and a $\theta \in \mathbb{R}/\mathbb{Z}$ such that
    \begin{equation}\label{eq:small_weight_lowerb}
         \left|\frac{1}{X \log X}\sum_{n \leq X} (\Lambda(n)-\nu_p(n))e(n \theta)\right| \geq C \left(\frac{\Delta}{10 C |a| \sqrt{p} }\right)^{4|a|}.
    \end{equation}
    Moreover, by the remark following the proof of Proposition~\ref{prop.Katai_arg}, we have 
    \begin{equation}\label{eq:sparse_rational}
        \theta = \sum_{j \in J} \frac{s_j}{p^j}, \qquad |J| \leq |a|, \qquad |s_j| \leq \left(\frac{C|a|\sqrt{p}}{\Delta}\right)^C.
    \end{equation}
    Now choose
    \begin{equation*}
        \Delta = |a|\sqrt{p} \exp\left(-c_1\frac{(\log X)^{1/2}}{|a|}\right)
    \end{equation*}
    for some $c_1 > 0$. We claim that if $c_1$ is sufficiently small, then
    \begin{equation}\label{eq:small_weight_upperb}
        \left|\sum_{n \leq X} (\Lambda(n)-\nu_p(n))e(n \theta)\right| \ll X \exp\left(-c (\log X)^{1/2}\right).
    \end{equation}
    Then Equation~\eqref{eq:small_weight_upperb} contradicts Equation~\eqref{eq:small_weight_lowerb}, so after absorbing the polynomial factors in $d$ we we must have
    \begin{equation*}
        \left|\sum_{n \leq X} (\Lambda(n)-\nu_p(n))\chi_a(n)\right| \ll X\log X |a| \sqrt{p}  \exp\left(-c\frac{d^{1/2}}{|a|}\right) \ll X\exp(-cH)
    \end{equation*}
    since $|a| \leq d^{1/2}H^{-1}$ with $H = d^{1/10}$.
    
    Now it remains to prove Equation~\eqref{eq:small_weight_upperb}. Take major and minor arc decomposition as in Equation~\eqref{eq:major_minor_arcs_def}. For $\theta \in \mathfrak{m}$, we get
    \begin{equation*}
        \sum_{n < X} \Lambda(n)e(n\theta) \ll X \exp\left(-c (\log X)^{1/2}\right)
    \end{equation*}
    from Vaughan's identity and the classical minor arc estimates, for example \cite[Theorem 13.6]{IwaniecKowalski2004}. Likewise
    \begin{equation*}
        \sum_{n < X} \nu_p(n)e(n\theta) = \frac{p}{p-1} \sum_{n < X}e(n\theta) - \frac{p}{p-1}\sum_{m < X/p}e(mp\theta),
    \end{equation*}
    and on minor arcs we have $\norm{\theta} \gg Q/X$ and $\norm{p\theta} \gg Q/X$ so the standard geometric sum bounds gives
    \begin{equation*}
        \sum_{n < X} \nu_p(n)e(n\theta) \ll \frac{X}{Q} \ll X\exp\left(-c (\log X)^{1/2}\right).
    \end{equation*}
    We conclude that Equation~\eqref{eq:small_weight_upperb} for $\theta \in \mathfrak{m}$. 
    
    Finally, we consider the case where $\theta \in \mathfrak{M}$. Taking $Q$ as above and recalling Equation~\eqref{eq:sparse_rational}, we see that the conditions of Lemma~\ref{lem:sparse_padic_major_arc} are satisfied. We conclude that $\theta \in \mathfrak{M}(b,q)$ with $q= p^s$ for some $s \geq 0$. Write $\theta = u/p^s + \beta$ with $(u,p) = 1$ if $s \geq 1$ and $\theta = \beta$ if $s = 0$. By Proposition~\ref{prop:p_power_major_arc_mu_Lambda}, we find
\begin{equation*}
    \sum_{n \leq X} \Lambda(n)e(n\theta) = \frac{\mu(p^s)}{\varphi(p^s)} \int_0^X e(\beta t)dt + O\left(X \exp\left(-c(\log X)^{1/2}\right)\right),
\end{equation*}
uniformly for $p^s \leq Q$ and $|\beta| \leq \frac{Q}{p^s X}$. Now we show that this agrees with the contribution of $\nu_p$. For $s \geq 1$ we write
\begin{align*}
    \sum_{n \leq X} \nu_p(n)e(n\theta) &= \frac{p}{p-1}\sum_{\substack{r \bmod p^s \\ p \nmid r}} e\left(\frac{ur}{p^s}\right)e(\beta r) \sum_{0 \leq m < X/p^s} e(\beta p^sm) \\
    &= \frac{p}{p-1} \frac{1}{p^s} \sum_{\substack{ r \bmod p^s \\ p \nmid r}} e\left(\frac{ur}{p^s}\right) \int_0^X e(\beta t)dt + O(Q^2) \\
    &= \frac{\mu(p^s)}{\varphi(p^s)} \int_0^X e(\beta t)dt + O(Q^2).
\end{align*}
since the sum over $r$ is a Ramanujan sum \cite[Chapter 26]{davenport2000multiplicative}. For $s=0$ we have
\begin{align*}
    \sum_{n\leq X}\nu_p(n)e(\beta n)
    &= \frac{p}{p-1}\sum_{n\leq X}e(\beta n) -
    \frac{p}{p-1}\sum_{m\leq X/p}e(\beta pm) \\
    &= \int_0^X e(\beta t)\,dt + O_p(Q),
\end{align*}
since $|\beta|X\leq Q$. We conclude that for any $s \geq 0$, we get
\begin{equation*}
    \sum_{n < X} (\Lambda(n)-\nu_p(n))e(n\theta) \ll X \exp\left(-c(\log X)^{1/2}\right).
\end{equation*}
which proves the proposition.
\end{proof}

\begin{proof}[Proof of Theorem \ref{thm:digital_PNT}]
    By orthogonality on $\mathbb{F}_p^m$ we write
    \begin{equation*}
        \mathbbm{1}_{Lx = b} = p^{-m} \sum_{\xi \in (\mathbb{F}_p^m)^\vee} e_p(-\xi(b))e_p(\xi(Lx)).
    \end{equation*}
    Notice that $e_p(\xi(Lx)) = \chi_{L^\vee \xi}(x)$, so we have
    \begin{equation*}
        \sum_{\substack{n < X \\ L(n_0,\ldots, n_{d-1})=b}} \Lambda(n)  = p^{-m} \sum_{\xi \in (\mathbb{F}_p^m)^\vee} e_p(-\xi(b)) \sum_{n < X} \Lambda(n)\chi_{L^\vee \xi}(n).
    \end{equation*}
    Write $\Lambda = \nu_p + (\Lambda - \nu_p)$ and apply Proposition~\ref{prop:balanced_Lambda_digital} to obtain
    \begin{align*}
        \sum_{\substack{n < X \\ L(n_0,\ldots, n_{d-1})=b}} \Lambda(n)& = \sum_{\substack{n < X \\ L(n_0,\ldots, n_{d-1})=b}} \nu_p(n) + O\left(p^{d-d^{1/10}}\right) \\
        &= \frac{p}{p-1} \#\{x \in \mathbb{F}_p^d : Lx = b, x_0 \neq 0\}
    \end{align*}
    It remains to compute the size of this set, which can be done with elementary linear algebra. 

    Since $L$ is surjective, $\# L^{-1}(b) = p^{\dim \ker L} = p^{d-m}$ by rank-nullity. First, suppose that $e_0 \notin \mathrm{im } L^\vee$. Then $e_0 \notin (\ker L)^\perp$, so $e_0\vert_{\ker L} \neq 0$. Fix a point $x_b \in L^{-1}(b)$ so that every point in the fiber can be written uniquely as $x = x_b + v$ for $v \in \ker L$. Since $e_0\vert_{\ker L}$ is a nonzero linear functional, each value of $\mathbb{F}_p$ is taken equal often on $\ker L$. Thus 
    \begin{equation*}
        \#\{x \in \mathbb{F}_p^d : Lx = b, x_0 \neq 0\} = (p-1)p^{d-m-1}.
    \end{equation*}
    Now suppose that $e_0 \in \mathrm{im } L^\vee$. Since $L$ is surjective, $L^\vee$ is injective, and so there exists a unique $\lambda \in (\mathbb{F}_p^m)^\vee$ such that $e_0 = L^\vee \lambda$. For every $x \in L^{-1}(b)$ we have 
    \begin{equation*}
        x_0 = e_0(x) = (L^\vee \lambda)(x) = \lambda(Lx) = \lambda(b)
    \end{equation*}
    so $x_0$ is constant on the fiber $L^{-1}(b)$. If $\lambda(b) = 0$ then no point has $x_0 \neq 0$, and otherwise every point $x_0 \neq 0$, which concludes the proof.
\end{proof}

\begin{proof}[Proof of Corollary \ref{cor:digital_PNT}]
    Taking out prime powers, we immediately have
    \begin{equation*}
        \sum_{\substack{\ell < X \\\ell \text{ prime} \\ L(\ell_0,\ldots, \ell_{d-1})=b}} \log \ell = \mathfrak{S}_p(L,b) \frac{X}{p^m} + O\left(Xp^{-d^{1/10}}+ X^{1/2} \log X\right).
    \end{equation*}
    Let $\eta = K \frac{\log d}{d}$ for notice that for $K \gg _B 0$
    \begin{equation*}
        X^{-\eta} p^m \log X \ll p^{m-\eta d} \log X  \ll  d^{(B- K )\log p +1} \ll \frac{\log d}{d}.
    \end{equation*}
    In particular, we get
    \begin{equation*}
        X^{1-\eta} \ll \frac{\log d}{d} \frac{X}{p^m \log X}.
    \end{equation*}
    For primes $X^{1-\eta} \leq \ell < X$ we have $(1-\eta) \log X \leq \log \ell < \log X$, so 
    \begin{equation*}
        \frac{1}{\log X}  \sum_{\substack{\ell < X \\\ell \text{ prime} \\ L(p_0,\ldots, p_{d-1})=b}} \log p \leq \pi_{L,b}(X) \leq \frac{1}{(1-\eta) \log X} \sum_{\substack{\ell < X \\\ell \text{ prime} \\ L(p_0,\ldots, p_{d-1})=b}} \log p + X^{1-\eta},
    \end{equation*}
    Since $\eta \ll_B \frac{\log d}{d}$ we conclude that
    \begin{equation*}
        \pi_{L,b}(X) = \frac{\mathfrak{S}_p(L,b)}{p^m} \frac{X}{\log X}\left(1 + O\left(\frac{\log d}{d}\right)\right).
    \end{equation*}
\end{proof}

\subsection{Binary multiplicative functions}
Consider the centered covariance matrix
\begin{equation*}
    C = \mathbb{E}_{h \sim \mu_\mathcal{H}}[(h-\mathbbm{1}_\square)(h - \mathbbm{1}_{\square})^T],
\end{equation*}
where $h, \mathbbm{1}_{\square} \in \ell^2(\Omega_d)$ and $\Omega_d = [1, X_d] \cap \mathbb{Z}$. This is a matrix with entries
\begin{equation*}
    C_{mn} = \mathbbm{1}_{\square}(mn) - \mathbbm{1}_\square(n)\mathbbm{1}_{\square}(m)
\end{equation*}
using the fact that $\mathbb{E}_{h \sim \mathcal{M}_d} [h(n)] = \mathbbm{1}_{\square}(n)$. It suffices to understand the spectrum of $C$ since
\begin{equation*}
    \mathbb{E}_{h \sim \mu_\mathcal{H}} [\langle h-\mathbbm{1}_{\square}, \phi \rangle^2] = \langle C \phi, \phi \rangle \leq \norm{C}_{\mathrm{op}} \norm{\phi}_2^2.
\end{equation*}
Now define the sets $B_a = \{am^2: m \leq \lfloor \sqrt{X_d/a}\rfloor\}$ and, noting that $B_a \cap B_b = \emptyset$ for $a \neq b$, consider the orthonormal vectors
\begin{equation*}
    u_a = \sqrt\frac{X_d}{|B_a|} \mathbbm{1}_{B_a} \in \ell^2(\Omega_d).
\end{equation*}
Then for $a= 1$ we have 
\begin{align*}
    (Cu_a)_n &= \frac{1}{X_d} \sqrt\frac{X_d}{|B_a|}\sum_{m \leq X_d} C_{mn} \mathbbm{1}_{\square}(m) \\
    &= \frac{1}{X_d}  \sqrt\frac{X_d}{|B_a|}\sum_{m \leq X_d} \left(\mathbbm{1}_{\square}(n) \mathbbm{1}_{\square}(m) - \mathbbm{1}_{\square}(n) \mathbbm{1}_{\square}(m)\right) = 0,
\end{align*}
and for $a > 1$ squarefree we have 
\begin{align*}
    (Cu_a)_n &= \frac{1}{X_d} \sqrt\frac{X_d}{|B_a|}\sum_{m \leq X_d} C_{mn} \mathbbm{1}_{B_a}(m) \\
    &= \frac{1}{X_d}  \sqrt\frac{X_d}{|B_a|}\sum_{m \leq X_d}\mathbbm{1}_{B_a}(n)\mathbbm{1}_{B_a}(m) = \frac{|B_a|}{X_d}(u_a)_n,
\end{align*}
so we conclude for $a > 1$
\begin{equation*}
    C u_a = \lambda_a u_a  \quad \text{ with} \quad  \lambda_a = \frac{|B_a|}{X_d} = \frac{1}{X_d} \left\lfloor\sqrt{\frac{X_d}{a}} \right\rfloor.
\end{equation*}
Now we observe that this gives the entire spectrum of $C$ since for any $f \in \ell^2(\Omega_d)$ we have
\begin{equation*}
     Cf = \sum_{a > 1} \frac{S_a(f)}{X_d} \mathbbm{1}_{B_a} \quad \text{ where } \quad S_a(f) = \sum_{m  \in B_a}f(m). 
\end{equation*}
Thus $\operatorname{Im} C \subseteq \mathrm{span } \{u_a : a > 1\}$. Since $C$ is self-adjoint and positive semi-definite, 
\begin{equation*}
    \norm{C}_{\mathrm{op}} = \max_{a > 1} \lambda_a = \frac{1}{X_d}\left\lfloor \sqrt{\frac{X_d}{2}} \right\rfloor \asymp X_d^{-1/2}.
\end{equation*}
This implies Theorem~\ref{thm:main_binary_mul} by an application of \cite[Theorem A.2]{abbe2022nonuniversalitydeeplearningquantifying}.

\bibliographystyle{alpha}  
\bibliography{Bibliography}  

\end{document}